\documentclass[11pt]{article}
\usepackage{fullpage}
\usepackage{natbib}
\setcitestyle{authoryear,round,citesep={;},aysep={,},yysep={;}}
\usepackage[T1]{fontenc}
\usepackage{amsmath,amssymb,amsthm,mathtools}
\usepackage{booktabs,graphicx,microtype}
\usepackage{longtable}
\usepackage{hyperref}
\usepackage{url}
\usepackage{xcolor}

\newcommand{\reals}{\mathbf R}
\newcommand{\integers}{\mathbf Z}
\newcommand{\Fclass}[1]{\mathcal F_{1,#1}}
\newcommand{\rhofix}{\rho_{\mathrm{fix}}}
\newcommand{\rhohor}{\rho_{\mathrm{hor}}}
\newcommand{\qopt}{q_\star}
\newcommand{\ProbMeasures}{\mathcal P}

\newtheorem{theorem}{Theorem}[section]
\newtheorem{lemma}[theorem]{Lemma}
\newtheorem{proposition}[theorem]{Proposition}
\newtheorem{corollary}[theorem]{Corollary}
\theoremstyle{definition}

\theoremstyle{remark}
\newtheorem{remark}[theorem]{Remark}

\title{The exact worst-case asymptotic rate of Nesterov acceleration
for strongly convex optimization}
\author{\begin{tabular}[t]{@{}c@{\qquad}c@{}}
    Sebastian Pokutta\textsuperscript{1,2,}\thanks{Equal contribution.} &
    Bartolomeo Stellato\textsuperscript{3,}\footnotemark[1]\\
    \normalfont\small \href{mailto:pokutta@zib.de}{\texttt{pokutta@zib.de}} &
    \normalfont\small \href{mailto:bstellato@princeton.edu}{\texttt{bstellato@princeton.edu}}
  \end{tabular}\\\noalign{\vskip 3pt}
  \small \textsuperscript{1}Department for AI in Society, Science, and Technology, Zuse Institute Berlin, Germany\\
  \small \textsuperscript{2}Institute of Mathematics, Technische Universit\"at Berlin, Germany\\
  \small \textsuperscript{3}Department of Operations Research and Financial Engineering, Princeton University, USA
}
\hypersetup{pdfauthor={Sebastian Pokutta and Bartolomeo Stellato}}

\date{October 10, 2026}
\hypersetup{colorlinks=true,citecolor=cyan,linkcolor=cyan,urlcolor=cyan}

\begin{document}
\maketitle

\begin{abstract}
We study Nesterov's accelerated gradient method with its standard constant parameters for smooth strongly convex optimization.
We characterize the exact worst-case asymptotic factor for the squared distance to the minimizer.
To this end, we construct candidate worst-case trajectories by combining planar spirals with a common contraction factor that satisfy a self-similarity condition.
Each component rotates by a fixed angle at every iteration.
The exact factor is the optimal value of a linear optimization problem over probability measures on a fixed disk, with constraints derived from smooth strongly convex interpolation.
Every optimal measure has a common contraction factor, which is also the largest factor for which the angular interpolation conditions are feasible.
We also show that this worst-case asymptotic rate can be realized by a \emph{single} function in a finite-dimensional space whose iterates attain this factor. A matching upper bound holds uniformly over all functions in the class and all finite dimensions.
Consequently, allowing the function, dimension, and initial point to vary with the iteration horizon does not increase the worst-case asymptotic factor.
We also show that the exact factor is $1-4/(3\sqrt\kappa)+o(\kappa^{-1/2})$ as $\kappa\to\infty$, where $\kappa=L/\mu$ is the ratio of the gradient Lipschitz constant to the strong convexity parameter.
\end{abstract}

\section{Introduction}

We study Nesterov's accelerated gradient method (NAG) with constant parameters for smooth strongly convex optimization.
Our goal is to determine its exact worst-case asymptotic convergence rate, measured by the squared distance to the minimizer.

Let $\kappa=L/\mu$ denote the condition number.
The classical convergence guarantee gives an asymptotic squared-distance factor of at most $1-1/\sqrt\kappa$ (\citealp[Theorem~2.2.1, specialized to scheme~(2.2.22)]{nesterov2018lectures}; \citealp[Proposition~12]{lessard2016analysis}).
The exact worst-case factor among quadratic objectives is $(1-1/\sqrt\kappa)^2$~\citep[Proposition~1 and Appendix~A]{lessard2016analysis}.
These bounds leave open how slowly the method can converge on a nonquadratic objective.

A function that nearly attains the worst-case error after $t$ iterations need not remain worst-case as the method continues.
We therefore distinguish between keeping the function and initial point fixed and allowing the function, dimension, and initial point to depend on the iteration count.
We prove that both choices give the same worst-case asymptotic factor, attained by one fixed function and initial point.

\subsection{Contribution}
Our main contributions are:

\paragraph{An optimization characterization of the exact rate.}
We characterize the exact worst-case asymptotic factor by a linear optimization problem over probability measures on a disk whose radius depends only on the condition number.
Every optimal measure is supported on one circle, so its components have a common contraction factor.
The value is also the largest factor for which the angular interpolation conditions are feasible; these conditions ensure consistency of the prescribed points, gradients, and function values with one function in the class.

\paragraph{Matching lower and upper bounds.}
We construct one function and initial point attaining the optimal asymptotic factor.
A matching upper bound holds uniformly over all functions in the class, all finite dimensions, and all nonoptimal initial points.
Thus, allowing the function, dimension, and initial point to vary with the iteration horizon does not increase the worst-case asymptotic factor.

\paragraph{Exact and asymptotic descriptions of the rate.}
We exhibit a condition number for which the exact asymptotic factor cannot be expressed by radicals.
For large condition numbers, we obtain the explicit expansion $1-4/(3\sqrt{\kappa})+o(\kappa^{-1/2})$ as $\kappa\to\infty$, proving that the leading constant $4/3$ in Gannot's upper bound is sharp.

\paragraph{Explicit worst-case examples and interactive visualizations.}
At $\kappa=10$, we construct a worst-case trajectory consisting of one planar spiral.
At $\kappa=100$, our construction combines two planar spirals with different rotation angles.
We provide interactive visualizations of these exact examples at \url{https://nagsc.iol.berlin}.
For other user-selected values of $\kappa$, the website computes candidate worst-case trajectories numerically on demand.

\subsection{Related work}

The classical convergence guarantee for NAG can be improved without changing the algorithm.
\citet[Section~4.5]{lessard2016analysis} use integral quadratic constraints (IQCs) to compute tighter upper bounds through semidefinite programming, and \citet{safavi2018explicit} derive explicit bounds by solving the associated inequalities analytically.
\citet[Proposition~2.13]{gannot2022frequency} obtains an analytic upper bound using frequency-domain inequalities.
For squared distance, his bound already has the leading term $1-4/(3\sqrt{\kappa})$ as $\kappa\to\infty$; we prove that the constant $4/3$ is sharp.
We extend this IQC and frequency-domain approach by constructing a finite separator for every contraction factor above the interpolation optimum.
The resulting upper bound matches the rate attained by one fixed objective.

Lyapunov functions provide a related approach to convergence analysis.
\citet{sanzserna2021connections} construct such functions for Nesterov methods and connect them to Lyapunov functions for differential equations.
More generally, searches over prescribed families of quadratic Lyapunov inequalities can be formulated as semidefinite programs~\citep{taylor2018lyapunov,upadhyaya2024automated}, an approach implemented in AutoLyap~\citep{upadhyaya2025autolyap}.
The exactness guarantees for these searches concern the chosen family of Lyapunov inequalities.

Establishing sharpness requires a matching lower bound.
Performance estimation formulates worst-case analysis after a fixed number of iterations as an optimization problem~\citep{drori2014performance}.
The interpolation conditions of \citet{taylor2017smooth} make this formulation exact for smooth strongly convex functions and provide a way to construct matching examples.
For the smooth convex version of Nesterov's method, \citet[Theorem~5.2]{kim2018generalizing} give an $O(t^{-3})$ upper bound on the smallest squared gradient norm evaluated through iteration $t$, under an initial distance bound.
\citet[Theorem~7.2]{du2026relaxedpep} constructs matching examples for $t\geq7$ in dimension $t-4$.
We impose these interpolation conditions for every pair of iteration indices along an infinite trajectory.
One strongly convex function realizes this trajectory and attains the worst-case asymptotic squared-distance factor from a fixed initial point.

Interpolation can also construct trajectories that demonstrate failure to converge.
\citet{goujaud2023counterexamples} use it to find periodic trajectories and identify parameter choices for which first-order methods, including Nesterov acceleration, fail to converge.
\citet{regis2026variational} studies related behavior by characterizing fixed points, two-point cycles, and their stability for heavy-ball and Nesterov momentum.
For the Nesterov parameters considered here, convergence is already guaranteed.
The trajectories we construct contract while rotating, and their decay determines the worst-case asymptotic factor.

The behavior of these examples depends on the function class.
For quadratic objectives, \citet{hagedorn2024iteration} derive iteration-complexity bounds for Nesterov's and Polyak's methods by analyzing their linear recurrences.
For twice continuously differentiable objectives, once convergence is established, linearization near the minimizer connects asymptotic convergence to a quadratic model~\citep[Lemma~3]{vanscoy2025c2}.
\citet{fengyuan2026local} establish local linear convergence of Nesterov's method from sufficiently close initial points under a local Polyak--\L{}ojasiewicz condition and twice continuous differentiability.
They obtain an optimal local rate by tuning the parameters.
Our function class assumes a Lipschitz continuous gradient but does not require twice continuous differentiability.
The matching nonquadratic examples show that restricting this class to quadratic objectives can miss the worst-case asymptotic factor.

The choice of algorithm parameters matters as well.
For the smooth convex version of Nesterov's method, which uses a varying momentum coefficient, \citet{jangryu2025point} prove that the iterates converge to a minimizer even when the minimizer is not unique.
\citet{hefang2026unified} study parameter sequences that combine accelerated behavior in the early iterations with a strongly convex asymptotic guarantee.
Improved guarantees can also come from modifying the update itself.
\citet{park2023factor} analyze the optimized gradient method using Lyapunov functions and linear coupling, and introduce a strongly convex variant.
This variant changes the momentum coefficient and adds a gradient correction, yielding a better function-value guarantee than the classical Nesterov bound.
For strongly convex objectives, \citet[Section~5]{lan2026gradient} develop regularization and restart methods with improved iteration-complexity guarantees for finding a small gradient.
We keep the prescribed Nesterov step size and momentum coefficient fixed and determine the exact worst-case asymptotic rate of that iteration.
Corollary~\ref{cor:metrics} gives the corresponding exact factor for the gradient norm at the current iterate.

\paragraph{Paper organization.}
Section~\ref{sec:problem} defines the function class, algorithm, and two asymptotic rates.
Section~\ref{sec:variational} gives the angular interpolation conditions, defines the rate through a disk LP, and proves the common-circle reduction.
Section~\ref{sec:rate-theorem} states the main theorem and outlines the matching lower and upper bounds.
Section~\ref{sec:example} presents a worst-case example at $\kappa=10$.
The appendices provide the full proofs, additional examples, numerical results, and an algebraic obstruction to formulas by radicals.

\section{Problem setting and rate definitions}
\label{sec:problem}

We work in $\reals^d$ with the Euclidean inner product and norm.
For $0<\mu\leq L$, let $\mathcal F_{\mu,L}$ denote the class of differentiable, $\mu$-strongly convex functions with $L$-Lipschitz gradients.
Dividing the objective by $\mu$ gives a function in $\Fclass{\kappa}$, where $\kappa=L/\mu$ is the condition number.
This scaling preserves the NAG iterates when the step size changes from $1/L$ to $1/\kappa$.
Every function in this class has a unique minimizer, denoted by $x_\star$.

For $\kappa\geq1$, we consider the NAG iteration
\begin{equation}
\label{eq:nag}
 y_t=x_t+\beta_\kappa(x_t-x_{t-1}),
 \qquad
 x_{t+1}=y_t-\frac{1}{\kappa}\nabla f(y_t),
 \qquad
 \beta_\kappa=\frac{\sqrt\kappa-1}{\sqrt\kappa+1},
\end{equation}
initialized with $x_{-1}=x_0\neq x_\star$.
Here $y_t$ is the point where the gradient is evaluated, and $x_t$ is the iterate whose distance to the minimizer we measure.
The step size $1/\kappa$ and momentum coefficient $\beta_\kappa$ remain constant throughout the iteration.

For a fixed function and initial point, define the normalized squared error
\[
 R_t(f,x_0)=\frac{\|x_t-x_\star\|^2}{\|x_0-x_\star\|^2}.
\]
The quantity $\limsup_{t\to\infty}R_t(f,x_0)^{1/t}$ measures its asymptotic geometric decay.
For example, if $R_t=q^t$, this quantity equals $q$.
Throughout, convergence factors refer to squared distance.
We compare the
\emph{fixed-objective asymptotic factor} and the
\emph{horizon-dependent asymptotic factor}:
\begin{align}
\rhofix(\kappa)
&=\sup_{\substack{d\geq1,\ f\in\Fclass{\kappa}(\reals^d)\\x_0\neq x_\star}}
  \limsup_{t\to\infty}R_t(f,x_0)^{1/t},
\label{eq:rho-fix}\\
\rhohor(\kappa)
&=\limsup_{t\to\infty}
  \left(\sup_{\substack{d\geq1,\ f\in\Fclass{\kappa}(\reals^d)\\x_0\neq x_\star}}
  R_t(f,x_0)
  \right)^{1/t}.
\label{eq:rho-hor}
\end{align}
In $\rhofix(\kappa)$, we first compute the asymptotic factor for each fixed function and initial point, then take the supremum over all such choices and dimensions.
In $\rhohor(\kappa)$, we first maximize the error at each iteration $t$, allowing the function, dimension, and initial point to depend on $t$, then take the asymptotic factor.
Hence $\rhofix(\kappa)\leq\rhohor(\kappa)$.
Our main theorem proves equality and, for each $\kappa$, attainment by a \emph{single fixed function} and initial point, both independent of $t$.

When $\kappa=1$, strong convexity and smoothness imply $\nabla f(x)=x-x_\star$.
The method reaches the minimizer in one iteration and remains there, so
$\rhofix(1)=\rhohor(1)=0$.
We therefore focus on $\kappa>1$.

The quadratic $f(x)=\tfrac12x^2$ belongs to $\Fclass{\kappa}$.
On this objective, the NAG iteration becomes
\[
 x_{t+1}=\left(1-\frac1\kappa\right)
 \bigl((1+\beta_\kappa)x_t-\beta_\kappa x_{t-1}\bigr).
\]
Its characteristic polynomial is $\bigl(\lambda-(1-1/\sqrt\kappa)\bigr)^2$, giving the asymptotic squared-distance factor
\begin{equation}
\label{eq:q0}
 q_{\mathrm{quad}}(s)=\left(1-\frac{1}{s}\right)^2,
 \qquad s=\sqrt\kappa.
\end{equation}
Together with Nesterov's classical bound, this gives
\[
 q_{\mathrm{quad}}(s)\leq\rhofix(\kappa)\leq\rhohor(\kappa)\leq1-\frac1s,
 \qquad s=\sqrt\kappa.
\]
We determine both asymptotic factors exactly over the full class $\Fclass{\kappa}$.

\section{Interpolation conditions for contracting trajectories}
\label{sec:variational}

We construct candidate worst-case trajectories whose components contract and rotate by fixed amounts at each iteration.
We first state the conditions under which prescribed query points and gradients come from one function in $\Fclass{\kappa}$.
We then express these conditions in terms of the contraction factor, rotation angles, and component weights.
We use these conditions to define the feasible factors, then give a disk LP whose value is their largest element.
Section~\ref{sec:rate-theorem} identifies this value with the exact worst-case rate.

\subsection{Exact interpolation and cyclic monotonicity}
\label{sec:interpolation}

Fix $\kappa>1$.
A finite collection $(y_i,g_i,f_i)$ is \emph{interpolable} by $\Fclass{\kappa}$ if there exists one $f\in\Fclass{\kappa}$ satisfying $\nabla f(y_i)=g_i$ and $f(y_i)=f_i$ for every index $i$.
The exact interpolation theorem~\citep[Theorem~4]{taylor2017smooth} states that this holds if and only if every ordered pair satisfies
\begin{equation}
\label{eq:ssc-interpolation}
 f_i\geq f_j+\langle g_j,y_i-y_j\rangle
 +\frac1{2\kappa}\|g_i-g_j\|^2
 +\frac{\kappa}{2(\kappa-1)}
  \left\|y_i-y_j-\frac1\kappa(g_i-g_j)\right\|^2.
\end{equation}
An interpolation inequality is \emph{active} when it holds with equality.
To construct a trajectory, we prescribe the query points and gradients and leave the function values free.
The following change of variables eliminates the function values from the interpolation conditions.
Write $s=\sqrt\kappa$ and translate the minimizer to the origin.
Set
\begin{equation}
\label{eq:interpolation-variables}
 h(y)=f(y)-\frac12\|y\|^2,
 \qquad
 \sigma=\nabla h(y),
 \qquad
 r=y-\frac{\sigma}{\kappa-1}.
\end{equation}
Subtracting the quadratic term makes $h$ convex and $(\kappa-1)$-smooth.
Its convex conjugate $h^*$ is $1/(\kappa-1)$-strongly convex, so
\begin{equation}
\label{eq:phi}
 \varphi(\sigma)=h^*(\sigma)-\frac{\|\sigma\|^2}{2(\kappa-1)}
\end{equation}
is a proper closed convex function, possibly taking the value $+\infty$.
The identity $\sigma=\nabla h(y)$ gives $y\in\partial h^*(\sigma)$ and hence $r\in\partial \varphi(\sigma)$.
Here $\partial$ denotes the convex subdifferential.
Conversely, a set of pairs $(\sigma_i,r_i)$ lies in the subdifferential of one proper closed convex function if and only if it is \emph{cyclically monotone}: for every finite cycle
$i_0,i_1,\ldots,i_m=i_0$,
\[
 \sum_{j=0}^{m-1}
 \langle r_{i_j},\sigma_{i_j}-\sigma_{i_{j+1}}\rangle\geq0.
\]
Rockafellar's cyclic monotonicity theorem guarantees the existence of such a function $\varphi$~\citep{rockafellar1970convex}.
Reversing the change of variables gives an interpolating function in $\Fclass{\kappa}$:
\[
 f(y)=\left(\varphi+\frac{\|\cdot\|^2}{2(\kappa-1)}\right)^*(y)
      +\frac12\|y\|^2.
\]
Thus cyclic monotonicity characterizes whether the prescribed points and gradients admit a function in $\Fclass{\kappa}$.
The resulting function determines the previously unspecified values $f_i$.

\subsection{Contracting and rotating trajectories}
\label{sec:rotating-trajectories}

Consider a sequence in $\reals^2$ obtained by applying the same contraction and rotation at every iteration.
Identify $\reals^2$ with the complex plane and set
$z=\sqrt q\,e^{i\theta}$, where $0<q<1$.
Multiplication by $z$ contracts lengths by $\sqrt q$ and rotates directions by
$\theta$, so squared norms contract by $q$.
Thus larger values of $q$ correspond to slower convergence.

Applying \eqref{eq:interpolation-variables} at $y_t$ gives $\sigma_t=\nabla f(y_t)-y_t$ and $r_t=y_t-\sigma_t/(\kappa-1)$.
The gradient update gives
\[
 x_{t+1}=y_t-\frac1\kappa(y_t+\sigma_t)
 =\frac{\kappa-1}{\kappa}
  \left(y_t-\frac{\sigma_t}{\kappa-1}\right)
 =\frac{\kappa-1}{\kappa}r_t.
\]
Thus $r_t$ is a constant multiple of $x_{t+1}$.
We consider candidate trajectories for which these iterates contract and rotate by fixed amounts:
\[
 r_{t+1}=zr_t,
 \qquad
 r_t=a z^t,
\]
where $a\neq0$ specifies the initial magnitude and direction.
We first take $a=1$, then combine components with different angles and amplitudes.

These sequences describe the trajectory after initialization.
Lemma~\ref{lem:initialization-splice} constructs an objective attaining the optimal factor with $x_{-1}=x_0$.

With $r_t=z^t$, we seek $\sigma_t=d z^t$.
Using $y_t=r_t+\sigma_t/(\kappa-1)$ and
$x_t=(\kappa-1)r_{t-1}/\kappa$ in the momentum update gives
\[
 \left(1+\frac{d}{\kappa-1}\right)z^2
 =\frac{\kappa-1}{\kappa}
  \bigl((1+\beta_\kappa)z-\beta_\kappa\bigr).
\]
Solving for $d$ with $\kappa=s^2$ gives
\begin{equation}
\label{eq:nag-transfer}
 r_t=z^t,
 \qquad \sigma_t=d_s(z)z^t,
 \qquad
 d_s(z)=-\frac{(s^2-1)(sz-(s-1))^2}{s^2z^2}.
\end{equation}
For $q>0$, this expression is defined at every angle.
At $q=q_{\mathrm{quad}}(s)$ and $\theta=0$, we have $d_s(z)=0$, so the transformed gradient vanishes.
This includes a geometric trajectory of the quadratic objective $f(x)=\|x\|^2/2$.
The interpolation conditions below determine which candidate trajectories come from one function in $\Fclass{\kappa}$.

\subsection{Interpolation conditions}

We apply the cyclic monotonicity conditions to the contracting and rotating trajectories from Section~\ref{sec:rotating-trajectories}.
Fix $s>1$ and $q\in(0,1)$.
Write $z=\sqrt q\,e^{i\theta}$, where $0\leq\theta\leq\pi$.
For one such trajectory, define the terms appearing in these inequalities by
\begin{equation}
\label{eq:edge-cost}
 c_{ij}(\theta)
 =\langle r_i,\sigma_i-\sigma_j\rangle
 =\operatorname{Re}\left[
 \overline{d_s(z)}z^i\left(\bar z^i-\bar z^j\right)
 \right],
 \qquad i,j\geq0.
\end{equation}
We suppress the dependence on the fixed parameters $s$ and $q$ in $c_{ij}(\theta)$.
Here $\operatorname{Re}$ denotes the real part, and an overbar on a complex number denotes complex conjugation.
The inner products are unchanged when $\theta$ is replaced by $-\theta$, so it suffices to consider $\theta\in[0,\pi]$.

We combine $m$ contracting and rotating trajectories with the same factor $q$.
Scaling all components by the same positive constant preserves the interpolation inequalities and convergence factor, so we normalize the weights to sum to one.
Choose angles $\theta_1,\ldots,\theta_m$ and positive weights $w_1,\ldots,w_m$ with $\sum_{k=1}^m w_k=1$.
Set $z_k=\sqrt q\,e^{i\theta_k}$ and
\[
 r_t=(\sqrt{w_k}\,z_k^t)_{k=1}^m,
 \qquad
 \sigma_t=(\sqrt{w_k}\,d_s(z_k)z_k^t)_{k=1}^m.
\]
The inner products add across coordinates:
\[
 \langle r_i,\sigma_i-\sigma_j\rangle
 =\sum_{k=1}^m w_k\,c_{ij}(\theta_k).
\]
Each complex coordinate represents two real coordinates.
An angle of $0$ or $\pi$ requires only one.

This finite collection of modes can be encoded by the finitely supported probability measure $\nu=\sum_{k=1}^m w_k\delta_{\theta_k}$, where $\delta_\theta$ is the unit point mass at $\theta$.
Its weighted inner products are then
\[
 \bar c_{ij}(\nu)=\int_0^\pi c_{ij}(\theta)\,d\nu(\theta).
\]
For the variational formulation, we now enlarge the domain from finitely supported measures to all probability measures on $[0,\pi]$, denoted by $\ProbMeasures([0,\pi])$.
For a general $\nu$ in this set, the same integral defines $\bar c_{ij}(\nu)$; the interpolation inequalities below still determine feasibility.
This enlargement permits compactness arguments without assuming a bound on the number of modes; an infinite-support measure does not itself specify a finite-dimensional modal construction.
At the largest feasible factor, if it exceeds $q_{\mathrm{quad}}(s)$, Lemma~\ref{lem:finite-optimal-support} guarantees that a feasible angular measure can nevertheless be chosen with finite support.

We seek function values $\varphi_i$ satisfying
\[
 \varphi_i-\varphi_j\leq\bar c_{ij}(\nu),
 \qquad i,j\geq0.
\]
Shifting both indices gives $\bar c_{i+t,j+t}(\nu)=q^t\bar c_{ij}(\nu)$.
Lemma~\ref{lem:stars} shows that, whenever these inequalities are feasible, we can choose $\varphi_i=q^i\varphi_0$ for some scalar $\varphi_0$.
The inequalities for $(0,t)$ and $(t,0)$ give the following equivalent conditions.

\begin{lemma}[Equivalent interpolation conditions]
\label{lem:stars}
For a probability measure $\nu$ on $[0,\pi]$, the cyclic monotonicity inequalities hold if and only if there is a real scalar $\varphi_0$ satisfying
\begin{equation}
\label{eq:star-feasible}
 -\bar c_{t0}(\nu)\leq(1-q^t)\varphi_0\leq\bar c_{0t}(\nu),
 \qquad\text{for every integer }t\geq1.
\end{equation}
In that case, we may choose $\varphi_i=q^i\varphi_0$.
\end{lemma}

The proof is given in Appendix~\ref{app:star-reduction}.
The averaged terms $\bar c_{ij}(\nu)$ depend on $s$ and $q$ through $d_s(z)$ and $z=\sqrt q\,e^{i\theta}$.
For fixed $q$, the conditions in \eqref{eq:star-feasible} are linear in $\nu$ and $\varphi_0$ and must hold for every integer $t\geq1$.
For $s>1$, define the set of feasible factors by
\begin{equation}
\label{eq:feasible-factors}
 \mathcal Q(s)=\left\{q\in(0,1)\,\middle|\,\exists\,\nu\in\ProbMeasures([0,\pi]),\ \varphi_0\in\reals\text{ satisfying }\eqref{eq:star-feasible}\right\}.
\end{equation}
The choice $q=q_{\mathrm{quad}}(s)$, $\nu=\delta_0$, and $\varphi_0=0$ is feasible.
Lemma~\ref{lem:closed-rate-problem} shows that $\mathcal Q(s)$ has a largest element and every feasible factor is at most $1-1/s$.

\paragraph{A linear rate formulation and common-circle reduction.}
\label{sec:fixed-domain-rate}

In \eqref{eq:star-feasible}, all components share a contraction factor $q$, and the angular measure $\nu$ specifies their rotation angles.
We now allow both the contraction factors and angles to vary, represented by a probability measure $\mu$ on the disk $D_s=\{z\in\mathbb C:|z|^2\leq1-1/s\}$.
The objective is the average squared contraction factor, $\int |z|^2\,d\mu(z)$.

Let $\Gamma$ contain all finite cycles of iteration indices $\gamma=(n_0,\ldots,n_m=n_0)$, with $m\geq1$ and nonnegative integer indices.
Indices may repeat, and no index is required to be zero.
For each cycle, define the scaled sum of the interpolation terms in \eqref{eq:edge-cost} by
\begin{equation}
\label{eq:disk-cycle-cost}
 \mathcal C_\gamma^s(z)=\frac{q^2}{s^2-1}
 \sum_{j=0}^{m-1}c_{n_j,n_{j+1}}(\theta),
 \qquad z=\sqrt q\,e^{i\theta},\quad q>0.
\end{equation}
Here $c_{ij}(\theta)$ is evaluated at $q=|z|^2$.
Define $\mathcal C_\gamma^s$ on the lower half of the disk by $\mathcal C_\gamma^s(\bar z)=\mathcal C_\gamma^s(z)$, and set $\mathcal C_\gamma^s(0)=0$.
On each circle, the positive scaling preserves the cyclic monotonicity inequalities.
The coefficients are continuous on the whole disk; Appendix~\ref{app:fixed-domain-rate} gives their polynomial expressions.

For $s>1$, define $\qopt(s)$ by the linear optimization problem
\begin{equation}
\label{eq:disk-rate-lp}
 \qopt(s)=
 \begin{array}[t]{ll}
  \text{maximize} &
       \displaystyle\int |z|^2\,d\mu(z)\\[2pt]
  \text{subject to} &
       \displaystyle\int\mathcal C_\gamma^s(z)\,d\mu(z)\geq0,
       \quad\gamma\in\Gamma.
 \end{array}
\end{equation}
where the optimization variable $\mu$ is a probability measure on $D_s$.

\begin{proposition}[Common-circle reduction]
\label{prop:fixed-domain-rate}
For $s>1$, problem~\eqref{eq:disk-rate-lp} attains its maximum, and every optimal measure is supported on the circle $|z|^2=\qopt(s)$.
For each $q\in(0,1)$, membership in $\mathcal Q(s)$ is equivalent to the
existence of a feasible disk measure supported on $|z|^2=q$.
Consequently,
\begin{equation}
\label{eq:qcyc}
 \qopt(s)=\max\mathcal Q(s),
 \qquad q_{\mathrm{quad}}(s)\leq\qopt(s)\leq1-1/s.
\end{equation}
\end{proposition}

An \emph{optimal angular pair} is a pair $(\nu,\varphi_0)$ satisfying
\eqref{eq:star-feasible} at $q=\qopt(s)$.
Thus every optimal disk measure uses a common contraction factor, although
feasible disk measures may mix radii. For such mixed-radius measures, the
objective is an average, not an asserted asymptotic trajectory factor.
The objective and constraints are linear in $\mu$.
For a feasible measure with positive objective, multiplying its weights by $|z|^2$ and normalizing preserves feasibility.
This strictly increases the objective unless $|z|^2$ is constant on the support.
This argument uses all cycles, including those obtained by shifting every index by the same amount.
Appendix~\ref{app:fixed-domain-rate} gives the proof and relaxations that provide upper bounds.
The disk LP remains infinite-dimensional; for trajectory constructions and
numerical evaluation we use the angular conditions \eqref{eq:star-feasible}
and the identity \eqref{eq:qcyc}.
The next section identifies $\qopt(s)$ with the worst-case asymptotic factor over the full function class.

\section{Exact rate theorem}
\label{sec:rate-theorem}

\begin{theorem}[Exact asymptotic rate]
\label{thm:main}
For every $\kappa>1$, the two worst-case asymptotic factors defined in Section~\ref{sec:problem} satisfy
\begin{equation}
\label{eq:main}
 \rhofix(\kappa)=\rhohor(\kappa)=\qopt(\sqrt\kappa).
\end{equation}
This factor is attained by a single function in a finite-dimensional space and an initial point, both independent of the iteration count.
The exact rate satisfies
\begin{equation}
\label{eq:classical-bracket}
 \left(1-\frac{1}{\sqrt\kappa}\right)^2
 \leq \qopt(\sqrt\kappa)
 \leq 1-\frac{1}{\sqrt\kappa}.
\end{equation}
In particular, the rate tends to zero as $\kappa\downarrow1$ and to one as $\kappa\to\infty$.
\end{theorem}

\begin{proof}[Proof sketch]
Set $s=\sqrt\kappa>1$.
For the lower bound, suppose $\qopt(s)>q_{\mathrm{quad}}(s)$.
An optimal angular pair can then be chosen with a finitely supported measure, so the construction in Section~\ref{sec:variational} uses finitely many planar trajectories.
We construct a function in $\Fclass{\kappa}$ for which Nesterov's method, initialized with $x_{-1}=x_0$, attains this asymptotic factor.
If $\qopt(s)=q_{\mathrm{quad}}(s)$, the quadratic $f(x)=x^2/2$ attains the factor.
For the upper bound, fix $q\in(\qopt(s),1)$.
The interpolation conditions are infeasible at this factor.
A separation argument gives a nonnegative combination of finitely many cyclic monotonicity inequalities.
After rescaling the iterates, we sum these inequalities and use the NAG recurrence to bound their squared norms.
The key identity averages a finite sum of integer powers over equally spaced points on the unit circle.
A finite induction first establishes that the sums of squared norms are finite, and then gives
\[
 \|x_t-x_\star\|^2\leq bq^t\|x_0-x_\star\|^2,
\]
where $b$ depends only on $s$ and $q$.
Taking $t$th roots and letting $t\to\infty$ gives $\rhohor(s^2)\leq q$.
Letting $q\downarrow\qopt(s)$ proves the upper bound.
The full proof is given in Appendix~\ref{app:generic-bridges}.
\end{proof}

\paragraph{Relation to Gannot's upper bound.}
Proposition~2.13 of \citet{gannot2022frequency} gives an upper bound $q_{\rm G}(\kappa)$ on the asymptotic squared-distance factor.
At $\kappa=10$ and $\kappa=100$, we determine the exact factor using exact arithmetic.
Numerical evaluation of Gannot's upper bound gives a larger value in both cases.
The differences reported in Table~\ref{tab:nag-rate-values} are numerical approximations.

\begin{remark}[No general formula by radicals]
\label{rem:no-radical-formula}
At $\kappa=(1421/1250)^2$, the exact factor $\qopt(1421/1250)$ has a minimal
polynomial of degree $20$ over $\mathbb Q$, with Galois group $S_{20}$.
Consequently, no formula built from rational functions of $s$ with rational
coefficients and finitely many radicals can be defined and valid for every
$s>1$. Appendix~\ref{app:rate-radicals} gives the polynomial and its exact
certificate. This does not exclude other notions of closed form.
\end{remark}

\begin{corollary}[Asymptotic rate for large condition numbers]
\label{cor:large-kappa}
The exact factor satisfies
\begin{equation}
\label{eq:large-kappa}
 \lim_{\kappa\to\infty}
 \sqrt\kappa\left(1-\qopt(\sqrt\kappa)\right)=\frac43,
 \qquad
 \qopt(\sqrt\kappa)
 =1-\frac{4}{3\sqrt\kappa}+o(\kappa^{-1/2}).
\end{equation}
The corresponding asymptotic factor for the distance $\|x_t-x_\star\|$ is
$1-2/(3\sqrt\kappa)+o(\kappa^{-1/2})$.
\end{corollary}

\begin{proof}
The necessary condition $\bar c_{01}(\nu)\geq0$ gives an upper bound on the exact factor.
For the lower bound, fix $2/3<\alpha<3/4$.
For all sufficiently large $\kappa$, we construct a probability measure satisfying every interpolation inequality with $\varphi_0=0$ at
\[
 q=\exp(-2\alpha/\sqrt\kappa).
\]
Letting $\alpha\downarrow2/3$ gives the matching coefficient $4/3$.
The construction and full proof are given in Appendix~\ref{app:large-kappa}.
\end{proof}

Figure~\ref{fig:nag-rate-curve} compares numerical evaluations of the rate formula with Gannot's upper bound for $1\leq\kappa\leq1000$.
We compute Gannot's upper bound from the smaller positive root of the cubic in \cite[Proposition~2.13]{gannot2022frequency}.
The two curves nearly coincide at this scale; panel (c) shows their difference.
The cubic defining Gannot's upper bound also arises from a subset of our interpolation inequalities involving consecutive iteration indices and the minimizer.
Keeping only these inequalities enlarges the set of feasible factors $\mathcal Q(s)$ and therefore gives an upper bound on the exact factor by \eqref{eq:qcyc}.
Selected numerical values and details of their verification are given in Appendix~\ref{app:rate-table}.

\begin{figure}[t]
\centering
\includegraphics[width=\linewidth]{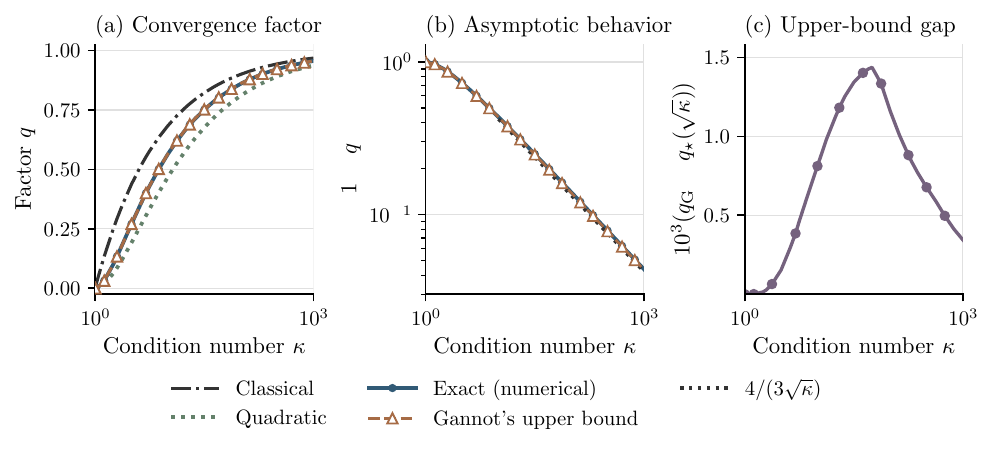}
\caption{Convergence factors for $1\leq\kappa\leq1000$.
(a) Numerical evaluations of $\qopt(\sqrt\kappa)$ using the angular conditions \eqref{eq:star-feasible} and \eqref{eq:qcyc}, compared with the quadratic rate and the classical and Gannot upper bounds.
At $\kappa=1$ we plot the limiting value zero; exact arithmetic certifies the values at $\kappa=10$ and $100$.
The remaining values are numerical approximations.
Filled circles show numerical evaluations of our rate formula, and hollow triangles show Gannot's upper bound.
(b) Comparison of $1-q$ with the asymptotic expression $4/(3\sqrt\kappa)$ from Corollary~\ref{cor:large-kappa}.
(c) The difference between Gannot's upper bound and the evaluated exact rate.
The largest displayed gap is $1.44\times10^{-3}$ (less than $0.18\%$ of Gannot's upper bound).}
\label{fig:nag-rate-curve}
\end{figure}

\begin{corollary}[Function gap and gradient norm]
\label{cor:metrics}
Replacing $R_t(f,x_0)$ in the two asymptotic factor definitions \eqref{eq:rho-fix}--\eqref{eq:rho-hor} by either
\[
 \frac{f(x_t)-f(x_\star)}{f(x_0)-f(x_\star)}
 \qquad\text{or}\qquad
 \frac{\|\nabla f(x_t)\|^2}{\|\nabla f(x_0)\|^2}
\]
gives the same exact factor as in Theorem~\ref{thm:main}.
The asymptotic factors for the distance $\|x_t-x_\star\|$ and the gradient norm $\|\nabla f(x_t)\|$ equal the square root of this factor.
\end{corollary}

\begin{proof}
For every $f\in\Fclass{\kappa}$ and $x$, strong convexity and smoothness give
\[
 \frac12\|x-x_\star\|^2\leq f(x)-f(x_\star)
 \leq\frac\kappa2\|x-x_\star\|^2,
 \qquad
 \|x-x_\star\|^2\leq\|\nabla f(x)\|^2
 \leq\kappa^2\|x-x_\star\|^2.
\]
Dividing by the corresponding initial values bounds the first ratio between $\kappa^{-1}R_t$ and $\kappa R_t$, and the second between $\kappa^{-2}R_t$ and $\kappa^2R_t$.
These constants depend only on $\kappa$, so the bounds remain valid when taking the supremum over dimensions, functions, and initial points.
Their $t$th roots tend to one, which proves equality of the asymptotic factors in both definitions.
\end{proof}

Classical convergence guarantees also give upper bounds on the gradient norm through strong convexity and smoothness~\citep{nesterov2018lectures,lessard2016analysis}.
Corollary~\ref{cor:metrics} identifies its exact asymptotic factor, attained by the same fixed function and initialization as the distance factor.
Equality of these asymptotic factors does not imply equality of finite-iteration bounds.

\begin{figure}[!t]
\centering
\includegraphics[width=\linewidth]{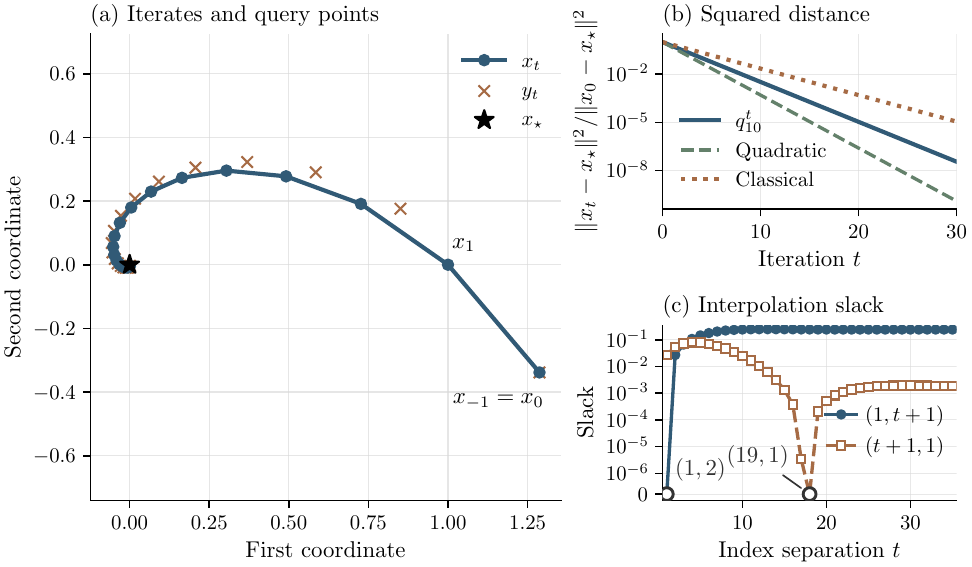}
\caption{The exact $\kappa=10$ example.
(a) Iterates and query points from the standard initialization.
(b) The normalized squared distance equals $q_{10}^t$, compared with the powers of the quadratic and classical factors.
(c) Slack in the interpolation inequalities for the ordered pairs $(1,t+1)$ and $(t+1,1)$.
Equality holds at $(1,2)$ and $(19,1)$.
The vertical axis is logarithmic above $10^{-6}$ and linear near zero.
Every comparison involving the initialization query is strict.
}
\label{fig:nag-exact-example}
\end{figure}

\section{A worst-case example at \texorpdfstring{$\kappa=10$}{kappa=10}}
\label{sec:example}

At $\kappa=10$, an optimal angular measure is concentrated at a single angle in $(0,\pi)$.
The corresponding trajectory therefore lies in $\reals^2$.
Identifying $\reals^2$ with the complex plane, the optimal factor and corresponding complex multiplier are
\begin{align*}
 q_{10}&=\qopt(\sqrt{10})
 \approx0.564423,\\
 z_\star
 &\approx0.726601+0.190984\,i,
 \qquad |z_\star|^2=q_{10}.
\end{align*}
The exact algebraic definitions are given in Appendix~\ref{app:kappa-ten}.

There exists a function $f\in\Fclass{10}(\reals^2)$, with minimizer at the origin, for which Nesterov's method produces
\begin{equation}
\label{eq:exact-spiral}
 x_{-1}=x_0=z_\star^{-1},
 \qquad
 x_t=z_\star^{\,t-1},\quad t\geq1.
\end{equation}
Each iteration rotates the error by the angle $\theta_\star=\arg z_\star\approx0.257032$ radians and multiplies its squared norm by $q_{10}$.
Therefore,
\begin{equation}
\label{eq:exact-decay}
 \frac{\|x_t-x_\star\|^2}{\|x_0-x_\star\|^2}
 =q_{10}^t,
 \qquad t\geq0.
\end{equation}
Appendix~\ref{app:kappa-ten} gives formulas for the query points, gradients, and function values.
These data satisfy the interpolation inequalities \eqref{eq:ssc-interpolation} for every ordered pair of indices.

For this example, the inequalities are active for the ordered pairs
\[
 (i,j)=(t,t+1)
 \qquad\text{and}\qquad
 (i,j)=(t+18,t),
 \qquad t\geq1.
\]
All inequalities comparing the initial query point $y_0$ with a later query point are strict.
The matching upper bound uses interpolation inequalities between query points at most eighteen iterations apart.
It proves that the worst-case asymptotic factor is at most $q_{10}$.
Independent checks using rigorous interval arithmetic and exact integer arithmetic verify the algebraic solution and all interpolation inequalities.
Appendix~\ref{app:kappa-ten} derives this bound from the inequalities for the cycle $(0,1,\ldots,18,0)$ and Proposition~\ref{prop:strict-separator-rate}.

\section*{Acknowledgements}

This research was conducted in the context of the Agentic AI for Mathematics (EF-LiOpt-3) project at the Berlin Mathematics Research Center MATH$^+$ (EXC-2046/2, project ID 390685689), funded by the Deutsche Forschungsgemeinschaft (DFG, German Research Foundation) under Germany's Excellence Strategy.
This research was also supported by a Sloan Research Fellowship from the Alfred P. Sloan Foundation (grant FG-2026-79479) and the ONR YIP Award N000142512147.
This work has been partially supported by OpenAI's Academic Researchers program.
The authors also thank Volker Mehrmann for helpful discussions.

\section*{AI use statement}
Generative AI tools assisted with literature comparison, formulation and
revision of mathematical arguments, proof writing, implementation of
verification and plotting scripts, formalization, and manuscript editing.
The associated checks include symbolic calculations, numerical and adversarial
tests, and machine-checked proofs for the results covered by the formalization.
The authors have reviewed all AI-assisted work, verified the claims, and take responsibility for the final manuscript.

\appendix
\section{Proof roadmap and dependencies}
\label{app:proof-roadmap}

The proof of Theorem~\ref{thm:main} has four steps.
First, an explicit finite cycle shows that cyclic monotonicity allows the interpolation values to be chosen as $\varphi_i=q^i\varphi_0$.
For a fixed factor $q$ and angular measure, determining these values therefore reduces to finding one scalar $\varphi_0$ satisfying \eqref{eq:star-feasible}.
Second, if the interpolation conditions have no solution at a given $q$, compactness gives a finite collection of cycle inequalities that no angular measure satisfies.
Convex separation then gives a \emph{strict cycle separator}: a nonnegative weighted sum of their costs that is strictly negative at every angle.
Third, suppose $\qopt(s)>q_{\mathrm{quad}}(s)$ and choose an optimal angular measure, whose existence follows from Lemma~\ref{lem:closed-rate-problem}.
For every $q\in(\qopt(s),1)$, the interpolation conditions have no solution.
We take a limit of the corresponding normalized separators as $q$ decreases to $\qopt(s)$.
This limit is nonpositive and has average zero under the chosen optimal measure.
The measure is therefore supported on the zeros of the limiting function.
We show that this function has only finitely many zeros.
We use the resulting measure to construct one finite-dimensional objective with the standard initialization and asymptotic factor $\qopt(s)$.
The quadratic example handles $\qopt(s)=q_{\mathrm{quad}}(s)$.
Fourth, every strict separator gives a uniform bound
$\|x_t-x_\star\|^2\leq b(s,q)q^t\|x_0-x_\star\|^2$.
Taking $t$th roots before letting $q\downarrow\qopt(s)$ proves equality of both asymptotic factors.

The lower-bound construction uses two limiting arguments.
Compactness of probability measures handles all cycle constraints without assuming a bound on the number of support angles.
Explicit bounds on the derivatives justify Taylor expansions of the limiting separator.
Together with Rolle's theorem and the fact that the limit is nonzero, these expansions show that it has only finitely many zeros.
Appendix~\ref{app:uniform-upper-bound} combines finite sums with the NAG recurrence to bound the sum of squared rescaled iterates.
A finite induction establishes that this sum is finite for every trajectory.
The resulting constant $b(s,q)$ is independent of the objective, dimension, initialization, and iteration count.

\section{Proof of the interpolation conditions}
\label{app:star-reduction}

\begin{proof}[Proof of Lemma~\ref{lem:stars}]
Fix $s>1$, $q\in(0,1)$, and a probability measure $\nu$.
For integers $t,u\geq1$, define
\[
 b_t=\frac{\bar c_{0t}(\nu)}{1-q^t},
 \qquad a_u=-\frac{\bar c_{u0}(\nu)}{1-q^u}.
\]
The interpolation conditions ask for a scalar between every lower bound $a_u$ and every upper bound $b_t$.
Assume cyclic monotonicity and consider the finite cycle
\[
 0,t,2t,\ldots,ut,(t-1)u,(t-2)u,\ldots,u,0.
\]
It first takes $u$ steps of length $t$ upwards, then $t$ steps of length $u$ downwards.
Repeated indices are allowed.
Using $\bar c_{i+k,j+k}=q^k\bar c_{ij}$, its cost is
\[
 \bar c_{0t}\sum_{j=0}^{u-1}q^{jt}
 +\bar c_{u0}\sum_{j=0}^{t-1}q^{ju}
 =(1-q^{tu})(b_t-a_u)\geq0.
\]
Since $q<1$, this gives $a_u\leq b_t$ for every $t,u\geq1$.
Consequently,
\[
 \varphi_0:=\sup_{u\geq1}a_u
\]
is finite, lies above every $a_u$, and lies below every $b_t$.
This proves \eqref{eq:star-feasible}.

Conversely, suppose \eqref{eq:star-feasible} holds and set $\varphi_i=q^i\varphi_0$.
For $j>i$, apply the upper bound in \eqref{eq:star-feasible} with $t=j-i$ to obtain
\[
 \varphi_i-\varphi_j
 =q^i(1-q^{j-i})\varphi_0
 \leq q^i\bar c_{0,j-i}=\bar c_{ij}.
\]
For $i>j$, the lower bound in \eqref{eq:star-feasible} with $t=i-j$ gives the same inequality.
The case $i=j$ is immediate.
Summing these inequalities around any finite cycle proves cyclic monotonicity.
\end{proof}

\subsection{Feasibility and attainment}

\begin{lemma}[Feasibility, the classical bound, and attainment]
\label{lem:closed-rate-problem}
For $s>1$, the set $\mathcal Q(s)$ in \eqref{eq:feasible-factors} is nonempty and has a largest element.
Every feasible factor $q$ satisfies $q\leq1-1/s$.
\end{lemma}

\begin{proof}
Fix $(q,\nu,\varphi_0)$ satisfying \eqref{eq:star-feasible}.
Since $|z|=\sqrt q<1$, the edge formula gives the uniform limits
\[
 c_{t0}(\theta)\longrightarrow0,
 \qquad c_{0t}(\theta)\longrightarrow\operatorname{Re}d_s(z).
\]
Writing $m=\int\operatorname{Re}d_s(z)\,d\nu$, the interpolation conditions imply
\begin{equation}
\label{eq:star-origin-bounds}
 0\leq\varphi_0\leq m,
 \qquad \bar c_{01}(\nu)\geq(1-q)\varphi_0\geq0.
\end{equation}
These bounds also justify including the optimum, with $(\sigma,r,\varphi)=(0,0,0)$.
Indeed, for the prescribed trajectory they give
$0\leq\varphi_i=q^i\varphi_0\leq q^im=\langle r_i,\sigma_i\rangle$, precisely the two interpolation inequalities with the optimum.

We show directly that $q\geq1-1/s$ is impossible.
Put $a=1-1/s\in(0,1)$.
Expanding $c_{02}(\theta)$, which compares iterations $0$ and $2$, gives
\[
 \frac{c_{02}(\theta)}{s^2-1}
 =a^2-1+\frac{2a(1-q)}{\sqrt q}\cos\theta
       +\frac{q^2-a^2}{q}\cos(2\theta).
\]
If $q\geq a$, both cosine coefficients are nonnegative.
Since $\sqrt q\geq\sqrt a>a$, it follows that
\[
 c_{02}(\theta)\leq c_{02}(0)
 =-(s^2-1)(1-q)(1-a/\sqrt q)^2<0.
\]
This contradicts
$\bar c_{02}(\nu)\geq(1-q^2)\varphi_0\geq0$, which follows from
\eqref{eq:star-feasible} and \eqref{eq:star-origin-bounds}.
Thus every feasible factor satisfies $q<1-1/s$.

The choice $q=q_{\mathrm{quad}}(s)$ and $\nu=\delta_0$ has $d_s(z)=0$ and all cycle costs zero, so it is feasible with $\varphi_0=0$.
To prove attainment, restrict $q$ to the compact interval $[q_{\mathrm{quad}}(s),1-1/s]$ and impose all cycle inequalities on $(q,\nu)$.
For each fixed cycle its cost is jointly continuous in $q$ and $\theta$ on this interval times $[0,\pi]$.
Each averaged inequality therefore defines a closed set under convergence of $q$ and weak convergence of $\nu$.
Their intersection is a nonempty closed subset of the compact product
$[q_{\mathrm{quad}}(s),1-1/s]\times\ProbMeasures([0,\pi])$.
The continuous objective $q$ attains its maximum there.
Lemma~\ref{lem:stars} supplies the corresponding scalar $\varphi_0$.
Since smaller factors cannot improve on $q_{\mathrm{quad}}(s)$, this is the largest feasible factor.
\end{proof}

\subsection{Proof of the fixed-domain linear formulation}
\label{app:fixed-domain-rate}

We prove Proposition~\ref{prop:fixed-domain-rate} using the bound $q\leq1-1/s$ from Lemma~\ref{lem:closed-rate-problem} and the equivalent interpolation conditions in Lemma~\ref{lem:stars}.
This proof does not use Theorem~\ref{thm:main}.
Put $\alpha_s=1-1/s$ and $J(\mu)=\int |z|^2\,d\mu(z)$.
Substituting $d_s(z)=-(s^2-1)(z-\alpha_s)^2/z^2$ into \eqref{eq:disk-cycle-cost} gives, for $\gamma=(n_0,\ldots,n_m=n_0)$,
\begin{equation}
\label{eq:disk-cycle-polynomial}
 \mathcal C_\gamma^s(z)
 =-\operatorname{Re}\!\left[
 (\bar z-\alpha_s)^2
 \sum_{j=0}^{m-1}z^{n_j+2}
 (\bar z^{n_j}-\bar z^{n_{j+1}})\right].
\end{equation}
This polynomial in $\operatorname{Re}z$ and $\operatorname{Im}z$ vanishes at $z=0$, justifying the continuous extension used in the main text.
Conjugating $z$ preserves the cost.
For a nonnegative integer $t$, let $\gamma+t$ denote the cycle obtained by adding $t$ to every index of $\gamma$.
Factoring $z^t\bar z^t=|z|^{2t}$ out of each summand gives
\begin{equation}
\label{eq:disk-cost-shift}
 \mathcal C_{\gamma+t}^s(z)
     =|z|^{2t}\mathcal C_\gamma^s(z),
 \qquad z\in\mathbb C.
\end{equation}

\begin{lemma}[Radial reweighting]
\label{lem:disk-radial-reweighting}
Let $\mu$ be feasible for \eqref{eq:disk-rate-lp} with $m=J(\mu)>0$.
Define $d\eta(z)=|z|^2\,d\mu(z)/m$, which gives greater relative weight to larger values of $|z|^2$.
Then $\eta$ is feasible for the same problem and
\begin{equation}
\label{eq:disk-variance-improvement}
 J(\eta)-J(\mu)
 =\frac{\int |z|^4\,d\mu-m^2}{m}
 =\frac{\operatorname{Var}_\mu(|z|^2)}{m}.
\end{equation}
\end{lemma}
\begin{proof}
Since $|z|^2\geq0$ and $\int |z|^2\,d\mu=m$, $\eta$ is a probability measure on $D_s$.
For every cycle $\gamma$, the shift identity \eqref{eq:disk-cost-shift} and feasibility of $\mu$ give
\[
 \int\mathcal C_\gamma^s\,d\eta
 =\frac1m\int |z|^2\mathcal C_\gamma^s\,d\mu
 =\frac1m\int\mathcal C_{\gamma+1}^s\,d\mu\geq0.
\]
Also $J(\eta)=\int |z|^4\,d\mu/m$, which proves the identity.
\end{proof}

\begin{proof}[Proof of Proposition~\ref{prop:fixed-domain-rate}]
Probability measures on the compact disk form a weakly compact set.
Every cycle cost is continuous, so the feasible set of \eqref{eq:disk-rate-lp} is closed.
It is nonempty: $\delta_{\alpha_s}$ has zero costs and objective $\alpha_s^2=q_{\mathrm{quad}}(s)$.
Hence the feasible set is nonempty and compact under weak convergence.
Since $z\mapsto|z|^2$ is continuous and bounded on $D_s$, the objective $J$ is continuous under weak convergence.
The extreme value theorem, applied in this topology, therefore gives a maximum $m\geq q_{\mathrm{quad}}(s)>0$.
If an optimal measure $\mu$ had $\operatorname{Var}_\mu(|z|^2)>0$, Lemma~\ref{lem:disk-radial-reweighting} would give a feasible measure with a strictly larger objective value.
Thus every optimal measure $\mu$ satisfies $\operatorname{Var}_\mu(|z|^2)=0$ and is supported on the circle $|z|^2=m$, with $0<m\leq\alpha_s<1$.

We now prove the fixed-factor equivalence. For any $q\in\mathcal Q(s)$,
Lemma~\ref{lem:closed-rate-problem} gives $q\leq\alpha_s$.
Choose an angular measure $\nu$ satisfying \eqref{eq:star-feasible}.
Replace each angle $\theta$ by the point $\sqrt q\,e^{i\theta}$ to obtain a probability measure $\mu$ on the corresponding circle in $D_s$.
By Lemma~\ref{lem:stars}, $\nu$ satisfies all cycle inequalities.
The positive scaling in \eqref{eq:disk-cycle-cost} then makes $\mu$ feasible for \eqref{eq:disk-rate-lp}, with objective value $q$.
Conversely, let $\mu$ be feasible for \eqref{eq:disk-rate-lp} and supported on $|z|^2=q$ for some $q\in(0,1)$.
Map each point $z=\sqrt q\,e^{\pm i\theta}$ to $\theta\in[0,\pi]$ to obtain an angular probability measure $\nu$.
Conjugation preserves every cost, and division by the common positive factor $q^2/(s^2-1)$ gives all cycle inequalities for $\nu$.
Lemma~\ref{lem:stars} then gives a scalar $\varphi_0$ satisfying \eqref{eq:star-feasible}, so $q\in\mathcal Q(s)$.
This argument remains valid at $q=q_{\mathrm{quad}}(s)$, even where $d_s(z)=0$, because it does not divide by $d_s(z)$.

Applying this equivalence to an optimizing disk measure gives
$m\in\mathcal Q(s)$.
Conversely, for every $q\in\mathcal Q(s)$, the construction above gives a feasible measure for \eqref{eq:disk-rate-lp} with objective value $q$, so $q\leq m$.
Thus $m=\max\mathcal Q(s)$, proving \eqref{eq:qcyc}.
The lower bound follows from $q_{\mathrm{quad}}(s)\in\mathcal Q(s)$, and the upper bound
from Lemma~\ref{lem:closed-rate-problem}.
\end{proof}

At $s=1$ the disk reduces to $\{0\}$ and the polynomial problem has value
zero; the division by $m$ is not needed in that case. The proof above does
not assert that every feasible measure has one radius, only that every
optimizing measure does.

\paragraph{Finite-index relaxations and duality.}
Let $q_{\mathrm{rel},n}(s)$ be the value of \eqref{eq:disk-rate-lp} when only cycles with
vertices in $\{0,\ldots,n\}$ are imposed. These conditions can equivalently
be written using real potentials $u_0=0,u_1,\ldots,u_n$ as
\[
 u_j-u_i\leq-\int\operatorname{Re}\!\left[
 z^{i+2}(\bar z-\alpha_s)^2(\bar z^i-\bar z^j)
 \right]\,d\mu(z),
 \qquad 0\leq i,j\leq n.
\]
Potentials imply the cycle inequalities by telescoping; conversely,
shortest-path costs from vertex zero give such potentials when no cycle
is negative.

\begin{proposition}[Finite-index bounds and duality]
\label{prop:disk-finite-dual}
Each $q_{\mathrm{rel},n}(s)$ is attained, and
\[
 \alpha_s=q_{\mathrm{rel},0}(s)\geq q_{\mathrm{rel},1}(s)\geq\cdots\geq\qopt(s),
 \qquad \lim_{n\to\infty}q_{\mathrm{rel},n}(s)=\qopt(s).
\]
Moreover,
\begin{equation}
\label{eq:disk-polynomial-dual}
 \qopt(s)=\inf\left\{b:
 \begin{array}{l}
  \Gamma_0\subset\Gamma\text{ finite},\quad\lambda_\gamma\geq0,\\
  b-|z|^2-\displaystyle\sum_{\gamma\in\Gamma_0}
       \lambda_\gamma\mathcal C_\gamma^s(z)\geq0,
       \quad z\in D_s
 \end{array}\right\}.
\end{equation}
The dual constraint requires a real polynomial in $\operatorname{Re}z$ and $\operatorname{Im}z$ to be nonnegative on $D_s$.
For every $b>\qopt(s)$ such a polynomial can be strictly positive on $D_s$.
No finite dual attainment at $b=\qopt(s)$ is asserted.
\end{proposition}
\begin{proof}
Each finite-index feasible set is closed in the compact probability space
and contains $\delta_{\alpha_s}$, so its maximum exists. The sets are
nested. A weak limit of optimizers along $n\to\infty$ satisfies every fixed
cycle constraint, since that constraint holds eventually. Continuity of
$J$ proves that the decreasing limit of $q_{\mathrm{rel},n}(s)$ equals $\qopt(s)$.

Integrating any polynomial inequality in \eqref{eq:disk-polynomial-dual}
against a feasible measure gives $J(\mu)\leq b$. Conversely, fix
$b>\qopt(s)$. Compactness supplies finitely many cycle constraints whose
intersection with $J(\mu)\geq b$ is already empty. Their compact convex
moment image is disjoint from
$\{t\in\reals:t\geq b\}\times\reals_+^{|\Gamma_0|}$.
Strict separation gives nonnegative coefficients
$(\lambda_0,(\eta_\gamma)_{\gamma\in\Gamma_0})$ with
\[
 \lambda_0|z|^2+\sum_{\gamma\in\Gamma_0}\eta_\gamma
           \mathcal C_\gamma^s(z)<\lambda_0b,
 \quad z\in D_s.
\]
The coefficient $\lambda_0$ is positive: otherwise evaluating at
$z=\alpha_s$, where all cycle costs vanish, contradicts strictness.
Division by $\lambda_0$ proves the claim.
\end{proof}

For fixed cycles and degree bounds, one can search for a sufficient dual
certificate of the form
\[
 b-|z|^2-\sum_{\gamma\in\Gamma_0}\lambda_\gamma
              \mathcal C_\gamma^s
   =S_0+(\alpha_s-|z|^2)S_1,
\]
where $S_0,S_1$ are sums of squares of real polynomials in
$\operatorname{Re}z$ and $\operatorname{Im}z$. Coefficient matching and
positive semidefinite Gram matrices give a semidefinite program with fixed
data. A checked polynomial certificate gives an upper bound, but no finite
exact degree or runtime improvement is claimed. Likewise, replacing the
disk by a finite support grid changes the feasible set: a grid computation
with truncated cycles is not automatically an upper or lower bound for the
original problem.

\section{Full proof of the exact-rate theorem}
\label{app:generic-bridges}

This section proves Theorem~\ref{thm:main}.
We first show that, when the interpolation conditions are infeasible, a nonnegative combination of finitely many cycle functions is strictly negative at every angle.
This separation argument is used in both the lower and upper bounds.

For the lower bound, we show that, if $\qopt(s)>q_{\mathrm{quad}}(s)$, an optimal feasible measure can be supported on finitely many angles.
We use this measure to construct one finite-dimensional objective and initial point satisfying $x_{-1}=x_0$ and attaining the asymptotic factor $\qopt(s)$.
If $\qopt(s)=q_{\mathrm{quad}}(s)$, a quadratic objective attains the factor.

For the upper bound, fix any $q\in(\qopt(s),1)$.
The negative combination of cycle functions gives a convergence bound with a constant depending only on $s$ and $q$.
The bound therefore holds uniformly over objectives, dimensions, initial points, and iteration counts.
Letting $q\downarrow\qopt(s)$ identifies both worst-case asymptotic factors with $\qopt(s)$.

The separation argument uses convex hulls and the separating hyperplane theorem.
Compactness of probability measures connects the finite formulation to the original interpolation conditions.
Rockafellar's cyclic-monotonicity theorem supplies the interpolating objective.
Taylor's theorem and Rolle's theorem show that an optimal measure has finite support.
The upper bound uses finite sums and induction applied to the rescaled NAG recurrence.
Each argument below states its hypotheses and checks the required implications.

\subsection{Finite alternatives, attainment, and finite support}

\begin{lemma}[Finite certificate of infeasibility]
\label{lem:strict-alternative}
Fix $s>1$ and $q\in(q_{\mathrm{quad}}(s),1)$.
For a finite cycle $\gamma=(i_0,i_1,\ldots,i_m=i_0)$, define the sum of the interpolation terms \eqref{eq:edge-cost} by
\[
 c_\gamma(\theta)=\sum_{j=0}^{m-1}c_{i_j i_{j+1}}(\theta).
\]
Exactly one of the following holds:
\begin{enumerate}
 \item Every finite list of cycles $\gamma_1,\ldots,\gamma_n$ admits angles $\theta_1,\ldots,\theta_m\in[0,\pi]$ and weights $a_i\geq0$, with $1\leq m\leq n+1$, satisfying
 \[
  \sum_{i=1}^m a_i=1,
  \qquad
  \sum_{i=1}^m a_i c_{\gamma_j}(\theta_i)\geq0,
  \qquad j=1,\ldots,n.
 \]
 \item There are finitely many cycles $\gamma_1,\ldots,\gamma_n$
 and weights $w_j\geq0$ such that
 \[
  \sum_{j=1}^nw_j
  c_{\gamma_j}(\theta)<0,
  \qquad 0\leq\theta\leq\pi.
 \]
\end{enumerate}
The angles and weights in alternative~1 may depend on the selected cycles.
\end{lemma}

\emph{Meaning of the weights.}
The weights $a_i$ in alternative~1 average the cycle costs over the selected angles and sum to one.
The weights $w_j$ in alternative~2 combine cycle inequalities and need not sum to one.
The same $w_j$ give a strictly negative sum at every angle, so its average is also negative for every choice of angles and nonnegative weights $a_i$ summing to one.
Thus no such mixture can satisfy all the selected cycle inequalities.

\begin{proof}
For any fixed list $\gamma_1,\ldots,\gamma_n$, consider the cost vectors
\[
 (c_{\gamma_1}(\theta),\ldots,c_{\gamma_n}(\theta)),
 \qquad \theta\in[0,\pi].
\]
Their convex hull is compact because the cycle costs are continuous.
By Carath\'eodory's theorem, every point in this hull is a convex combination of at most $n+1$ cost vectors.
Thus the hull intersects $\reals_+^n$ exactly when the inequalities in alternative~1 hold for the selected cycles.

If alternative~1 fails, choose a list for which the hull and $\reals_+^n$ are disjoint.
The separating hyperplane theorem gives a vector $w$ such that
\[
 \max_{\theta\in[0,\pi]}\sum_{j=1}^n w_jc_{\gamma_j}(\theta)
 <\inf_{v\in\reals_+^n}w^Tv.
\]
The vector $w$ must be nonnegative, since a negative component would make the infimum $-\infty$.
The infimum is therefore zero, which gives alternative~2.

If both alternatives held, choose angles and weights from alternative~1 for the cycles in alternative~2.
Then
\[
 0\leq\sum_{j=1}^n w_j\sum_{i=1}^m a_i c_{\gamma_j}(\theta_i)
 =\sum_{i=1}^m a_i\sum_{j=1}^n w_j c_{\gamma_j}(\theta_i)<0,
\]
a contradiction.
\end{proof}

\paragraph{Connection with the measure formulation.}
Alternative~1 is equivalent to the existence of a probability measure satisfying all cycle inequalities.
For one direction, the average of each finite cost vector under such a measure lies in the compact convex hull used above.
Carath\'eodory's theorem then gives the angles and weights in alternative~1.
Conversely, alternative~1 gives a discrete probability measure for each finite list of constraints.
Probability measures on $[0,\pi]$ form a compact set when convergence is tested by averages of continuous functions.
Each cycle inequality defines a closed subset, and every finite intersection is nonempty.
Compactness therefore gives one measure satisfying all cycle inequalities.
When $\qopt(s)>q_{\mathrm{quad}}(s)$, the next lemma gives one finite set of angles and weights satisfying all cycles at the optimal factor.

\begin{lemma}[Finite support at the optimal factor]
\label{lem:finite-optimal-support}
Fix $s>1$ and suppose $\qopt(s)>q_{\mathrm{quad}}(s)$.
There are a probability measure $\nu_\star$ supported on finitely many angles in $[0,\pi]$ and a real scalar $\varphi_0^\star$ that satisfy \eqref{eq:star-feasible} at $q=\qopt(s)$.
\end{lemma}

\begin{proof}
Write $q_\star=\qopt(s)$ and choose an optimal pair $(\nu_\star,\varphi_0^\star)$.
We construct a function that vanishes on the support of $\nu_\star$ and has only finitely many zeros.
Since $q$ varies in this proof, we write $c_{ij}(q,\theta)$ for the term in \eqref{eq:edge-cost}, evaluated at $z=\sqrt q\,e^{i\theta}$ with $s$ fixed.
For $t\geq1$, put
\[
 u_t(q,\theta)=\frac{c_{0t}(q,\theta)}{1-q^t},
 \qquad v_t(q,\theta)=\frac{c_{t0}(q,\theta)}{1-q^t},
 \qquad d(q,\theta)=\operatorname{Re}d_s(\sqrt q\,e^{i\theta}).
\]
Here $u_t$ and $v_t$ compare iterations $0$ and $t$ in the two possible orders.
At $q_\star$, the interpolation inequalities give
$\int u_t\,d\nu_\star\geq\varphi_0^\star$ and
$\int v_t\,d\nu_\star\geq-\varphi_0^\star$.

Choose $q_\star<q_+<1$ and $q_n\downarrow q_\star$ with $q_\star<q_n\leq q_+$.
By Lemma~\ref{lem:strict-alternative} and the compactness argument above, there is a strict cycle separator at every $q_n$.
The interpolation term $c_{ij}(q,\theta)$ in \eqref{eq:edge-cost} compares iterations $i$ and $j$ and satisfies
\[
 c_{ij}(q,\theta)=
 \begin{cases}
  (q^i-q^j)u_{j-i}(q,\theta), & i<j,\\
  (q^j-q^i)v_{i-j}(q,\theta), & i>j,\\
  0, & i=j.
 \end{cases}
\]
For a cycle $\gamma=(i_0,\ldots,i_m=i_0)$, the differences telescope to $\sum_{k=0}^{m-1}(q^{i_k}-q^{i_{k+1}})=q^{i_0}-q^{i_m}=0$.
Since $0<q<1$, the positive differences are the coefficients of the $u_t$ terms, and the absolute values of the negative differences are the coefficients of the $v_t$ terms.
Their totals therefore agree for each cycle, and hence for every nonnegative weighted sum of cycles.
For each $n$, choose one strict separator at $q_n$ and expand its cycle costs using the displayed formula for $c_{ij}$.
Collect the coefficients of each $u_t(q_n,\theta)$ and $v_t(q_n,\theta)$.
Their common total is positive, since a zero total would make the separator identically zero.
Define $a_{t,n}$ and $b_{t,n}$ as these respective coefficients divided by their common total.
They are nonnegative, only finitely many are nonzero for each $n$, and
\[
 \sum_{t\geq1}a_{t,n}=\sum_{t\geq1}b_{t,n}=1.
\]
These coefficients depend on the separator chosen at $q_n$ and remain fixed as $q$ varies.
For $q\in[q_\star,q_+]$ and $\theta\in[0,\pi]$, define
\begin{equation}
\label{eq:lag-separator-limit}
 w_n(q,\theta)
 =\sum_{t\geq1}a_{t,n}u_t(q,\theta)
 +\sum_{t\geq1}b_{t,n}v_t(q,\theta).
\end{equation}
At $q=q_n$, this is the chosen separator divided by its positive normalizing constant, so $w_n(q_n,\theta)<0$ on $[0,\pi]$.
At $q=q_\star$, the interpolation inequalities and the two coefficient sums give $\int w_n(q_\star,\theta)\,d\nu_\star\geq\varphi_0^\star-\varphi_0^\star=0$.
Since $\sum_t a_{t,n}=1$, we can also write
\[
 w_n(q,\theta)=d(q,\theta)
 +\sum_{t\geq1}a_{t,n}\bigl(u_t(q,\theta)-d(q,\theta)\bigr)
 +\sum_{t\geq1}b_{t,n}v_t(q,\theta).
\]

We next bound $u_t-d$ and $v_t$ to justify taking a limit in $w_n$.
We also bound their derivatives with respect to $\theta$ to control the Taylor remainder of the limiting function later in the proof.
Throughout this calculation, $q$ is held fixed when differentiating.
With $z=\sqrt q\,e^{i\theta}$, the definitions and \eqref{eq:edge-cost} give
\[
 \begin{aligned}
 u_t(q,\theta)-d(q,\theta)
 &=\frac{q^t d(q,\theta)-\operatorname{Re}[\overline{d_s(z)}\bar z^t]}{1-q^t},\\
 v_t(q,\theta)
 &=\frac{q^t d(q,\theta)-\operatorname{Re}[\overline{d_s(z)}z^t]}{1-q^t}.
 \end{aligned}
\]
To identify the cosine terms in these expressions, write \eqref{eq:nag-transfer} as
\[
 d_s(z)=-(s^2-1)\left(1-\frac{2(1-1/s)}{z}
 +\frac{(1-1/s)^2}{z^2}\right).
\]
Since $\operatorname{Re}(e^{ik\theta})=\cos(k\theta)$ and $\cos(-k\theta)=\cos(k\theta)$, both $u_t-d$ and $v_t$ are finite sums of $\cos(k\theta)$ with integer indices $0\leq k\leq t+2$.
For any finite sum $h(\theta)=\sum_k p_k\cos(k\theta)$, the triangle inequality gives $|h(\theta)|\leq\sum_k|p_k|$ at every angle.
The combined sum of the absolute values of the cosine coefficients of $u_t-d$ and $v_t$ is at most
\[
 \frac{2(s^2-1)(q^t+q^{t/2})}{1-q^t}
 \left(1+\frac{1-1/s}{\sqrt q}\right)^2
 \leq c r_+^t,
 \qquad q\in[q_\star,q_+],
\]
where
\[
 r_+=\sqrt{q_+}<1,\qquad
 c=\frac{4(s^2-1)}{1-q_+}
 \left(1+\frac{1-1/s}{\sqrt{q_\star}}\right)^2.
\]
Each differentiation multiplies the absolute value of a sine or cosine term's coefficient by its index $k$.
After $j$ derivatives, the absolute value of each coefficient is therefore multiplied by at most $(t+2)^j$.
Applying the same triangle inequality gives
\begin{equation}
\label{eq:lag-geometric-tails}
 \left|\frac{d^j}{d\theta^j}(u_t-d)\right|
 +\left|\frac{d^j}{d\theta^j}v_t\right|
 \leq c r_+^t(t+2)^j,
 \qquad j\geq0,\quad\theta\in\reals.
\end{equation}
For $j=0$, this bounds the functions themselves and shows that $u_t-d$ and $v_t$ tend to zero uniformly in $q$ and $\theta$ on $[q_\star,q_+]\times[0,\pi]$.
The function $d$ is a linear combination of $1$, $\cos\theta$, and $\cos(2\theta)$, and satisfies $|\frac{d^j}{d\theta^j}d|\leq c2^j$.

The separator coefficients $a_{t,n}$ and $b_{t,n}$ lie in $[0,1]$.
Choose nested subsequences along which the coefficients converge for $t=1$, then for $t=2$, and so on.
Taking the $n$th element of the $n$th subsequence gives a single subsequence with $a_{t,n}\to a_t$ and $b_{t,n}\to b_t$ for every fixed $t$.
At the fixed factor $q_\star$, define the function of the angle
\[
 w_\star(\theta)=d(q_\star,\theta)
 +\sum_{t\geq1}a_t\bigl(u_t(q_\star,\theta)-d(q_\star,\theta)\bigr)
 +\sum_{t\geq1}b_tv_t(q_\star,\theta).
\]
Along the chosen subsequence, both $w_n(q_n,\theta)$ and $w_n(q_\star,\theta)$ converge to $w_\star(\theta)$ uniformly for $\theta\in[0,\pi]$.
Indeed, the bound with $j=0$ makes the sums over sufficiently large values of $t$ uniformly small, and each finite sum converges by coefficient convergence and continuity in $q$.
The limiting coefficient sums may be less than one because weights can move to larger and larger values of $t$ as $n$ increases.
The term $d$ outside the sums preserves their limiting contribution, since $u_t$ tends to $d$ and $v_t$ tends to zero as $t$ increases.
Passing to these limits in $w_n(q_n,\theta)<0$ and $\int w_n(q_\star,\theta)\,d\nu_\star\geq0$ gives
\[
 w_\star(\theta)\leq0\quad(0\leq\theta\leq\pi),
 \qquad \int w_\star(\theta)\,d\nu_\star\geq0.
\]
Since $w_\star$ is continuous and nonpositive, its nonnegative integral forces it to vanish on the support of $\nu_\star$.

At angle zero, writing $r=\sqrt q$ gives
\[
 u_t(q,0)=\frac{d_s(r)}{1+r^t},
 \qquad v_t(q,0)=-\frac{d_s(r)r^t}{1+r^t}.
\]
Since $r^t\leq r<1$ and $d_s(r)<0$, the two coefficient sums imply
$w_n(q,0)\leq d_s(r)(1-r)/(1+r)$.
Consequently,
\[
 w_\star(0)\leq d_s(\sqrt{q_\star})
 \frac{1-\sqrt{q_\star}}{1+\sqrt{q_\star}}<0.
\]

It remains to show that $w_\star$ has only finitely many zeros in $[0,\pi]$.
The same series defines $w_\star$ for every real $\theta$.
The derivative bound \eqref{eq:lag-geometric-tails} permits termwise differentiation of every order, uniformly on the real line.
Choose $b>0$ with $r_+e^b<1$.
Using $(t+2)^j\leq j!b^{-j}e^{b(t+2)}$ and summing the geometric series gives
\[
 \sup_{\theta\in\reals}|w_\star^{(j)}(\theta)|
 \leq c_0j!b^{-j},
 \qquad c_0=\frac{ce^{2b}}{1-r_+e^b}.
\]
For any real expansion point $\theta_0$, Taylor's theorem therefore gives
\[
 \left|w_\star(\theta)
 -\sum_{j=0}^m\frac{w_\star^{(j)}(\theta_0)}{j!}(\theta-\theta_0)^j\right|
 \leq c_0\left(\frac{|\theta-\theta_0|}{b}\right)^{m+1}.
\]
When $|\theta-\theta_0|<b$, the right-hand side tends to zero as $m\to\infty$, so the Taylor series equals $w_\star$ on this neighborhood.
The same radius $b$ works at every real expansion point.

Suppose that $w_\star$ has infinitely many distinct zeros in $[0,\pi]$.
By compactness, we can choose a strictly monotone sequence of these zeros converging to some $\theta_0\in[0,\pi]$, possibly an endpoint.
Between consecutive zeros, Rolle's theorem gives a zero of $w_\star'$.
These zeros also converge to $\theta_0$, and the same argument applies successively to every derivative.
By continuity, $w_\star^{(j)}(\theta_0)=0$ for every $j\geq0$.
The Taylor expansion then gives $w_\star(\theta)=0$ whenever $|\theta-\theta_0|<b$.
At every point inside this interval all derivatives vanish, so a Taylor expansion there gives another interval of radius $b$ on which $w_\star$ is zero.
Repeating with expansion points spaced $b/2$ apart toward zero eventually gives $w_\star(0)=0$, contradicting the strict inequality above.
Thus $w_\star$ has only finitely many zeros in $[0,\pi]$, and the support of $\nu_\star$ is finite.
\end{proof}

\subsection{Finite-dimensional realization with the standard initialization}

An optimal angular measure with finite support gives finitely many planar components, each rotating and contracting by $\sqrt{\qopt(s)}$ at every iteration.
We must still enforce the standard initialization $x_{-1}=x_0$.
We scale the prescribed trajectory, add one scalar coordinate, and choose an initial point and gradient consistent with the NAG updates.
The additional coordinate supplies slack in the interpolation inequalities involving the initial query point.
Its iterate is zero from iteration two onward, so this coordinate does not affect the asymptotic factor.

\begin{lemma}[Enforcing the standard initialization]
\label{lem:initialization-splice}
Fix $s>1$ with $\qopt(s)>q_{\mathrm{quad}}(s)$, and set $\kappa=s^2$.
Consider any trajectory associated with a finitely supported optimal angular pair as in Section~\ref{sec:variational}.
After scaling this trajectory by a positive constant and adding one coordinate, it can be realized from iteration one onward by Nesterov's method on a single finite-dimensional objective in $\Fclass{\kappa}$.
The method starts with $x_{-1}=x_0$, and its squared-distance asymptotic factor is $\qopt(s)$.
\end{lemma}

\begin{proof}
Write $\beta=\beta_\kappa\in(0,1)$ and set
\[
 a=\frac{\kappa}{1+\beta}=\frac{s(s+1)}2,
 \qquad
 p(v)=
 \begin{cases}
  a v^2/2,&v\geq0,\\
  \kappa v^2/2,&v\leq0.
 \end{cases}
\]
Since $1<a<\kappa$, this differentiable piecewise quadratic belongs to $\Fclass{\kappa}$.
Let $u_t$ and $v_t$ denote the scalar iterates and query points, respectively.
Starting NAG on $p$ from $u_{-1}=u_0=1$ gives
\[
 u_1=\frac{\beta}{1+\beta},\qquad u_t=0\quad(t\geq2).
\]
Indeed, the corresponding query points are
\[
 v_0=1,\qquad v_1=0,\qquad
 v_2=-\frac{\beta^2}{1+\beta},\qquad v_t=0\quad(t\geq3).
\]
The first gradient step gives $u_1$, the next extrapolation gives $v_1=0$, and every query $v\leq0$ is mapped to zero by the gradient step.

Write the interpolation condition \eqref{eq:ssc-interpolation} as $e_{ij}\leq0$, where
\[
 e_{ij}
 =f_j-f_i+\langle g_j,y_i-y_j\rangle
 +\frac{\|g_i-g_j\|^2}{2\kappa}
 +\frac{\kappa}{2(\kappa-1)}
 \left\|y_i-y_j-\frac{g_i-g_j}{\kappa}\right\|^2.
\]
Let $e_{ij}^p$ be this residual with $y_t=v_t$, $g_t=p'(v_t)$, and $f_t=p(v_t)$.
Every later query and the optimum satisfy $v_t\leq0$, with $v_\star=0$.
Substituting $(y_0,g_0,f_0)=(1,a,a/2)$ and $(y_t,g_t,f_t)=(v_t,\kappa v_t,\kappa v_t^2/2)$ gives
\[
 e^p_{0t}=-c+(\kappa-a)v_t\leq-c,
 \qquad e^p_{t0}=-c,
 \qquad
 c=\frac{(a-1)(\kappa-a)}{2(\kappa-1)}>0.
\]
Thus the inequalities for index pairs $(0,t)$ and $(t,0)$ have uniform slack $c$.
Pairs of later query points, including the optimum, have residual zero because $p(v)=\kappa v^2/2$ for $v\leq0$.

Use hats for all quantities associated with the prescribed trajectory, beginning at $t=1$.
The extrapolation equation requires
\[
 \widehat y_1
 =(1+\beta)\widehat x_1-\beta \widehat x_0,
\]
so set
\[
 \widehat x_0
 =\frac{(1+\beta)\widehat x_1-\widehat y_1}{\beta}.
\]
Set $\widehat x_{-1}=\widehat y_0=\widehat x_0$ and prescribe
\[
 \widehat g_0
 =\kappa(\widehat x_0-\widehat x_1).
\]
Then the first gradient step reaches $\widehat x_1$, and the next extrapolation gives $\widehat y_1$.
The NAG updates do not depend on $\widehat f_0$, so assign it any finite value.
Scaling the prescribed trajectory and adding the scalar coordinate will enforce the interpolation inequalities involving index $0$.

At the optimum, prescribe
\[
 (\widehat y_\star,\widehat g_\star,\widehat f_\star)
 =(0,0,0),
\]
which is compatible with the prescribed trajectory by the proof of Lemma~\ref{lem:closed-rate-problem}.
The prescribed points and gradients are bounded because $\qopt(s)<1$.
The interpolation inequalities with the optimum imply
\[
 0\leq \widehat f_t
 \leq\langle \widehat g_t,\widehat y_t\rangle
 \leq\|\widehat g_t\|\|\widehat y_t\|.
\]
The prescribed function values are therefore bounded as well.
Since $\widehat y_0$, $\widehat g_0$, and $\widehat f_0$ are fixed, the interpolation residuals involving index $0$ are uniformly bounded:
\[
 b=\sup_{t\in\{\star,1,2,\ldots\}}
 \max\{0,\widehat e_{0t},\widehat e_{t0}\}<\infty.
\]

For $\varepsilon>0$, define the combined iterates, query points, gradients, and function values by
\[
 x_t=(u_t,\varepsilon \widehat x_t),
 \qquad
 y_t=(v_t,\varepsilon \widehat y_t),
\]
\[
 g_t=(p'(v_t),\varepsilon \widehat g_t),
 \qquad
 f_t=p(v_t)+\varepsilon^2\widehat f_t.
\]
At the optimum, set $(y_\star,g_\star,f_\star)=(0,0,0)$.
Inner products and squared norms add across coordinates, so
\[
 e_{ij}=e_{ij}^p+\varepsilon^2\widehat e_{ij}.
\]
For pairs with neither index equal to $0$, we have $\widehat e_{ij}\leq0$, and therefore $e_{ij}\leq0$.
Choose
\[
 \varepsilon^2=\frac{c}{b+1}.
\]
For every $t\in\{\star,1,2,\ldots\}$, this gives
\[
 \max\{e_{0t},e_{t0}\}
 \leq-c+\varepsilon^2b
 =-\frac{c}{b+1}<0.
\]
Together with $e_{00}=0$, this proves all the interpolation inequalities.

Summing the transformed interpolation inequalities around any finite cycle proves cyclic monotonicity of the entire countable set of pairs $(\sigma_i,r_i)$.
The construction in Section~\ref{sec:interpolation} therefore gives one function $f\in\Fclass{\kappa}$ with $\nabla f(y_i)=g_i$ at every prescribed query point, including the optimum.
This applies to the whole set at once and uses the same finite-dimensional space as the combined trajectory.

The constructed sequences satisfy the NAG updates \eqref{eq:nag} because $\nabla f(y_t)=g_t$.
Thus $(x_t)$ is the trajectory of Nesterov's method on $f$, with $x_{-1}=x_0$.
Interpolation at the origin gives $\nabla f(0)=0$, so $x_\star=0$.
The initial point is nonoptimal because its first coordinate is $u_0=1$.

Let $\theta_k$ be a support angle of the optimal measure and put $z_k=\sqrt{\qopt(s)}\,e^{i\theta_k}$.
The construction in Section~\ref{sec:rotating-trajectories} makes the corresponding component of $r_t$ proportional to $z_k^t$.
Since $x_{t+1}=(\kappa-1)r_t/\kappa$, the corresponding prescribed iterate, written as a complex number $\widehat x_{t,k}$, satisfies
\[
 \widehat x_{t,k}=z_k^{t-1}\widehat x_{1,k},
 \qquad
 |\widehat x_{t,k}|^2=\qopt(s)^{t-1}|\widehat x_{1,k}|^2,
 \qquad t\geq1.
\]
All $z_k$ have modulus $\sqrt{\qopt(s)}$, so the norm of each prescribed planar component contracts by this same factor.
The components occupy separate coordinates, so their squared norms add.
Including the additional scalar coordinate gives, for $t\geq1$,
\[
 \|x_t\|^2
 =u_t^2+\varepsilon^2\|\widehat x_t\|^2
 =u_t^2+\varepsilon^2\qopt(s)^{t-1}\|\widehat x_1\|^2.
\]
The vector $\widehat x_1$ is nonzero, and $u_t=0$ for every $t\geq2$.
Thus the additional coordinate contributes nothing to the squared distance from iteration two onward.
The positive scaling and the initial-distance normalization contribute only fixed factors, whose $t$th roots tend to one.
Taking $t$th roots therefore yields
\[
 \lim_{t\to\infty}
 \left(\frac{\|x_t-x_\star\|^2}{\|x_0-x_\star\|^2}\right)^{1/t}
 =\qopt(s).
\]
\end{proof}

\subsection{Uniform upper bound from a strict separator}
\label{app:uniform-upper-bound}

We combine the cycle inequalities with the NAG recurrence to bound the sum of the squared rescaled iterates.
We first derive the bound assuming that this sum is finite.
A finite induction then proves that every trajectory satisfies this assumption.

\begin{proposition}[Uniform rate bound from a strict cycle separator]
\label{prop:strict-separator-rate}
Fix $s>1$ and $q>q_{\mathrm{quad}}(s)$.
Suppose a nonnegative linear combination of finitely many cycle costs is strictly negative at every angle $\theta\in[0,\pi]$.
Then there is a finite $b(s,q)$ such that every NAG trajectory \eqref{eq:nag} with $\kappa=s^2$, on any finite-dimensional objective $f\in\Fclass{s^2}$ and with any initialization $x_{-1}=x_0$, satisfies
\[
 \|x_t-x_\star\|^2\leq b(s,q)q^t\|x_0-x_\star\|^2,
 \qquad t\geq0.
\]
\end{proposition}

\begin{proof}
For $q>q_{\mathrm{quad}}(s)$, the coefficient $d_s(z)$ is nonzero on $|z|=\sqrt q$.
Define its reciprocal
\begin{equation}
\label{eq:inverse-transfer}
 p_s(z)=\frac1{d_s(z)}
 =-\frac{s^2z^2}{(s^2-1)(sz-(s-1))^2}.
\end{equation}
Multiplying both vectors in \eqref{eq:nag-transfer} by $p_s(z)$ changes the normalization from $r_t=z^t$ to $\sigma_t=z^t$.
All inner products are multiplied by the same positive factor:
\begin{equation}
\label{eq:normalization-inner-product}
 \operatorname{Re}\bigl[p_s(z)z^i(\bar z^i-\bar z^j)\bigr]
 =|p_s(z)|^2c_{ij}(\theta).
\end{equation}
This change preserves the sign of every cycle separator.

Let $f$ be defined on $\reals^d$.
Translate the minimizer to zero and set $f(0)=0$.
Write
\[
 c=(s^2-1)^{-1},\qquad
 a=\frac{1-1/s}{\sqrt q}\in(0,1),\qquad
 h(y)=f(y)-\frac12\|y\|^2.
\]
The function $h$ is convex, has gradient Lipschitz constant $s^2-1$, and satisfies $\nabla h(0)=0$.
For every sequence, use the scaling $\bar u_t=q^{-t/2}u_t$.
The gradient bound becomes
\begin{equation}
\label{eq:scaled-uncertainty}
 \bigl\|q^{-t/2}\nabla h(q^{t/2}v)\bigr\|
 \leq(s^2-1)\|v\|,
 \qquad v\in\reals^d.
\end{equation}

For the NAG trajectory, set $\sigma_t=\nabla h(y_t)$.
Eliminating $x_t$ from the NAG equations gives, for $t\geq2$,
\[
 \bar y_t-2a\bar y_{t-1}+a^2\bar y_{t-2}
 =-2ca\bar\sigma_{t-1}+ca^2\bar\sigma_{t-2}.
\]
The standard initialization gives
\[
 \bar y_0=x_0,\qquad
 \bar y_1=a\left(1+\frac1{s+1}\right)x_0-2ca\bar\sigma_0.
\]
Expanding this recurrence yields, for $t\geq0$,
\[
 \bar y_t=e_t-c\sum_{j=1}^t(j+1)a^j\bar\sigma_{t-j},
 \qquad
 e_t=\left(1+\frac{t}{s+1}\right)a^t x_0.
\]
The sequence $e$ accounts for the initialization and solves the same recurrence with every $\sigma_t$ set to zero.

Let $\ell_2^d$ denote the space of square-summable $\reals^d$-valued sequences indexed by $\integers$, with norm
\[
 \|u\|_{\ell_2}^2=\sum_{t\in\integers}\|u_t\|^2.
\]
Sequences defined only for $t\geq0$ are set to zero for $t<0$.
To abbreviate the weighted sum in the recurrence, define the linear map
\[
 g(u)_t=-c\sum_{j=1}^{\infty}(j+1)a^j u_{t-j},
 \qquad u\in\ell_2^d.
\]
Shifting the indices of a sequence preserves its $\ell_2$ norm.
The triangle inequality and the identity $\sum_{j=1}^{\infty}(j+1)a^j=(1-a)^{-2}-1$ therefore give
\[
 \|g(u)\|_{\ell_2}\leq c_g\|u\|_{\ell_2},
 \qquad c_g=c\bigl((1-a)^{-2}-1\bigr).
\]
For a sequence that vanishes at negative indices, the sum defining $g(u)_t$ has only $t$ possibly nonzero terms when $t\geq0$, and $g(u)_t=0$ when $t<0$.
Thus the same formula defines $g(u)$ without assuming square summability.
We will prove a bound for every square-summable sequence $e$ that vanishes at negative indices.
Allowing arbitrary $e$ will let us apply the same bound to the modified sequence used in the final induction.
For any such sequence and any scalar $\tau\in[0,1]$, consider
\begin{equation}
\label{eq:rescaled-recursion}
 \begin{aligned}
 \bar y_t&=e_t+g(\bar\sigma)_t
 =e_t-c\sum_{j=1}^t(j+1)a^j\bar\sigma_{t-j},\\
 \bar\sigma_t&=\tau q^{-t/2}\nabla h(q^{t/2}\bar y_t),
 \qquad t\geq0.
 \end{aligned}
\end{equation}
The NAG trajectory corresponds to $\tau=1$ and the initialization sequence $e$ derived above.
These equations determine a unique sequence by induction on $t$.
Each value of $\bar y_t$ depends only on the previously determined gradients.

To express the cycle inequalities, write
\[
 f_\tau(y)=\frac12\|y\|^2+\tau h(y).
\]
As $\tau$ varies from zero to one, this function interpolates between the quadratic $\|y\|^2/2$ and $f$.
For every $\tau\in[0,1]$, this function belongs to $\Fclass{s^2}$ and has minimizer zero.
For arbitrary query points, write $\sigma_t=\tau\nabla h(y_t)$, consistently with \eqref{eq:rescaled-recursion} after rescaling.
Take $y_t=\sigma_t=0$ for $t<0$.
Applying \eqref{eq:interpolation-variables} to $f_\tau$ gives $\sigma=\sigma_t$ and $r=y_t-c\sigma_t$ at each query point $y_t$.
For each selected cycle $\gamma=(n_0,\ldots,n_\ell=n_0)$, cyclic monotonicity therefore gives
\begin{equation}
\label{eq:local-cycle-supply}
 a_{\gamma,t}
 =\sum_{j=0}^{\ell-1}
 \langle y_{t+n_j}-c\sigma_{t+n_j},
 \sigma_{t+n_j}-\sigma_{t+n_{j+1}}\rangle\geq0,
 \qquad t\in\integers.
\end{equation}
The sum vanishes for $t<-\max_j n_j$.
After rescaling,
\begin{equation}
\label{eq:scaled-local-supply}
 q^{-t}a_{\gamma,t}
 =\sum_{j=0}^{\ell-1}
 \left\langle q^{n_j/2}(\bar y_{t+n_j}-c\bar\sigma_{t+n_j}),
 q^{n_j/2}\bar\sigma_{t+n_j}
 -q^{n_{j+1}/2}\bar\sigma_{t+n_{j+1}}\right\rangle\geq0.
\end{equation}
As in alternative~2 of Lemma~\ref{lem:strict-alternative}, write $w_\gamma\geq0$ for the separator weights, with one weight for each selected cycle.
We collect terms according to the difference between their two iteration indices.
For example, replacing $t+n_j$ by $t$ pairs the first vector in \eqref{eq:scaled-local-supply} with $\bar\sigma_{t+n_{j+1}-n_j}$ in its second term.
The following finite sums record these index differences as powers of $z\neq0$:
\begin{equation}
\label{eq:cycle-laurent-multiplier}
 k_\gamma(z)=\sum_{j=0}^{\ell-1}
 \left(q^{n_j}-q^{(n_j+n_{j+1})/2}z^{n_{j+1}-n_j}\right),
 \qquad
 k(z)=\sum_\gamma w_\gamma k_\gamma(z)=\sum_j k_jz^j.
\end{equation}
Only finitely many real coefficients $k_j$ are nonzero.
In sums involving these coefficients, indices without stated limits range over $\integers$.
Set $c_k=\sum_j|k_j|$.
Since each shift preserves the norm, the triangle inequality gives
\[
 \left\|\left(\sum_j k_ju_{t+j}\right)_{t\in\integers}\right\|_{\ell_2}
 \leq c_k\|u\|_{\ell_2}.
\]
For $\bar y\in\ell_2^d$, the gradient bound \eqref{eq:scaled-uncertainty} gives $\bar\sigma\in\ell_2^d$.
Each sum of shifted inner products then converges absolutely by Cauchy--Schwarz.
Summing \eqref{eq:scaled-local-supply} and reindexing therefore gives
\begin{equation}
\label{eq:summed-cycle-supply}
 \sum_{t,j}k_j\langle\bar y_t-c\bar\sigma_t,\bar\sigma_{t+j}\rangle
 =\sum_\gamma w_\gamma\sum_{t\in\integers}q^{-t}a_{\gamma,t}\geq0.
\end{equation}

For $z=\sqrt q\,e^{i\theta}$, direct substitution gives
\begin{equation}
\label{eq:strict-cycle-fdi}
 \operatorname{Re}\!\left[\overline{p_s(z)}k(e^{i\theta})\right]
 =|p_s(z)|^2\sum_\gamma w_\gamma c_\gamma(\theta)
 \leq-c_0
\end{equation}
for some $c_0>0$.
Continuity supplies this margin on $[0,\pi]$, and real coefficients extend it to the full unit circle.
Since $p_s(\sqrt q\,e^{i\theta})=-c/(1-ae^{-i\theta})^2$, clearing the denominator gives
\[
 c\operatorname{Re}\bigl[k(e^{i\theta})(1-ae^{-i\theta})^2\bigr]
 \geq c_0|1-ae^{-i\theta}|^4,
 \qquad \theta\in\reals.
\]
We next turn this pointwise inequality into an inequality for sequences.
For any $u\in\ell_2^d$, define
\[
 v_t=u_t-2au_{t-1}+a^2u_{t-2}.
\]
We will show that
\[
 c\sum_{t,j}k_j\langle u_t,v_{t+j}\rangle
 \geq c_0\|v\|_{\ell_2}^2.
\]
First suppose that $u$ has finite support.
Multiplying the pointwise inequality by a squared norm shows that
\[
 \left\|\sum_tu_te^{-it\theta}\right\|^2
 \left(c\operatorname{Re}\bigl[k(e^{i\theta})(1-ae^{-i\theta})^2\bigr]
       -c_0|1-ae^{-i\theta}|^4\right)
\]
is nonnegative for every real $\theta$.
After expansion, this expression is a finite sum of terms proportional to $e^{ij\theta}$ with integer indices $j$.
Choose a positive integer $p$ larger than every $|j|$ occurring in this sum.
The finite geometric-sum identity
\[
 \frac1p\sum_{r=0}^{p-1}\exp(2\pi i r j/p)
 =\begin{cases}1,&j=0,\\0,&j\ne0,\end{cases}
 \qquad j\in\integers,\quad |j|<p,
\]
shows that averaging at $\theta=2\pi r/p$, for $r=0,\ldots,p-1$, keeps only the constant coefficient.
To identify that coefficient, note that $\sum_tv_te^{-it\theta}=(1-ae^{-i\theta})^2\sum_tu_te^{-it\theta}$.
In the product with $k(e^{i\theta})$, a term pairing $u_t$ with $v_r$ has factor $e^{i(t+j-r)\theta}$ and contributes to the constant coefficient exactly when $r=t+j$.
The average is therefore $c\sum_{t,j}k_j\langle u_t,v_{t+j}\rangle-c_0\|v\|_{\ell_2}^2$.
Every expression being averaged is nonnegative, so this proves the claimed sequence inequality for finite support.
Truncation extends the inequality to every square-summable $u$.
Indeed, the displayed bound for the weighted shifts and the inequality $\|v\|_{\ell_2}\leq(1+a)^2\|u\|_{\ell_2}$ make both quadratic expressions continuous.

Conversely, for any $v\in\ell_2^d$, the sequence
\[
 u_t=\sum_{j=0}^{\infty}(j+1)a^jv_{t-j}
\]
has norm at most $(1-a)^{-2}\|v\|_{\ell_2}$ by the triangle inequality.
It satisfies $v_t=u_t-2au_{t-1}+a^2u_{t-2}$ by cancellation of the coefficients.
The coefficient of $v_t$ is one, the coefficient of $v_{t-1}$ is $2a-2a=0$, and every later coefficient is $\bigl((j+1)-2j+(j-1)\bigr)a^j=0$ for $j\geq2$.
The same definition gives $g(v)_t-cv_t=-cu_t$.
Consequently,
\begin{equation}
\label{eq:cycle-convolution-coercivity}
 \sum_{t,j}k_j\langle g(v)_t-cv_t,v_{t+j}\rangle
 \leq-c_0\|v\|_{\ell_2}^2.
\end{equation}

Suppose temporarily that the sequence $\bar y$ generated by \eqref{eq:rescaled-recursion} is square summable for a given $e\in\ell_2^d$ and $\tau\in[0,1]$.
Substitute $\bar y=e+g(\bar\sigma)$ into \eqref{eq:summed-cycle-supply}.
The preceding inequality bounds the terms involving $g(\bar\sigma)-c\bar\sigma$.
Cauchy--Schwarz and the bound for the weighted shifts bound the remaining terms by $c_k\|e\|_{\ell_2}\|\bar\sigma\|_{\ell_2}$.
Thus
\[
 0\leq-c_0\|\bar\sigma\|_{\ell_2}^2
       +c_k\|e\|_{\ell_2}\|\bar\sigma\|_{\ell_2}.
\]
If $\|\bar\sigma\|_{\ell_2}>0$, divide by this norm and then use $\bar y=e+g(\bar\sigma)$.
If this norm is zero, the same bounds hold directly.
We obtain
\[
 \|\bar\sigma\|_{\ell_2}\leq\frac{c_k}{c_0}\|e\|_{\ell_2},
 \qquad
 \|\bar y\|_{\ell_2}\leq c_y\|e\|_{\ell_2},
 \qquad c_y=1+\frac{c_gc_k}{c_0}.
\]
Once the separator is fixed, these constants depend only on $s,q$.
In particular, they are independent of $\tau$, $e$, the objective, and the dimension.

We now remove this assumption by increasing $\tau$ from zero to one in finitely many steps.
Choose a positive integer $m>2c_yc_g(s^2-1)$ and set $\tau_j=j/m$ for $j=0,\ldots,m$.
We prove by induction on $j$ that, for every square-summable $e$ vanishing at negative indices, the sequence $\bar y$ generated by \eqref{eq:rescaled-recursion} at $\tau=\tau_j$ satisfies
\[
 \sum_{t=0}^{\infty}\|\bar y_t\|^2<\infty.
\]

\emph{Base case.}
At $\tau_0=0$, we have $\bar\sigma=0$ and $\bar y=e$, so the claim holds.

\emph{Induction step.}
Fix $1\leq j\leq m$ and assume the claim holds at $\tau_{j-1}$.
Fix any allowed $e$ and let $\bar y$ be the sequence at $\tau_j$.
We will bound $\sum_{t=0}^n\|\bar y_t\|^2$ independently of $n$.
For each integer $n\geq0$, define
\[
 u_t^{(n)}=
 \begin{cases}
  q^{-t/2}\nabla h(q^{t/2}\bar y_t),&0\leq t\leq n,\\
  0,&\text{otherwise},
 \end{cases}
 \qquad
 e^{(n)}=e+\frac1m g(u^{(n)}).
\]
Only finitely many terms of $u^{(n)}$ are nonzero, so the bound on $g$ makes $e^{(n)}$ square summable.
It also vanishes at negative indices.
The induction hypothesis therefore gives a square-summable sequence $\bar y^{(n)}$ at $\tau_{j-1}$ with $e^{(n)}$ in place of $e$.

The modification of $e$ makes the first $n+1$ terms of these two sequences agree.
Indeed, $g(u)_t$ uses only terms with indices less than $t$, so for $0\leq t\leq n$,
\[
 \bar y_t=e_t+\tau_jg(u^{(n)})_t
 =e_t^{(n)}+\tau_{j-1}g(u^{(n)})_t.
\]
The last equality uses $\tau_j=\tau_{j-1}+1/m$.
These are exactly the recurrence equations at $\tau_{j-1}$ with $e^{(n)}$ through time $n$.
They determine each term from the previous ones, so $\bar y_t^{(n)}=\bar y_t$ for $0\leq t\leq n$.

Since $\bar y^{(n)}$ is square summable, the preceding norm estimate applies to it.
Combining that estimate with the bound on $g$ and \eqref{eq:scaled-uncertainty} gives
\[
 \begin{aligned}
 \left(\sum_{t=0}^n\|\bar y_t\|^2\right)^{1/2}
 &\leq\|\bar y^{(n)}\|_{\ell_2}
 \leq c_y\|e^{(n)}\|_{\ell_2}\\
 &\leq c_y\|e\|_{\ell_2}
 +\frac{c_yc_g(s^2-1)}m
  \left(\sum_{t=0}^n\|\bar y_t\|^2\right)^{1/2}.
 \end{aligned}
\]
The coefficient multiplying the partial norm on the right is less than $1/2$.
Moving this term to the left gives, for every $n$,
\[
 \left(\sum_{t=0}^n\|\bar y_t\|^2\right)^{1/2}
 \leq2c_y\|e\|_{\ell_2}.
\]
Every partial sum of squared norms is therefore at most $4c_y^2\|e\|_{\ell_2}^2$.
These partial sums are nondecreasing and bounded, so their limit is finite.
This proves the claim at $\tau_j$ and completes the induction.

At $\tau_m=1$, the sequence $\bar y$ is square summable, and the gradient bound gives $\bar\sigma\in\ell_2^d$.
The earlier estimates for both sequences now apply and yield
\begin{equation}
\label{eq:scaled-loop-gain}
 \|\bar y\|_{\ell_2}^2+\|\bar\sigma\|_{\ell_2}^2
 \leq\left(c_y^2+\frac{c_k^2}{c_0^2}\right)\|e\|_{\ell_2}^2.
\end{equation}

Return to the NAG initialization sequence $e$ derived above.
Since $a<1$, the squared norm $\|e\|_{\ell_2}^2$ is a finite multiple of $\|x_0\|^2$ depending only on $s,q$.
Finally, $s^2x_{t+1}=(s^2-1)y_t-\sigma_t$ gives
\[
 q^{-(t+1)}\|x_{t+1}\|^2
 \leq\frac{(s^2-1)^2+1}{s^4q}
 \left(\|\bar y\|_{\ell_2}^2+\|\bar\sigma\|_{\ell_2}^2\right).
\]
Combining this with \eqref{eq:scaled-loop-gain} gives a finite constant $b_1(s,q)$ such that
\[
 \|x_t\|^2\leq b_1(s,q)q^t\|x_0\|^2,
 \qquad t\geq1.
\]
Set $b(s,q)=\max\{1,b_1(s,q)\}$.
Since $b(s,q)\geq b_1(s,q)$, the bound still holds for every $t\geq1$.
At $t=0$, we have $q^0=1$ and $b(s,q)\geq1$, so the bound reduces to $\|x_0\|^2\leq b(s,q)\|x_0\|^2$, which also holds.

The proof used coordinates centered at the minimizer.
Thus the vectors denoted above by $x_t$ and $x_0$ are $x_t-x_\star$ and $x_0-x_\star$ in the original coordinates.
The bound in those coordinates is therefore
\[
 \|x_t-x_\star\|^2\leq b(s,q)q^t\|x_0-x_\star\|^2,
 \qquad t\geq0,
\]
as claimed.
\end{proof}

\subsection{Assembly of the lower and upper bounds}

\begin{remark}[The case $\kappa=1$]
For $f\in\mathcal F_{1,1}$, strong convexity and smoothness at the minimizer give
\[
 \frac12\|x-x_\star\|^2
 \leq f(x)-f(x_\star)
 \leq\frac12\|x-x_\star\|^2.
\]
Hence $f(x)=f(x_\star)+\frac12\|x-x_\star\|^2$ and $\nabla f(x)=x-x_\star$ for every $x$.
When $\kappa=1$, the momentum coefficient is $\beta_1=0$ and the step size is one.
Thus $x_{t+1}=x_t-\nabla f(x_t)=x_\star$ for every $t\geq0$.
Both asymptotic factors in \eqref{eq:rho-fix}--\eqref{eq:rho-hor} are therefore zero.
\end{remark}

\begin{proof}[Proof of Theorem~\ref{thm:main}]
Fix $\kappa>1$ and set $s=\sqrt\kappa$.
Proposition~\ref{prop:fixed-domain-rate} gives
\[
 q_{\mathrm{quad}}(s)\leq\qopt(s)\leq1-1/s<1.
\]
We first prove $\qopt(s)\leq\rhofix(s^2)$ by constructing one fixed objective and initial point that attain $\qopt(s)$.

If $\qopt(s)=q_{\mathrm{quad}}(s)$, the scalar quadratic $f(x)=x^2/2$ attains this factor from any nonzero standard initialization by \eqref{eq:q0}.

Suppose $\qopt(s)>q_{\mathrm{quad}}(s)$.
Lemma~\ref{lem:finite-optimal-support} supplies a pair $(\nu_\star,\varphi_0^\star)$ satisfying \eqref{eq:star-feasible} at $q=\qopt(s)$, with $\nu_\star$ supported on finitely many angles.
The associated contracting and rotating trajectory uses at most two real coordinates per support angle, so it is finite-dimensional.
By Lemma~\ref{lem:stars}, the corresponding interpolation values can be chosen as $\varphi_t=\qopt(s)^t\varphi_0^\star$.
Lemma~\ref{lem:initialization-splice} then gives a single objective in $\Fclass{s^2}$ and an initialization $x_{-1}=x_0\neq x_\star$ whose NAG trajectory has asymptotic squared-distance factor $\qopt(s)$.

Thus in both cases,
\[
 \qopt(s)\leq\rhofix(s^2)\leq\rhohor(s^2).
\]

For the upper bound, fix any $q\in(\qopt(s),1)$.
Since $\qopt(s)=\max\mathcal Q(s)$ by \eqref{eq:qcyc}, the interpolation conditions \eqref{eq:star-feasible} are infeasible at this $q$.
By Lemma~\ref{lem:stars}, no probability measure on $[0,\pi]$ satisfies all finite-cycle inequalities.
The equivalence following Lemma~\ref{lem:strict-alternative} rules out alternative~1, so alternative~2 gives a nonnegative combination of finitely many cycle costs that is strictly negative at every angle.

Since $q>q_{\mathrm{quad}}(s)$, Proposition~\ref{prop:strict-separator-rate} gives a finite $b(s,q)\geq1$ that depends only on $s,q$.
Taking the supremum over objectives, dimensions, and initial points yields
\[
 \sup_{d\geq1}\sup_{f\in\Fclass{s^2}(\reals^d)}
 \sup_{x_0\neq x_\star}R_t(f,x_0)
 \leq b(s,q)q^t,
 \qquad t\geq0.
\]
Taking $t$th roots and the limit superior, with $q$ fixed, gives
\[
 \rhohor(s^2)\leq\lim_{t\to\infty}b(s,q)^{1/t}q=q.
\]

This holds for every $q\in(\qopt(s),1)$, so letting $q\downarrow\qopt(s)$ gives $\rhohor(s^2)\leq\qopt(s)$.
The constant $b(s,q)$ need not remain bounded as $q\downarrow\qopt(s)$, since the limit in $t$ was taken first.
Combining this upper bound with the lower bound proves \eqref{eq:main} for $\kappa>1$.

The bounds in \eqref{eq:qcyc} follow from the quadratic example and the inequality comparing iterations $0$ and $2$ in Lemma~\ref{lem:closed-rate-problem}.
Substituting $q_{\mathrm{quad}}(s)=(1-1/s)^2$ and $s=\sqrt\kappa$ gives \eqref{eq:classical-bracket}.
Both bounds tend to zero as $\kappa\downarrow1$ and to one as $\kappa\to\infty$, so $\qopt(\sqrt\kappa)$ has the same limits.
\end{proof}

\section{Proof of the sharp large-condition-number expansion}
\label{app:large-kappa}

\begin{proof}[Proof of Corollary~\ref{cor:large-kappa}]
Write $s=\sqrt\kappa$.
The quadratic lower bound in \eqref{eq:classical-bracket} already implies $\qopt(s)\to1$, so it suffices to prove
\begin{equation}
\label{eq:log-asymptotic}
 \lim_{s\to\infty}-s\log\qopt(s)=\frac43.
\end{equation}
We first show that, for each fixed $0<\lambda<4/3$ and all sufficiently large $s$, no $q\in[e^{-\lambda/s},1)$ is feasible for \eqref{eq:star-feasible}.
We then construct a feasible measure at $q=e^{-2\alpha/s}$ for each fixed $2/3<\alpha<3/4$ and all sufficiently large $s$.
These give upper and lower bounds on $\qopt(s)$, respectively, and the inequalities reverse for $-s\log\qopt(s)$.

\paragraph{Upper bound.}
Fix $q\in(q_{\mathrm{quad}}(s),1)$ and a pair $(\nu,\varphi_0)$ satisfying \eqref{eq:star-feasible}.
Because $d_s(z)$ is bounded on $|z|=\sqrt q<1$, the terms in \eqref{eq:edge-cost} defining $c_{t0}(\theta)$ tend to zero uniformly in $\theta$ as $t\to\infty$.
Hence $\bar c_{t0}(\nu)\to0$, so the lower inequality in \eqref{eq:star-feasible} gives $\varphi_0\geq0$ in the limit.
The upper inequality with $t=1$ then gives $\bar c_{01}(\nu)\geq0$.
Put $r=\sqrt q$ and $x=2r\cos\theta\in[-2r,2r]$.
Equations \eqref{eq:nag-transfer} and \eqref{eq:edge-cost} give
\[
 c_{01}(\theta)=-\frac{s^2-1}{2s^2q^2}\psi_s(q,x),
\]
where
\begin{equation}
\label{eq:lag-one-quadratic}
 \psi_s(q,x)=(s-1)^2x^2-q(qs^2+3s^2-4s+1)x
              +6q^2s^2-4q^2s-2qs^2+4qs-2q.
\end{equation}
Its unconstrained minimizer and value at the right endpoint are
\begin{equation}
\label{eq:lag-one-vertex}
 x_{\rm v}=\frac{q(qs^2+3s^2-4s+1)}{2(s-1)^2},
 \qquad
 \psi_s(q,2r)=2q(1-r)(sr-s+1)^2>0.
\end{equation}
If $x_{\rm v}\geq2r$, the minimum of $\psi_s(q,x)$ over $[-2r,2r]$ is its value at the right endpoint, which is positive by \eqref{eq:lag-one-vertex}.
Hence $c_{01}(\theta)<0$ for every angle, contradicting $\bar c_{01}(\nu)\geq0$.
Fix $0<\lambda<4/3$.
At $q=e^{-\lambda/s}$,
\[
 s(2\sqrt q-x_{\rm v})\longrightarrow\frac{3\lambda}{2}-2<0.
\]
Since $3\lambda/2-2<0$, the displayed limit implies $x_{\rm v}>2r$ at $q=e^{-\lambda/s}$ for all sufficiently large $s$.
To cover every larger $q<1$, write the condition $x_{\rm v}\geq2r$ as
\[
 r(r^2s^2+3s^2-4s+1)-4(s-1)^2\geq0.
\]
Its left-hand side is strictly increasing in $r$ for $s>1$, so the condition also holds for every $q\in[e^{-\lambda/s},1)$.
Since $e^{-\lambda/s}>q_{\mathrm{quad}}(s)$ for all sufficiently large $s$, the preceding argument rules out feasibility throughout this interval.
Therefore $\qopt(s)\leq e^{-\lambda/s}$ and $-s\log\qopt(s)\geq\lambda$.
Taking the limit inferior as $s\to\infty$ and then letting $\lambda\uparrow4/3$ gives
\[
 \liminf_{s\to\infty}-s\log\qopt(s)\geq\frac43.
\]
The same asymptotic upper bound follows from Gannot's result~\cite[Proposition~2.13]{gannot2022frequency}.

\paragraph{Lower bound.}
Fix $2/3<\alpha<3/4$ and set $r=e^{-\alpha/s}$ and $q=r^2$.
We seek a probability measure $\nu$ on $[0,\pi]$ such that $\bar c_{t0}(\nu)\geq0$ and $\bar c_{0t}(\nu)\geq0$ for every integer $t\geq1$.
Then \eqref{eq:star-feasible} holds with $\varphi_0=0$.
For $0<h<r$, let $\nu$ be the probability measure on $[0,\pi]$ with density
\begin{equation}
\label{eq:rational-circle-measure}
 \frac{(1-h^2)^5}{\pi(1+4h^2+h^4)|1-he^{i\theta}|^6}.
\end{equation}

This density gives explicit formulas for the two integrals needed to check feasibility.
Write $z=re^{i\theta}$ and $m_s=\int\operatorname{Re}\overline{d_s(z)}\,d\nu$.
Equation \eqref{eq:edge-cost} gives
\begin{align*}
 \bar c_{t0}(\nu)&=q^t m_s-\int\operatorname{Re}\left[\overline{d_s(z)}z^t\right]\,d\nu,\\
 \bar c_{0t}(\nu)&=m_s-\int\operatorname{Re}\left[\overline{d_s(z)}\bar z^t\right]\,d\nu.
\end{align*}
The subscripts $1$ and $2$ below correspond to the integrals containing $z^t$ and $\bar z^t$ in these two costs, respectively.
To evaluate them, use the absolutely convergent series
\[
 (1-he^{i\theta})^{-3}
 =\sum_{k=0}^{\infty}\binom{k+2}{2}h^ke^{ik\theta}.
\]
Substitute this series and its complex conjugate into the density, and use \eqref{eq:nag-transfer}.
After taking real parts, all nonconstant terms integrate to zero because $\int_0^\pi\cos(j\theta)\,d\theta=0$ for every nonzero integer $j$.
With
\begin{align*}
 \delta={}&s^2r^2(h^4+4h^2+1)
 -6sr(s-1)h(h^2+1)+6(s-1)^2h^2,\\
 e_1={}&sr-(s-1)h,
 &e_2={}&srh-(s-1),
\end{align*}
this gives
\[
 m_s=-\frac{(s^2-1)\delta}{s^2r^2(1+4h^2+h^4)}.
\]
For $\delta\ne0$ and every integer $t\geq0$, we obtain
\begin{align}
 \int\operatorname{Re}\left[\overline{d_s(z)}z^t\right]\,d\nu
 &=m_s(rh)^t\left(1+a_1t+b_1\binom{t}{2}\right),\notag\\
 \int\operatorname{Re}\left[\overline{d_s(z)}\bar z^t\right]\,d\nu
 &=m_s(rh)^t\left(1+a_2t+b_2\binom{t}{2}\right),
\label{eq:rational-circle-moments}
\end{align}
where the coefficients are
\begin{align*}
 a_1&=-\frac{(h^2-1)e_1
 [(s-1)h^3+srh^2-4(s-1)h+2sr]}{\delta},
 &b_1&=\frac{(h^2-1)^2e_1^2}{\delta},\\
 a_2&=-\frac{(h^2-1)e_2
 [srh^2-3(s-1)h+2sr]}{h\delta},
 &b_2&=\frac{(h^2-1)^2e_2^2}{h^2\delta}.
\end{align*}
To satisfy \eqref{eq:star-feasible} with $\varphi_0=0$, it suffices to have
\begin{equation}
\label{eq:coefficient-tests}
 m_s>0,\qquad
 a_1\leq r/h-1,\qquad
 a_2\leq1/(rh)-1.
\end{equation}
Indeed, $m_s>0$ implies $\delta<0$, and hence $b_1\leq0$ and $b_2\leq0$.
For $u\geq0$, $a\leq u$, and $b\leq0$, Bernoulli's inequality gives
\[
 1+at+b\binom{t}{2}\leq1+ut\leq(1+u)^t,
 \qquad t=0,1,\ldots.
\]
Apply this with $u=r/h-1$ and $u=1/(rh)-1$, respectively.
The first integral in \eqref{eq:rational-circle-moments} is then at most $m_s q^t$, and the second is at most $m_s$.
Thus both averaged costs are nonnegative for every integer $t\geq1$.

We now choose $h$ to satisfy \eqref{eq:coefficient-tests}.
Write $\epsilon=1/s$ and $h=e^{-\eta\epsilon}$, and define the two slacks
\[
 g_1=r/h-1-a_1,\qquad
 g_2=1/(rh)-1-a_2.
\]
We will choose $\eta$ so that these slacks are equal and positive.
Set
\[
 \gamma=\sqrt{\frac{3\alpha(1-\alpha)^2}{3\alpha-2}}.
\]
The function $g(\epsilon,\eta)=(g_1-g_2)/\epsilon$ extends smoothly to a neighborhood of $(0,\gamma)$, with
\[
 g(0,\eta)=-2\,
 \frac{3\alpha(1-\alpha)^2-(3\alpha-2)\eta^2}
 {3\alpha^2-6\alpha-\eta^2+3}.
\]
Our choice of $\gamma$ gives $g(0,\gamma)=0$, and
\[
 \partial_\eta g(0,\gamma)
 =\frac{4(3\alpha-2)\gamma}
 {3\alpha^2-6\alpha-\gamma^2+3}\ne0.
\]
The denominator equals $6(\alpha-1)^3/(3\alpha-2)<0$.
The implicit function theorem therefore gives a smooth choice $\gamma_s=\gamma+O(s^{-1})$ for which $g_1=g_2$ when $\eta=\gamma_s$.
For this choice, Taylor expansion gives
\begin{equation}
\label{eq:circle-gap-expansions}
 g_1=g_2=\frac{\alpha}{2(3\alpha-2)}s^{-2}+O(s^{-3}).
\end{equation}
The same substitution gives
\[
 m_s\longrightarrow
 -\frac{3\alpha^2-6\alpha-\gamma^2+3}{3}
 =\frac{2(1-\alpha)^3}{3\alpha-2}>0.
\]
Thus $m_s$ and both slacks are positive for all sufficiently large $s$.
Also, $\gamma^2-\alpha^2=\alpha(3-4\alpha)/(3\alpha-2)>0$, so $\gamma>\alpha$.
Consequently, $h=e^{-\gamma_s/s}$ satisfies $0<h<r<1$ for all sufficiently large $s$.
Since $\alpha<1$, we have $r=e^{-\alpha/s}>1-\alpha/s>1-1/s$, so $q>q_{\mathrm{quad}}(s)$.
The measure $\nu$ therefore satisfies \eqref{eq:star-feasible} with $\varphi_0=0$ at $q=e^{-2\alpha/s}$ for all sufficiently large $s$.
Hence $\qopt(s)\geq e^{-2\alpha/s}$, which gives
\[
 \limsup_{s\to\infty}-s\log\qopt(s)\leq2\alpha.
\]
Letting $\alpha\downarrow2/3$ and combining with the upper bound proves \eqref{eq:log-asymptotic}.
Since $(-\log u)/(1-u)\to1$ as $u\uparrow1$, this is equivalent to \eqref{eq:large-kappa}.
The square-root expansion follows from $\sqrt{1-v}=1-v/2+o(v)$.
\end{proof}

\section{Quadratic forms and the KYP lemma}
\label{app:kyp-energy}

The Kalman--Yakubovich--Popov (KYP) lemma expresses a frequency-domain inequality as a linear matrix inequality in a real symmetric matrix $P$.
We apply it to the finite combination of interpolation inequalities used in Proposition~\ref{prop:strict-separator-rate}.
This representation is not needed for Theorem~\ref{thm:main}.

The matrix $P$ defines a quadratic form satisfying an inequality at each iteration.
It defines a squared norm if $P$ is positive definite, which the KYP lemma does not guarantee.

We also construct an explicit positive quadratic form for each oscillatory component of the lower-bound trajectories.
After removing the geometric contraction, its value remains constant along the trajectory.
In suitable coordinates, the resulting motion is a rotation.

\subsection{From the cycle inequalities to a matrix inequality}

Fix $s>1$ and $q\in(\qopt(s),1)$.
As in Appendix~\ref{app:uniform-upper-bound}, translate the minimizer to zero and set $\bar u_t=q^{-t/2}u_t$.
Use the finite sum $k(z)$ and constants $c_0,c_k,c_g$ from the proof of Proposition~\ref{prop:strict-separator-rate}, and put $c=(s^2-1)^{-1}$ and $a=(1-1/s)/\sqrt q\in(0,1)$.
For this appendix, write $\widehat g(z)=p_s(\sqrt q\,z)+c$.
On the unit circle, the series defining $g$ gives
\[
 \widehat g(z)=-c\sum_{j=1}^{\infty}(j+1)a^jz^{-j}.
\]
The absolute sums of the coefficients give $|\widehat g(z)|\leq c_g$ and $|k(z)|\leq c_k$ on the unit circle.
Define the Hermitian matrix
\begin{equation}
\label{eq:explicit-cycle-multiplier}
 \Pi(z)=
 \begin{bmatrix}
  0&k(z)/2\\
  \overline{k(z)}/2&-c\operatorname{Re}k(z)
 \end{bmatrix}\otimes I_d.
\end{equation}
The scalar inequality \eqref{eq:strict-cycle-fdi} gives
$[\widehat g(z)I_d;I_d]^*\Pi(z)[\widehat g(z)I_d;I_d]\preceq-c_0I_d$ for $|z|=1$, with $c_0>0$ as in that inequality.
For each cycle, shift the indices in \eqref{eq:scaled-local-supply} so its largest index becomes the current time.
Let $\omega$ collect the resulting finitely many delayed variables, with the factors $q^{n_j/2}$ appearing in that equation.
Let the real symmetric matrix $M$ represent their sum with the nonnegative coefficients $w_\gamma$ of the separator.
The common shift of each cycle cancels on the unit circle, giving $\Pi=\omega^*M\omega$.
For the actual query points and transformed gradients, with $\bar y_t=\bar\sigma_t=0$ for $t<0$, the represented sum is nonnegative at each iteration.
As in \eqref{eq:rescaled-recursion}, write $\bar y=e+g(\bar\sigma)$, or equivalently
\[
 v=\begin{bmatrix}\bar\sigma\\e\end{bmatrix},
 \qquad T(z)=\begin{bmatrix}\widehat g(z)I_d&I_d\\I_d&0\end{bmatrix},
 \qquad \begin{bmatrix}\bar y\\\bar\sigma\end{bmatrix}=Tv.
\]

For $u,b\in\mathbb C^d$, put $v=[u;b]$, $X=\|u\|$, and $E=\|b\|$.
The preceding bounds and Cauchy--Schwarz give, for $|z|=1$,
\[
 v^*T^*\Pi Tv\leq-c_0X^2+c_kXE,
 \qquad \|Tv\|^2\leq(2c_g^2+1)X^2+2E^2.
\]
Choose
\[
 \alpha\doteq\frac{c_0}{4(2c_g^2+1)},\qquad
 b_0\doteq\frac{c_k^2}{c_0}+2\alpha+1.
\]
Completing the square yields
\[
 v^*T^*\Pi Tv+\alpha\|Tv\|^2-b_0E^2
 \leq-\frac{c_0}{2}X^2-E^2-\frac{(c_0X/2-c_kE)^2}{c_0}.
\]
The right-hand side is strictly negative for $v\ne0$, so
\begin{equation}
\label{eq:performance-fdi}
 T^*\Pi T+\alpha T^*T
 \prec b_0\begin{bmatrix}0&0\\0&I\end{bmatrix}.
\end{equation}
These positive constants depend only on $s,q$ and the selected cycles and coefficients.

To express \eqref{eq:performance-fdi} as one quadratic form, let $E_ev=e$ and define
\begin{equation}
\label{eq:kyp-augmented-transfer}
 H(z)=
 \begin{bmatrix}
  \omega(z)T(z)\\T(z)\\E_e
 \end{bmatrix},
 \qquad
 \widehat M=\operatorname{diag}(M,\alpha I,-b_0 I).
\end{equation}
Writing $w=\omega Tv$, the combined output is
\[
 o=Hv=\begin{bmatrix}w\\\bar y\\\bar\sigma\\e\end{bmatrix}.
\]
Thus \eqref{eq:performance-fdi} is
\[
 H(e^{i\theta})^*\widehat M H(e^{i\theta})\prec0,
 \qquad -\pi\leq\theta\leq\pi.
\]

The map $H$ can be evaluated recursively using the rescaled NAG recurrence and the delayed variables in $\omega$.
Choose a real linear recurrence with the smallest possible state dimension, called a \emph{minimal realization}:
\begin{equation}
\label{eq:kyp-realization}
 \xi_{t+1}=A\xi_t+Bv_t,
 \qquad o_t=C\xi_t+Dv_t.
\end{equation}
With $\xi_0=0$, these equations give $o=Hv$.
Minimality ensures that $(A,B)$ is controllable: every state can be reached from zero by a finite sequence of inputs.
The only possible poles of $H$ are $0$ and $a$, so the eigenvalues of $A$ lie strictly inside the unit circle.
In particular, $\det(zI-A)\ne0$ for $|z|=1$.
These are the hypotheses of the KYP lemma used below.

\begin{proposition}[Quadratic inequality from the KYP lemma]
\label{prop:kyp-time-domain}
There exist a real symmetric matrix $P$ and $\delta>0$ such that
\begin{equation}
\label{eq:kyp-lmi}
 K(P)\doteq
 \begin{bmatrix}
  A^TPA-P&A^TPB\\B^TPA&B^TPB
 \end{bmatrix}
 +\begin{bmatrix}C^T\\D^T\end{bmatrix}
  \widehat M
  \begin{bmatrix}C&D\end{bmatrix}
 \preceq-\delta I.
\end{equation}
Along every trajectory of \eqref{eq:kyp-realization}, the quadratic form $V_P(\xi)\doteq\xi^TP\xi$ satisfies
\begin{equation}
\label{eq:kyp-one-step-balance}
 V_P(\xi_{t+1})-V_P(\xi_t)+o_t^T\widehat M o_t
 \leq-\delta(\|\xi_t\|^2+\|v_t\|^2).
\end{equation}
\end{proposition}

\begin{proof}
The strict discrete KYP equivalence~\citep{rantzer1996kyp}, also used in \citet[equation~(21) and the proof of Theorem~4.5]{zhang2019noncausal}, gives a real symmetric $P$ such that $K(P)\prec0$.
Thus $K(P)\preceq-\delta I$ for some $\delta>0$.
For one step of \eqref{eq:kyp-realization}, expansion gives the identity
\begin{equation}
\label{eq:kyp-exact-balance}
 V_P(\xi_{t+1})-V_P(\xi_t)+o_t^T\widehat M o_t
 =\begin{bmatrix}\xi_t\\v_t\end{bmatrix}^{\!T}
 K(P)
 \begin{bmatrix}\xi_t\\v_t\end{bmatrix}.
\end{equation}
Applying \eqref{eq:kyp-lmi} proves \eqref{eq:kyp-one-step-balance}.
\end{proof}

For a rescaled NAG trajectory, \eqref{eq:scaled-local-supply} gives $w_t^TMw_t\geq0$ for every $t\geq0$.
Expanding the output term in \eqref{eq:kyp-one-step-balance} and summing gives, for every integer $N\geq0$,
\begin{equation}
\label{eq:kyp-energy-balance}
 \begin{split}
 V_P(\xi_{N+1})
 &+\sum_{t=0}^N\left[
   w_t^TMw_t
   +\alpha(\|\bar y_t\|^2+\|\bar\sigma_t\|^2)
   +\delta(\|\xi_t\|^2+\|v_t\|^2)
  \right]\\
 &\leq V_P(\xi_0)+b_0\sum_{t=0}^N\|e_t\|^2.
 \end{split}
\end{equation}

This has the form of a dissipation inequality~\citep{willems1972dissipative}, but $V_P$ need not be nonnegative.
For the standard NAG initialization, $\xi_0=0$ and $e_t=(1+t/(s+1))a^tx_0$.
Discarding $V_P(\xi_{N+1})$ at a finite $N$ requires a separate proof of its nonnegativity.

For each fixed $q>\qopt(s)$, the proof of Proposition~\ref{prop:strict-separator-rate} already establishes $\bar\sigma,e\in\ell_2^d$ for the actual trajectory.
Since the eigenvalues of $A$ have modulus less than one, $\sum_{j\geq0}\|A^jB\|<\infty$.
The formula
\[
 \xi_t=\sum_{j=0}^{t-1}A^jBv_{t-1-j},\qquad t\geq1,
\]
and the triangle inequality therefore give $\xi\in\ell_2$.
Consequently $\xi_t\to0$ and $V_P(\xi_t)\to0$ for every fixed symmetric $P$.
One may thus let $N\to\infty$ in \eqref{eq:kyp-energy-balance} without assuming $P\succeq0$.
This uses the square summability proved in Appendix~\ref{app:uniform-upper-bound}; it is not an independent proof of convergence.

\subsection{When the quadratic form defines a squared norm}

The cited KYP lemma asserts only $P=P^T$; it does \emph{not} assert $P\succeq0$.
For example, consider the stable scalar system $u_{t+1}=v_t$, $o_t=u_t$, with quadratic output term $-o_t^2$.
On the unit circle, the relation between input and output is multiplication by $z^{-1}$, so the frequency-domain expression is $-|e^{-i\theta}|^2=-1$.
For a scalar $p$, the KYP matrix is
\[
 \begin{bmatrix}-p-1&0\\0&p\end{bmatrix},
\]
which is negative definite exactly when $-1<p<0$.
Every feasible $p$ is therefore negative, so the quadratic form cannot be a squared norm.

Interpolation-based Lyapunov analyses, including AutoLyap, impose lower-bound and decrease conditions on iterates, gradients, and function values satisfying the algorithm and interpolation inequalities~\citep{taylor2018lyapunov,upadhyaya2024automated,upadhyaya2025autolyap}.
Their potentials may combine quadratic terms in iterates and gradients with linear terms in function values.
Such a potential can be nonnegative on these quantities even when its quadratic coefficient matrix is indefinite.
A sufficient condition is a decomposition into a nonnegative quadratic form and a nonnegative combination of the interpolation residuals, with their signs chosen so that each residual is nonnegative.

We have not established such a lower bound for the KYP quadratic form above.
A Lyapunov convergence argument would also need to account for $e_t=(1+t/(s+1))a^tx_0$.
The following corollary assumes $P\succ0$, so the quadratic form is a squared norm for every state $\xi$.

\begin{corollary}[Coordinates for a positive-definite quadratic form]
\label{cor:kyp-positive-energy}
Suppose $P\succ0$ satisfies \eqref{eq:kyp-lmi}.
Set $R=P^{1/2}$ and $\eta_t=R\xi_t$, so $V_P(\xi_t)=\|\eta_t\|^2$.
Along a rescaled NAG trajectory satisfying the local inequalities \eqref{eq:scaled-local-supply},
\begin{equation}
\label{eq:kyp-euclidean-balance}
 \|\eta_{t+1}\|^2
 \leq\|\eta_t\|^2+b_0\|e_t\|^2
 -\delta(\|R^{-1}\eta_t\|^2+\|v_t\|^2).
\end{equation}
Whenever $e_t=0$, the squared norm cannot increase at that step.
\end{corollary}

\begin{proof}
Since $P=R^TR$, we have $V_P(\xi_t)=\|R\xi_t\|^2$.
Expand the output term in \eqref{eq:kyp-one-step-balance} and substitute $\xi_t=R^{-1}\eta_t$.
Dropping $w_t^TMw_t\geq0$ and $\alpha(\|\bar y_t\|^2+\|\bar\sigma_t\|^2)\geq0$ gives \eqref{eq:kyp-euclidean-balance}.
The same substitution in \eqref{eq:kyp-energy-balance} gives the corresponding finite-horizon bound.
\end{proof}

For fixed matrices $A,B,C,D,\widehat M$, one can seek a positive-definite $P$ satisfying
\begin{equation}
\label{eq:kyp-positive-feasibility}
 P\succeq\epsilon I,
 \qquad K(P)\preceq-\delta I,
\end{equation}
with variables $P=P^T$, $\epsilon>0$, and $\delta>0$.
We do not claim generic feasibility of \eqref{eq:kyp-positive-feasibility}.

Removing the time scaling by setting $z_t=q^{t/2}\xi_t$ gives $V_P(z_t)=q^t\|\eta_t\|^2$.
At any step for which $e_t=0$, Corollary~\ref{cor:kyp-positive-energy} therefore gives
\[
 V_P(z_{t+1})\leq q\,V_P(z_t).
\]

\subsection{A quadratic invariant for one oscillatory component}

Fix $0<q<1$ and $0<\theta<\pi$.
A real coordinate of a component with geometric factor $\sqrt q\,e^{i\theta}$ satisfies
\begin{equation}
\label{eq:modal-recurrence}
 u_{t+1}=2\sqrt q\cos\theta\,u_t-q u_{t-1},
 \qquad t\geq0.
\end{equation}
For any real sequence $(u_t)_{t\geq-1}$ satisfying this recurrence, remove the contraction by setting $\bar u_t=q^{-t/2}u_t$.
Collect two successive values as
\[
 \xi_t=\begin{bmatrix}\bar u_t\\\bar u_{t-1}\end{bmatrix},
 \qquad
 A_\theta=\begin{bmatrix}2\cos\theta&-1\\1&0\end{bmatrix}.
\]
Dividing \eqref{eq:modal-recurrence} by $q^{(t+1)/2}$ gives $\xi_{t+1}=A_\theta\xi_t$.

\begin{proposition}[Rotation and quadratic invariant]
\label{prop:self-similar-energy}
Define
\[
 P_\theta=\begin{bmatrix}1&-\cos\theta\\-\cos\theta&1\end{bmatrix},
 \qquad
 R_\theta=\begin{bmatrix}1&-\cos\theta\\0&\sin\theta\end{bmatrix}.
\]
Then $P_\theta=R_\theta^TR_\theta\succ0$ and
\begin{equation}
\label{eq:modal-energy}
 A_\theta^TP_\theta A_\theta=P_\theta,
 \qquad
 E_t\doteq\xi_t^TP_\theta\xi_t=E_0.
\end{equation}
Moreover,
\begin{equation}
\label{eq:modal-rotation}
 R_\theta A_\theta R_\theta^{-1}
 =\begin{bmatrix}
   \cos\theta&-\sin\theta\\
   \sin\theta&\cos\theta
  \end{bmatrix}.
\end{equation}
\end{proposition}

\begin{proof}
Direct multiplication gives $P_\theta=R_\theta^TR_\theta$.
Since $\sin\theta>0$, the matrix $R_\theta$ is invertible, so $P_\theta\succ0$.
Direct multiplication also gives \eqref{eq:modal-rotation}.
The matrix on its right is orthogonal, which proves both identities in \eqref{eq:modal-energy}.
\end{proof}

In the coordinates $\eta_t=R_\theta\xi_t$, each step is a rotation by $\theta$, and $E_t=\|\eta_t\|^2$.
Thus the positive level sets of $\xi^TP_\theta\xi$ are ellipses that become circles under this change of coordinates.
For the state before rescaling,
\[
 z_t=\begin{bmatrix}u_t\\\sqrt q\,u_{t-1}\end{bmatrix}
       =q^{t/2}\xi_t,
\]
the quadratic form satisfies
\begin{equation}
\label{eq:modal-unscaled-energy}
 z_t^TP_\theta z_t=q^t E_0.
\end{equation}

At $\theta=0$ or $\theta=\pi$, the matrix $P_\theta$ is singular.
The general normalized solution is $\bar u_t=(a+bt)(\pm1)^t$, with $+$ at $\theta=0$ and $-$ at $\theta=\pi$.
The geometric components used in our construction have $b=0$, so one scalar coordinate suffices and $\bar u_t^2$ is constant.
When $b\neq0$, the normalized state is unbounded.
Hence the full two-dimensional recurrence has no positive-definite invariant quadratic form at these angles.

For $0<\theta<\pi$ and any $\widehat q>q$, set $\widehat\xi_t=\widehat q^{-t/2}z_t$.
Then $\widehat\xi_{t+1}=\sqrt{q/\widehat q}\,A_\theta\widehat\xi_t$, and \eqref{eq:modal-energy} gives
\[
 \widehat\xi_{t+1}^TP_\theta\widehat\xi_{t+1}
 =\frac{q}{\widehat q}\,
   \widehat\xi_t^TP_\theta\widehat\xi_t.
\]
The same matrix $P_\theta$ therefore gives a strict decrease whenever $\widehat\xi_t\neq0$.

For a finitely supported optimal angular measure with $\qopt(s)>q_{\mathrm{quad}}(s)$, place one $P_{\theta_j}$ block for each angle in $(0,\pi)$ and a scalar block $1$ for each angle equal to $0$ or $\pi$ on the diagonal.
Two successive values of one nonzero real coordinate determine each oscillatory component.
The resulting positive quadratic form is therefore defined on the finite-dimensional geometric trajectory and is constant after removing its contraction.
Lemma~\ref{lem:initialization-splice} realizes a scaled copy of these components from iteration one onward and adds a coordinate to $x_t$ that is zero for $t\geq2$, so it does not change the asymptotic factor.
The added query coordinate at iteration two need not be zero.
The invariant applies to successive values of the prescribed components, not to the full trajectory starting at $x_{-1}=x_0$.

\subsection{Equality conditions at the optimal factor}

For every $q\in(\qopt(s),1)$, the upper-bound proof uses a nonnegative combination of finitely many cycle costs that is strictly negative at every angle.
It does not establish a finite combination that is nonpositive at $q=\qopt(s)$.
We now describe the equality conditions under this additional assumption.

Throughout this subsection, set $q=\qopt(s)$ and assume $q>q_{\mathrm{quad}}(s)$.
Suppose finitely many cycles and nonnegative coefficients, not all zero, give
\[
 \phi_\star(\theta)=\sum_\gamma w_\gamma c_\gamma(s,q,\theta)\leq0,
 \qquad 0\leq\theta\leq\pi.
\]
Construct $\omega_\star$ and $M_\star$ from these delayed cycle comparisons as above, so $\Pi_\star=\omega_\star^*M_\star\omega_\star$.
At this rate, the linear map from the rescaled transformed gradients to those delayed variables is
\[
 H_\star(z)=\omega_\star(z)
 \begin{bmatrix}\widehat g(z)I_d\\I_d\end{bmatrix},
 \qquad \widehat g(z)=p_s(\sqrt q\,z)+c.
\]
Choose a minimal realization $(A,B,C,D)$ of $H_\star$.
It is stable because $q>q_{\mathrm{quad}}(s)$.
Define $K_\star(P)$ by the matrix expression in \eqref{eq:kyp-lmi} for this realization, with $M_\star$ in place of $\widehat M$.

\begin{proposition}[Support and equality conditions at the optimal factor]
\label{prop:kyp-equality-frequencies}
If $\nu_\star$ is feasible at $q$, then
\begin{equation}
\label{eq:contact-support}
 \operatorname{supp}\nu_\star
 \subseteq\{\theta\in[0,\pi]:\phi_\star(\theta)=0\}.
\end{equation}
Assume in addition that a symmetric $P=P^T$ satisfies $K_\star(P)\preceq0$.
For every $\theta\in[0,\pi]$ with $\phi_\star(\theta)=0$ and every nonzero complex input vector $a$, define
\[
 h_\theta(a)=
 \begin{bmatrix}
  (e^{i\theta}I-A)^{-1}Ba\\a
 \end{bmatrix}.
\]
Then
\begin{equation}
\label{eq:contact-kyp-kernel}
 K_\star(P)h_\theta(a)=0.
\end{equation}
\end{proposition}

\begin{proof}
Feasibility says that the average of every cycle cost is nonnegative, so the average of $\phi_\star$ is nonnegative as well.
If $\nu_\star$ has finite support, as in Lemma~\ref{lem:finite-optimal-support}, write $\nu_\star=\sum_{j=1}^m a_j\delta_{\theta_j}$ with $a_j>0$.
Then
\[
 0\leq\sum_{j=1}^m a_j\phi_\star(\theta_j)\leq0.
\]
Every summand is nonpositive, so $\phi_\star(\theta_j)=0$ for every $j$.
For a general feasible measure, we likewise have $0\leq\int\phi_\star\,d\nu_\star\leq0$.
If $\phi_\star$ were negative at a point in its support, continuity would give a neighborhood of positive measure on which it is bounded above by a negative number.
Since $\phi_\star\leq0$ everywhere, this would make its integral negative.
Thus \eqref{eq:contact-support} holds.

For the matrix condition, set $x_\theta=(e^{i\theta}I-A)^{-1}Ba$.
Then $Ax_\theta+Ba=e^{i\theta}x_\theta$ and $Cx_\theta+Da=H_\star(e^{i\theta})a$.
In the quadratic form defined by $K_\star(P)$, the state contribution is $(|e^{i\theta}|^2-1)x_\theta^*Px_\theta=0$.
The same cycle calculation as in \eqref{eq:strict-cycle-fdi} therefore gives
\[
 \begin{split}
 h_\theta(a)^*K_\star(P)h_\theta(a)
 &=a^*H_\star(e^{i\theta})^*M_\star H_\star(e^{i\theta})a\\
 &=|p_s(\sqrt q\,e^{i\theta})|^2
   \phi_\star(\theta)\|a\|^2=0.
 \end{split}
\]
Since $-K_\star(P)$ is positive semidefinite, a zero quadratic value implies that $h_\theta(a)$ lies in its kernel.
This proves \eqref{eq:contact-kyp-kernel}.
\end{proof}

Condition~\eqref{eq:contact-support} restricts the support to zeros of $\phi_\star$.
Equation~\eqref{eq:contact-kyp-kernel} identifies the corresponding vectors in the kernel of $K_\star(P)$.
If $P\succ0$, the quadratic form $V_P$ is a squared norm, but positive definiteness is not asserted here.
If $\phi_\star$ has a zero, a strict inequality $K_\star(P)\prec0$ cannot hold for the realization in this proposition.

The explicit invariant $P_\theta$ acts on two consecutive values of one component.
The matrix $P$ in this subsection acts on the state of the linear recurrence for $H_\star$, whose output consists of the variables selected by $\omega_\star$.

\section{Exact characterization of \texorpdfstring{$\qopt$}{q*} for \texorpdfstring{$\kappa = 10$}{kappa = 10}}
\label{app:kappa-ten}

We explain how the interpolation inequalities determine the algebraic data
for the planar example at $\kappa=10$. For completeness we also provide full polynomial coefficients further below. Throughout, we identify $\mathbb C$ with $\reals^2$, so
$\langle a,b\rangle=\operatorname{Re}(\bar a b)$ and $\|a\|=|a|$.
The momentum parameter is
\[
 \beta=\beta_{10}=\frac{\sqrt{10}-1}{\sqrt{10}+1}
 =\frac{11-2\sqrt{10}}9.
\]

\paragraph{Candidate query points, gradients, and function values.}
Let $z=\sqrt q\,e^{i\theta}$, where $0<q<1$, and put
$\chi=z+\bar z=2\sqrt q\cos\theta$.
Thus $q$ is the squared contraction factor, while the real coordinate $\chi$
encodes the rotation angle. Define the query and gradient amplitudes
\[
 e\doteq e(z)=1+\beta-\beta z^{-1},\qquad g\doteq g(z)=10(e-z).
\]
These expressions come directly from the momentum and gradient steps on a
geometric trajectory; in particular, $e-g/10=z$.
We first consider only its contracting part, for which we prescribe
\[
 \widehat y_k\doteq z^k e,\qquad
 \widehat g_k\doteq z^k g,\qquad
 \widehat f_k\doteq q^kF,\qquad k\geq0,
\]
where $F$ is an unknown scalar function value.
The hats denote quantities for the contracting part of the trajectory.
We assign them indices $k+1$ after adding a separate initial query.
We must still verify that the prescribed points, gradients, and function values come from one function in $\Fclass{10}$.

\paragraph{The inequalities producing the polynomial.}
For $\kappa=10$, the interpolation condition \eqref{eq:ssc-interpolation}
reads
\[
 f_i-f_j\geq\langle g_j,y_i-y_j\rangle
 +\frac1{20}\|g_i-g_j\|^2
 +\frac59\left\|y_i-y_j-\frac{g_i-g_j}{10}\right\|^2.
\]
To substitute these expressions into the interpolation condition, define
\[
 B\doteq B(q,z)=\frac{|g|^2}{20}+\frac{5q}{9},
\]
and, for each integer $m\geq1$, we have to compute the interpolation inequalities for $(i,j) = (0,m)$ and $(i,j) = (m,0)$. To this end observe that for the case $(i,j) = (0,m)$ we have $g_j = z^m g$ and $y_i - y_j = (1-z^m)e$ and for $(i,j) = (m,0)$ we have $g_j = g$ and $y_i - y_j = (z^m - 1)e$. Further for the first configuration we have $g_i - g_j = (1-z^m)g$ and $y_i - y_j = (1-z^m)e$, so that using $e-g/10=z$ we obtain  $y_i-y_j-\frac{g_i-g_j}{10} = (1-z^m) (e- \frac{g}{10}) = (1-z^m) z$, so that 
$$\frac1{20}\|g_i-g_j\|^2 +\frac59\left\|y_i-y_j-\frac{g_i-g_j}{10}\right\|^2 = |1-z^m|^2 B,$$
follows; we obtain the same quadratic term when expanding for the second case. With this we arrive at the right hand sides for the two interpolation inequalities 
\begin{align}
 C_m^+(q,z)&\doteq\langle z^m g,(1-z^m)e\rangle+|1-z^m|^2B,\notag\\
 C_m^-(q,z)&\doteq\langle g,(z^m-1)e\rangle+|1-z^m|^2B,
\label{eq:k10-interpolation-terms}
\end{align}
and we obtain the bounds
\[
 C_m^+(q,z) \leq (1-q^m)F \leq  -  C_m^-(q,z)
\]
Shifting both indices by $k$ turns $(0,m)$ and $(m,0)$ into
$(k,k+m)$ and $(k+m,k)$, respectively, and multiplies either inequality by
the positive factor $q^k$. The factor $1-q^m$ is then included in the
difference of function values:
\[
 \widehat f_k-\widehat f_{k+m}
 =q^kF-q^{k+m}F=q^k(1-q^m)F.
\]

We now select two different members of this family: the forward inequality
with $m=1$ and the reverse inequality with $m=18$. These are the bounds that
are active (hold with equality) at the optimal $\kappa = 10$ example; the
choice of eighteen is specific to this solution. Taking $m=1$ in the forward inequality and shifting
by $k=0,\ldots,17$ gives the consecutive comparisons
$(0,1),(1,2),\ldots,(17,18)$. Taking $m=18$ in the reverse inequality gives
$(18,0)$, which closes this chain. The particular inequalities we sum are
\begin{align}
 \widehat f_k-\widehat f_{k+1}&\geq q^k C_1^+(q,z),
       &&k=0,\ldots,17,\notag\\
 \widehat f_{18}-\widehat f_0&\geq C_{18}^-(q,z).
\label{eq:k10-active-cycle}
\end{align}
After summing we obtain for the left-hand sides the two terms:
\[
\begin{aligned}
 \sum_{k=0}^{17}(\widehat f_k-\widehat f_{k+1})
 &=(1-q)F\sum_{k=0}^{17}q^k=(1-q^{18})F,\\
 \widehat f_{18}-\widehat f_0&=-(1-q^{18})F.
\end{aligned}
\]
Thus $F$ is cancelled out from the summed inequalities.
The first eighteen comparisons go from index $0$ to index $18$, and the
last returns to index $0$, i.e., a cycle of interpolation comparisons.
Consequently, with $H_{18}(q)\doteq\sum_{k=0}^{17}q^k$, feasibility requires
\[
 W_{18}(q,z)\doteq H_{18}(q)C_1^+(q,z)+C_{18}^-(q,z)\leq0.
\]
Equivalently, the two active comparisons bound the same function value:
\[
 \frac{C_1^+(q,z)}{1-q}\leq F
 \leq-\frac{C_{18}^-(q,z)}{1-q^{18}}.
\]
When both are equalities, $W_{18}(q,z)=0$ and these bounds coincide.

To express $W_{18}$ using only the real quantities $q$ and $\chi$, introduce
the sums of conjugate powers
\[
 T_j\doteq z^j+\bar z^j,\qquad j\geq0.
\]
Since $z+\bar z=\chi$ and $z\bar z=q$, multiplying out gives
$\chi T_j=T_{j+1}+qT_{j-1}$ for $j\geq1$. Thus these sums satisfy
\[
 T_0=2,\qquad T_1=\chi,\qquad
 T_{j+1}=\chi T_j-qT_{j-1},\quad j\geq1.
\]
This recurrence expresses each $T_j$ as a polynomial in $(q,\chi)$; for
example, $T_2=\chi^2-2q$ and $T_3=\chi^3-3q\chi$. We use these polynomials to replace the complex powers in
\eqref{eq:k10-interpolation-terms}. In particular, the common quadratic
factor becomes
\[
 |1-z^m|^2=1+q^m-(z^m+\bar z^m)=1+q^m-T_m(q,\chi).
\]
For the inner-product terms, substitute $\bar z=q/z$ and expand into real
parts of powers of $z$, then use
\[
 \operatorname{Re}(z^j)=\frac{T_j(q,\chi)}2,\qquad
 \operatorname{Re}(z^{-j})=\frac{T_j(q,\chi)}{2q^j},\qquad j\geq1.
\]
Applying these replacements to $C_1^+(q,z)$ and $C_{18}^-(q,z)$ eliminates
$z$ and $\bar z$ from their combination $W_{18}(q,z)$. Simplifying the
resulting rational expression in $(q,\chi)$ and factoring gives the exact identity
\begin{equation}
\label{eq:k10-cycle-polynomial}
 W_{18}(q,z)=\frac{1+q-\chi}{9q}P(q,\chi),
\end{equation}
via $|z|^2 = q$ and $\chi = z + \bar z$.
Here $P$ is a polynomial with real coefficients of total degree $19$ in the two real variables $q,\chi$; in fact $P \in \mathbb Q(\sqrt{10})[q,\chi] \subseteq \mathbb R[q,\chi]$ and its complete coefficient table appears at the end of this section.
The factor multiplying $P$ is positive, since
$1+q-\chi=|1-z|^2\geq(1-\sqrt q)^2>0$.
Thus this cycle requires $P(q,\chi)\leq0$, and its active boundary is at equality, i.e., when $P(q,\chi)=0$; we call $P(q,\chi)$ the \emph{active-cycle polynomial}.

\paragraph{Selecting an exact candidate.}
We seek a value of $q$ for which the graph of $\chi\mapsto P(q,\chi)$ touches
the horizontal axis at a minimum strictly between the endpoints
$-2\sqrt q$ and $2\sqrt q$. Since $\chi=2\sqrt q\cos\theta$, this means that
the trajectory's rotation angle lies strictly between $0^\circ$ and $180^\circ$.
At such a minimum, both the polynomial and its derivative vanish. We therefore
look for candidate pairs satisfying
\begin{equation}
\label{eq:k10-tangency}
 P(q,\chi)=0,\qquad \partial_\chi P(q,\chi)=0.
\end{equation}
These are necessary conditions, not a guarantee that the graph has the desired
shape. The box below selects one candidate, denoted $(q_{10},\chi_\star)$;
we then establish that it gives the required minimum and, separately, the
optimal rate. In particular, the identification $q_{10}=\qopt(\sqrt{10})$ remains to be established by a
realizable trajectory and a matching upper bound. Specifically, our pair is the unique solution in the rational box
\begin{equation}
\label{eq:k10-root-box}
\begin{aligned}
 |q-0.5644232779728935119106424484621701711426|&<10^{-40},\\
 |\chi-1.4532011129657160483729421248197410663214|&<10^{-40}.
\end{aligned}
\end{equation}
The decimal centers here are exact rational numbers defining the box. Directed interval arithmetic certifies that the Jacobian of $(P,\partial_\chi P)$ is nonsingular throughout this box and that an interval Newton image lies strictly inside it, establishing existence and uniqueness.
The computation also establishes $q_{\mathrm{quad}}(\sqrt{10})<q_{10}<1-1/\sqrt{10}$ and
$\chi_\star^2<4q_{10}$. The latter places $\chi_\star$ strictly inside the
allowed interval $[-2\sqrt{q_{10}},2\sqrt{q_{10}}]$. Since
$\chi=2\sqrt q\cos\theta$, this corresponds to an angle strictly between
$0$ and $\pi$, rather than either real-mode endpoint. This establishes the
interiority of the candidate.

To show that the stationary zero actually touches zero rather than crosses
it, we use the exact polynomial factorization
\begin{equation}
\label{eq:k10-contact-factorization}
 P(q_{10},\chi)=(\chi-\chi_\star)^2R(\chi),
\end{equation}
where $R$ has degree $17$. Interval bounds on $64$ rational subintervals
certify $R(\chi)>17$ for $|\chi|\leq1.503$, an interval containing
$[-2\sqrt{q_{10}},2\sqrt{q_{10}}]$.
Hence $P(q_{10},\chi)$ is nonnegative throughout the allowed frequency
interval and vanishes only at $\chi_\star$: the curve genuinely touches zero
at this interior point; see Figure~\ref{fig:k10-tangency} for an illustration. This establishes the contact for the exact candidate
$q_{10}$, but does not yet show that $q_{10}=\qopt(\sqrt{10})$. The remaining
interpolation conditions and the matching upper bound below are essential.

\begin{figure}[tbp]
\centering
\includegraphics[width=\linewidth]{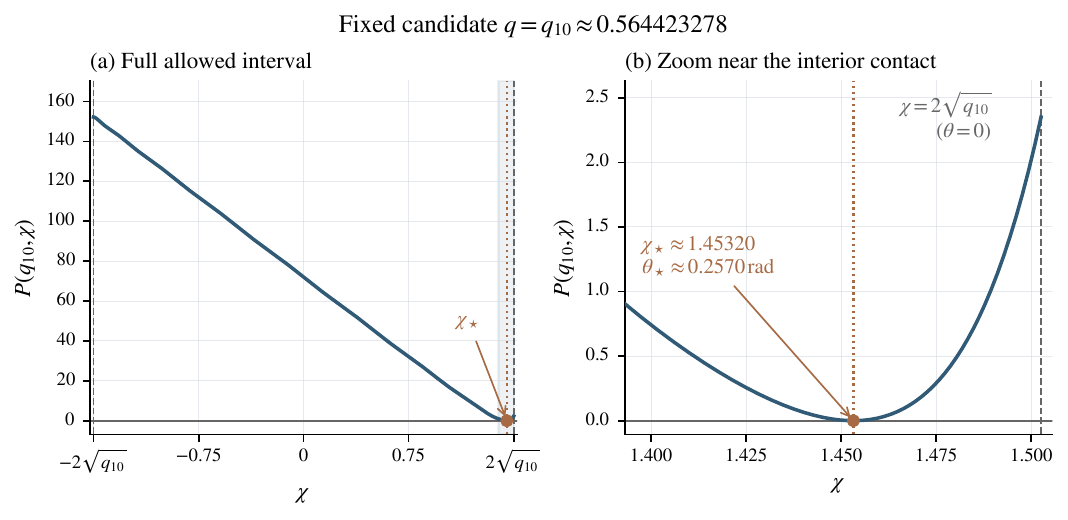}
\caption{Numerical illustration of $\chi\mapsto P(q_{10},\chi)$
at the (algebraic) candidate $q_{10}$.
Left: the full allowed interval $[-2\sqrt{q_{10}},2\sqrt{q_{10}}]$, with the
shaded region enlarged on the right.
Right: the contact at $\chi_\star$ is strictly inside the interval, away from
the real-mode endpoint $\chi=2\sqrt{q_{10}}$ ($\theta=0$). The marked angle is
$\theta_\star=\arccos(\chi_\star/(2\sqrt{q_{10}}))$.
Here $P$ comes from one active comparison cycle, whereas
Figures~\ref{fig:k100-tangency} and~\ref{fig:radical-tangency} plot the
weighted dual functions $\Phi_{100}$ and $\Phi_\dagger$ combining several
interpolation constraints. All three use the real frequency coordinate $\chi$.}
\label{fig:k10-tangency}
\end{figure}

\paragraph{The trajectory with the standard initialization realizing $q_{10}$.} We now show that $q_{10}$ is realizable, i.e., we construct a trajectory that contracts with a factor $q_{10}$ and thus $\qopt(\sqrt{10}) \geq q_{10}$.
Unlike the general construction in Lemma~\ref{lem:initialization-splice},
the construction below satisfies $x_{-1}=x_0$ in $\reals^2$ without adding a coordinate.
Set
\[
 z_\star\doteq\frac{\chi_\star}{2}
 +i\sqrt{q_{10}-\frac{\chi_\star^2}{4}},
 \qquad
 e_\star\doteq1+\beta_{10}-\beta_{10}z_\star^{-1},
 \qquad
 g_\star\doteq10(e_\star-z_\star).
\]
The active forward equality explicitly determines the function-value amplitude:
\begin{equation}
\label{eq:k10-function-amplitude}
 F_\star\doteq\frac{C_1^+(q_{10},z_\star)}{1-q_{10}}
 =0.6356563261206633998\ldots.
\end{equation}
We translate coordinates to place the minimizer at the origin and subtract the minimum value from the objective. The full query points,
gradients, and function values are
\begin{equation}
\label{eq:k10-trajectory-data}
\begin{aligned}
 y_0&=z_\star^{-1},
 &g_0&=10(z_\star^{-1}-1),
 &f_0&=\frac52,\\
 y_{t+1}&=z_\star^t e_\star,
 &g_{t+1}&=z_\star^t g_\star,
 &f_{t+1}&=q_{10}^tF_\star,
 \qquad t\geq0.
\end{aligned}
\end{equation}
In particular, the initial value $f_0=5/2$ is different from the amplitude $F_\star$ of the function values along the contracting part of the trajectory.
Starting with $x_{-1}=x_0=z_\star^{-1}$, the first gradient step gives
$x_1=y_0-g_0/10=1$. Thereafter $e_\star-g_\star/10=z_\star$ makes the
Nesterov identities give $x_t=z_\star^{t-1}$ for every $t\geq0$.
Consequently the normalized squared distance is exactly $q_{10}^t$, as in
\eqref{eq:exact-spiral}, including at $t=0$.

\paragraph{Checking that one objective realizes the trajectory.}
We still have to show that the prescribed points, gradients, and function
values come from one function in $\Fclass{10}$. For the contracting part,
this means checking precisely the bounds already derived:
\begin{equation}
  \label{eq:residual-checks}
C_m^+(q_{10},z_\star)\leq(1-q_{10}^m)F_\star
 \leq-C_m^-(q_{10},z_\star),\qquad m\geq1.  
\end{equation}
Equality in the left inequality at $m=1$ follows from the definition of
$F_\star$ in \eqref{eq:k10-function-amplitude}. Equality in the right
inequality at $m=18$ follows from the cycle identity
\eqref{eq:k10-cycle-polynomial} and $P(q_{10},\chi_\star)=0$; this is an exact identity. Outward-rounded interval arithmetic checks both inequalities for
$m=1,\ldots,400$. Apart from the two equalities for $m=1$ and $m=18$, each checked
inequality has a margin greater than $3\cdot10^{-6}$, where the margin is
the right-hand side minus the left-hand side. This finite check alone does
not cover all iteration indices. To handle the remaining ones, observe that
$z_\star^m \to 0$ and $q_{10}^m \to 0$, so the bounds in \eqref{eq:residual-checks} tend to
\begin{equation}
\label{eq:k10-minimizer-checks}
 B(q_{10},z_\star)\leq F_\star
 \leq\langle g_\star,e_\star\rangle-B(q_{10},z_\star).
\end{equation}
Both limiting inequalities also have margins greater than $3\cdot10^{-6}$.
These are the two interpolation inequalities for $(e_\star,g_\star,F_\star)$ and $(0,0,0)$, corresponding to the query point $e_\star$ and the optimum.
Expanding the formulas for $C_m^\pm$ and using the triangle inequality shows
that each finite-index margin differs from its limiting margin by at most
\[
 \left(F_\star+2|g_\star||e_\star|+3B(q_{10},z_\star)\right)q_{10}^{m/2}.
\]
This quantity is less than $10^{-49}$ for every $m\geq401$, so both
inequalities remain strict at all remaining indices.

The interpolation inequalities for all distinct pairs are covered by three
cases:
\begin{enumerate}
\item \emph{Two contracting queries.} For $1\leq i<j$, set $m=j-i$ and
multiply \eqref{eq:residual-checks} by $q_{10}^{i-1}$. Using the data in
\eqref{eq:k10-trajectory-data}, these are exactly the interpolation
inequalities \eqref{eq:ssc-interpolation} for the ordered pairs $(i,j)$ and
$(j,i)$.
\item \emph{A contracting query and the minimizer.} For query $i\geq1$,
multiply \eqref{eq:k10-minimizer-checks} by $q_{10}^{i-1}$ to obtain both
ordered comparisons with the minimizer's data $(0,0,0)$.
\item \emph{The initial query with index $0$.} Substitute
\eqref{eq:k10-trajectory-data} into \eqref{eq:ssc-interpolation} for $(0,j)$
and $(j,0)$ with $j\geq1$, and for both comparisons between the initial
query and the minimizer. The infinitely many comparisons with contracting
queries are covered by finite interval checks and geometric bounds for the
remaining indices.
\end{enumerate}
Together with interpolation, this establishes \emph{one objective} realizing the
entire trajectory, not a different objective for each horizon. Its normalized
squared distance is $q_{10}^t$, which gives the lower bound
$\qopt(\sqrt{10})\geq q_{10}$.

\paragraph{Ruling out larger asymptotic factors.}
The upper bound must apply to every admissible objective, not just the one
constructed above. We obtain it from the same polynomial $P$ and
Proposition~\ref{prop:strict-separator-rate}. That proposition applies when
a cycle has a strictly negative cost at every angle. We will establish this
for the cycle $\gamma=(0,1,\ldots,18,0)$ from
\eqref{eq:k10-active-cycle}, with $q>q_{10}$ sufficiently close to $q_{10}$.

First we relate $W_{18}$ to the cycle cost used in that proposition.
The change of variables \eqref{eq:interpolation-variables}--\eqref{eq:phi}
gives, at $\kappa=10$,
\[
 \varphi_i+f_i=\langle g_i,y_i\rangle-\frac12\|y_i\|^2
 -\frac1{18}\|g_i-y_i\|^2.
\]
Substituting in \eqref{eq:ssc-interpolation} converts its lower bounds on
$f_i-f_j$ into upper bounds on $\varphi_i-\varphi_j$.
The additional terms cancel when summed around the cycle, reversing the sign of the summed costs.
There is also a scaling factor.
For the contracting trajectory defining $C_m^\pm$, we have $e-(g-e)/9=10z/9$, so its $r$-sequence is $(10z/9)z^k$, rather than $z^k$ in \eqref{eq:nag-transfer}.
The inner products are therefore multiplied by $|10z/9|^2=100q/81$.
Thus the cycle cost from Lemma~\ref{lem:strict-alternative} satisfies
\begin{equation}
\label{eq:k10-separator-cycle}
 c_\gamma(\sqrt{10},q,\theta)
 =-\frac{81}{100q}W_{18}(q,z)
 =-\frac{9(1+q-\chi)}{100q^2}P(q,\chi),
 \qquad \chi=2\sqrt q\cos\theta.
\end{equation}
The last equality uses \eqref{eq:k10-cycle-polynomial}. Since $1+q-\chi>0$,
it remains to show that $P(q,\chi)>0$ throughout the allowed interval for
$q$ just above $q_{10}$, not merely at one angle.

Using the coefficients in Table~\ref{tab:k10-polynomial}, we differentiate
$P(q,\chi)$ in \eqref{eq:k10-expanded-polynomial} with respect to $q$ and
evaluate the derivative with outward-rounded intervals on
\eqref{eq:k10-root-box}, obtaining
\begin{equation}
\label{eq:k10-q-derivative}
 114<\partial_qP(q_{10},\chi_\star)<115.
\end{equation}
By continuity, $\partial_qP(q,\chi)>0$ near $(q_{10},\chi_\star)$.
For $q>q_{10}$ and $\chi$ in this neighborhood,
\[
 P(q,\chi)=P(q_{10},\chi)
 +\int_{q_{10}}^q\partial_qP(u,\chi)\,du>0,
\]
because \eqref{eq:k10-contact-factorization} makes the first term
nonnegative. Away from a fixed small neighborhood of $\chi_\star$, that
factorization and $R>17$ give a strictly positive minimum on the remaining
closed subset of $[-1.503,1.503]$. Continuity preserves this positivity for
$q$ sufficiently close to $q_{10}$. Moreover, $2\sqrt q<1.503$ for such $q$,
so this fixed interval still contains the entire allowed interval.
Consequently there exists $\varepsilon>0$, with $q_{10}+\varepsilon<1$,
such that
\begin{equation}
\label{eq:k10-strict-polynomial}
 P(q,\chi)>0
 \quad\text{whenever}\quad
 q_{10}<q<q_{10}+\varepsilon,
 \qquad |\chi|\leq2\sqrt q.
\end{equation}

For each such $q$, \eqref{eq:k10-separator-cycle} is strictly negative at
every angle, so this single cycle satisfies the strict alternative in
Lemma~\ref{lem:strict-alternative}. Since $q>q_{10}>q_{\mathrm{quad}}(\sqrt{10})$,
Proposition~\ref{prop:strict-separator-rate} now gives
\begin{equation}
\label{eq:k10-uniform-upper}
 \|x_t-x_\star\|^2\leq b(\sqrt{10},q)q^t\|x_0-x_\star\|^2,
 \qquad t\geq0,
\end{equation}
for every finite-dimensional $f\in\Fclass{10}$ with the standard
initialization. The finite constant depends only on $q$ and the fixed
condition number, not on $f$, the dimension, $x_0$, or $t$.
For $x_0\ne x_\star$ and $t\geq1$, \eqref{eq:k10-uniform-upper} gives
\[
 R_t(f,x_0)^{1/t}\leq b(\sqrt{10},q)^{1/t}q.
\]
The same bound holds after taking the supremum over objectives, dimensions,
and starting points. Letting $t\to\infty$ therefore bounds both rates in
\eqref{eq:rho-fix}--\eqref{eq:rho-hor} by $q$. This holds for every
$q\in(q_{10},q_{10}+\varepsilon)$; letting $q\downarrow q_{10}$ gives the
upper bound $q_{10}$. Combining it with the realizing trajectory and
Theorem~\ref{thm:main} proves
\[
 \rhofix(10)=\rhohor(10)=\qopt(\sqrt{10})=q_{10}.
\]
The limit in $q$ is taken after the asymptotic limit in $t$; no bound on
$b(\sqrt{10},q)$ as $q\downarrow q_{10}$ is needed. The tangency equations
alone are insufficient: the positive factorization, the derivative bound
\eqref{eq:k10-q-derivative}, and the uniform-rate proposition are the
additional ingredients of the upper bound.

\paragraph{Full polynomial coefficients.}
To specify the polynomial in \eqref{eq:k10-cycle-polynomial} without any
external data, write
\begin{equation}
\label{eq:k10-expanded-polynomial}
 P(q,\chi)=5\sum_{j=0}^{19}\sum_{k=0}^{19-j}
 (a_{jk}+b_{jk}\sqrt{10})q^k\chi^j.
\end{equation}
The table below lists all $210$ integer pairs $(a_{jk},b_{jk})$, including
the zero coefficients. The common factor $5$ is part of this normalization.
For example, the terms with the two highest powers of $\chi$ are
\[
 5(11-2\sqrt{10})\chi^{19}
 +5(11-2\sqrt{10}-9q)\chi^{18}.
\]
For the full set of values see the table.
\begingroup
\small
\setlength{\tabcolsep}{3pt}
\begin{longtable}{@{}rrrr@{\qquad}rrrr@{\qquad}rrrr@{}}
\caption{Complete coefficient table for \eqref{eq:k10-expanded-polynomial}.
The coefficient of $q^k\chi^j$ is $5(a_{jk}+b_{jk}\sqrt{10})$.
Read each of the three blocks from top to bottom. All other coefficients are zero.}
\label{tab:k10-polynomial}\\
\toprule
$j$ & $k$ & $a_{jk}$ & $b_{jk}$ & $j$ & $k$ & $a_{jk}$ & $b_{jk}$ & $j$ & $k$ & $a_{jk}$ & $b_{jk}$ \\
\midrule
\endfirsthead
\toprule
$j$ & $k$ & $a_{jk}$ & $b_{jk}$ & $j$ & $k$ & $a_{jk}$ & $b_{jk}$ & $j$ & $k$ & $a_{jk}$ & $b_{jk}$ \\
\midrule
\endhead
\endfoot
\bottomrule
\endlastfoot
$19$ & $0$ & $11$ & $-2$ & $8$ & $4$ & $9195$ & $-1320$ & $3$ & $4$ & $885$ & $-110$ \\
$18$ & $0$ & $11$ & $-2$ & $8$ & $5$ & $-25707$ & $3564$ & $3$ & $5$ & $-1516$ & $182$ \\
$18$ & $1$ & $-9$ & $0$ & $8$ & $6$ & $16533$ & $0$ & $3$ & $6$ & $2394$ & $-280$ \\
$17$ & $0$ & $11$ & $-2$ & $8$ & $7$ & $495$ & $0$ & $3$ & $7$ & $-3560$ & $408$ \\
$17$ & $1$ & $-218$ & $38$ & $8$ & $8$ & $-165$ & $0$ & $3$ & $8$ & $5055$ & $-570$ \\
$17$ & $2$ & $1$ & $0$ & $8$ & $9$ & $45$ & $0$ & $3$ & $9$ & $-120$ & $0$ \\
$16$ & $0$ & $11$ & $-2$ & $8$ & $10$ & $-9$ & $0$ & $3$ & $10$ & $84$ & $0$ \\
$16$ & $1$ & $-207$ & $36$ & $8$ & $11$ & $1$ & $0$ & $3$ & $11$ & $-56$ & $0$ \\
$16$ & $2$ & $163$ & $0$ & $7$ & $0$ & $11$ & $-2$ & $3$ & $12$ & $35$ & $0$ \\
$16$ & $3$ & $1$ & $0$ & $7$ & $1$ & $-108$ & $18$ & $3$ & $13$ & $-20$ & $0$ \\
$15$ & $0$ & $11$ & $-2$ & $7$ & $2$ & $566$ & $-88$ & $3$ & $14$ & $10$ & $0$ \\
$15$ & $1$ & $-196$ & $34$ & $7$ & $3$ & $-2120$ & $312$ & $3$ & $15$ & $-4$ & $0$ \\
$15$ & $2$ & $1826$ & $-304$ & $7$ & $4$ & $6390$ & $-900$ & $3$ & $16$ & $1$ & $0$ \\
$15$ & $3$ & $-16$ & $0$ & $7$ & $5$ & $-16512$ & $2244$ & $2$ & $0$ & $11$ & $-2$ \\
$15$ & $4$ & $1$ & $0$ & $7$ & $6$ & $38016$ & $-5016$ & $2$ & $1$ & $-53$ & $8$ \\
$14$ & $0$ & $11$ & $-2$ & $7$ & $7$ & $-792$ & $0$ & $2$ & $2$ & $136$ & $-18$ \\
$14$ & $1$ & $-185$ & $32$ & $7$ & $8$ & $330$ & $0$ & $2$ & $3$ & $-260$ & $32$ \\
$14$ & $2$ & $1630$ & $-270$ & $7$ & $9$ & $-120$ & $0$ & $2$ & $4$ & $425$ & $-50$ \\
$14$ & $3$ & $-1230$ & $0$ & $7$ & $10$ & $36$ & $0$ & $2$ & $5$ & $-631$ & $72$ \\
$14$ & $4$ & $-15$ & $0$ & $7$ & $11$ & $-8$ & $0$ & $2$ & $6$ & $878$ & $-98$ \\
$14$ & $5$ & $1$ & $0$ & $7$ & $12$ & $1$ & $0$ & $2$ & $7$ & $-1166$ & $128$ \\
$13$ & $0$ & $11$ & $-2$ & $6$ & $0$ & $11$ & $-2$ & $2$ & $8$ & $1495$ & $-162$ \\
$13$ & $1$ & $-174$ & $30$ & $6$ & $1$ & $-97$ & $16$ & $2$ & $9$ & $-765$ & $0$ \\
$13$ & $2$ & $1445$ & $-238$ & $6$ & $2$ & $458$ & $-70$ & $2$ & $10$ & $-36$ & $0$ \\
$13$ & $3$ & $-8400$ & $1330$ & $6$ & $3$ & $-1554$ & $224$ & $2$ & $11$ & $28$ & $0$ \\
$13$ & $4$ & $105$ & $0$ & $6$ & $4$ & $4270$ & $-588$ & $2$ & $12$ & $-21$ & $0$ \\
$13$ & $5$ & $-14$ & $0$ & $6$ & $5$ & $-10122$ & $1344$ & $2$ & $13$ & $15$ & $0$ \\
$13$ & $6$ & $1$ & $0$ & $6$ & $6$ & $21504$ & $-2772$ & $2$ & $14$ & $-10$ & $0$ \\
$12$ & $0$ & $11$ & $-2$ & $6$ & $7$ & $-12936$ & $0$ & $2$ & $15$ & $6$ & $0$ \\
$12$ & $1$ & $-163$ & $28$ & $6$ & $8$ & $-462$ & $0$ & $2$ & $16$ & $-3$ & $0$ \\
$12$ & $2$ & $1271$ & $-208$ & $6$ & $9$ & $210$ & $0$ & $2$ & $17$ & $1$ & $0$ \\
$12$ & $3$ & $-6955$ & $1092$ & $6$ & $10$ & $-84$ & $0$ & $1$ & $0$ & $0$ & $0$ \\
$12$ & $4$ & $5005$ & $0$ & $6$ & $11$ & $28$ & $0$ & $1$ & $1$ & $-53$ & $8$ \\
$12$ & $5$ & $91$ & $0$ & $6$ & $12$ & $-7$ & $0$ & $1$ & $2$ & $72$ & $-8$ \\
$12$ & $6$ & $-13$ & $0$ & $6$ & $13$ & $1$ & $0$ & $1$ & $3$ & $-135$ & $16$ \\
$12$ & $7$ & $1$ & $0$ & $5$ & $0$ & $11$ & $-2$ & $1$ & $4$ & $154$ & $-16$ \\
$11$ & $0$ & $11$ & $-2$ & $5$ & $1$ & $-86$ & $14$ & $1$ & $5$ & $-217$ & $24$ \\
$11$ & $1$ & $-152$ & $26$ & $5$ & $2$ & $361$ & $-54$ & $1$ & $6$ & $236$ & $-24$ \\
$11$ & $2$ & $1108$ & $-180$ & $5$ & $3$ & $-1096$ & $154$ & $1$ & $7$ & $-299$ & $32$ \\
$11$ & $3$ & $-5684$ & $884$ & $5$ & $4$ & $2716$ & $-364$ & $1$ & $8$ & $318$ & $-32$ \\
$11$ & $4$ & $23075$ & $-3458$ & $5$ & $5$ & $-5852$ & $756$ & $1$ & $9$ & $-381$ & $40$ \\
$11$ & $5$ & $-364$ & $0$ & $5$ & $6$ & $11382$ & $-1428$ & $1$ & $10$ & $-2$ & $2$ \\
$11$ & $6$ & $78$ & $0$ & $5$ & $7$ & $-20472$ & $2508$ & $1$ & $11$ & $-19$ & $2$ \\
$11$ & $7$ & $-12$ & $0$ & $5$ & $8$ & $462$ & $0$ & $1$ & $12$ & $-4$ & $2$ \\
$11$ & $8$ & $1$ & $0$ & $5$ & $9$ & $-252$ & $0$ & $1$ & $13$ & $-17$ & $2$ \\
$10$ & $0$ & $11$ & $-2$ & $5$ & $10$ & $126$ & $0$ & $1$ & $14$ & $-6$ & $2$ \\
$10$ & $1$ & $-141$ & $24$ & $5$ & $11$ & $-56$ & $0$ & $1$ & $15$ & $-15$ & $2$ \\
$10$ & $2$ & $956$ & $-154$ & $5$ & $12$ & $21$ & $0$ & $1$ & $16$ & $-8$ & $2$ \\
$10$ & $3$ & $-4576$ & $704$ & $5$ & $13$ & $-6$ & $0$ & $1$ & $17$ & $-13$ & $2$ \\
$10$ & $4$ & $17391$ & $-2574$ & $5$ & $14$ & $1$ & $0$ & $1$ & $18$ & $-10$ & $2$ \\
$10$ & $5$ & $-11869$ & $0$ & $4$ & $0$ & $11$ & $-2$ & $0$ & $0$ & $0$ & $0$ \\
$10$ & $6$ & $-286$ & $0$ & $4$ & $1$ & $-75$ & $12$ & $0$ & $1$ & $-22$ & $4$ \\
$10$ & $7$ & $66$ & $0$ & $4$ & $2$ & $275$ & $-40$ & $0$ & $2$ & $60$ & $-4$ \\
$10$ & $8$ & $-11$ & $0$ & $4$ & $3$ & $-735$ & $100$ & $0$ & $3$ & $-22$ & $4$ \\
$10$ & $9$ & $1$ & $0$ & $4$ & $4$ & $1620$ & $-210$ & $0$ & $4$ & $60$ & $-4$ \\
$9$ & $0$ & $11$ & $-2$ & $4$ & $5$ & $-3136$ & $392$ & $0$ & $5$ & $-22$ & $4$ \\
$9$ & $1$ & $-130$ & $22$ & $4$ & $6$ & $5530$ & $-672$ & $0$ & $6$ & $60$ & $-4$ \\
$9$ & $2$ & $815$ & $-130$ & $4$ & $7$ & $-9090$ & $1080$ & $0$ & $7$ & $-22$ & $4$ \\
$9$ & $3$ & $-3620$ & $550$ & $4$ & $8$ & $5070$ & $0$ & $0$ & $8$ & $60$ & $-4$ \\
$9$ & $4$ & $12815$ & $-1870$ & $4$ & $9$ & $210$ & $0$ & $0$ & $9$ & $-22$ & $4$ \\
$9$ & $5$ & $-38522$ & $5434$ & $4$ & $10$ & $-126$ & $0$ & $0$ & $10$ & $38$ & $0$ \\
$9$ & $6$ & $715$ & $0$ & $4$ & $11$ & $70$ & $0$ & $0$ & $11$ & $20$ & $0$ \\
$9$ & $7$ & $-220$ & $0$ & $4$ & $12$ & $-35$ & $0$ & $0$ & $12$ & $18$ & $0$ \\
$9$ & $8$ & $55$ & $0$ & $4$ & $13$ & $15$ & $0$ & $0$ & $13$ & $20$ & $0$ \\
$9$ & $9$ & $-10$ & $0$ & $4$ & $14$ & $-5$ & $0$ & $0$ & $14$ & $18$ & $0$ \\
$9$ & $10$ & $1$ & $0$ & $4$ & $15$ & $1$ & $0$ & $0$ & $15$ & $20$ & $0$ \\
$8$ & $0$ & $11$ & $-2$ & $3$ & $0$ & $11$ & $-2$ & $0$ & $16$ & $18$ & $0$ \\
$8$ & $1$ & $-119$ & $20$ & $3$ & $1$ & $-64$ & $10$ & $0$ & $17$ & $20$ & $0$ \\
$8$ & $2$ & $685$ & $-108$ & $3$ & $2$ & $200$ & $-28$ & $0$ & $18$ & $18$ & $0$ \\
$8$ & $3$ & $-2805$ & $420$ & $3$ & $3$ & $-460$ & $60$ & $0$ & $19$ & $20$ & $0$ \\
\end{longtable}
\endgroup
 
\section{An exact two-component optimal solution at
\texorpdfstring{$\kappa=100$}{kappa=100}}
\label{app:kappa-hundred}
 The $\kappa=100$ case provides a nontrivial example of the finite-support property beyond one planar spiral. The calculation is similar to Appendix~\ref{app:kappa-ten} and we only repeat key steps.

\begin{proposition}[An optimal solution requiring two components]
\label{prop:kappa-hundred}
At $\kappa=100$,
\[
 q_{100}=\qopt(\sqrt{100})=0.859899778052548365740446035906789693\ldots .
\]
There is an optimal angular pair whose measure is supported on two frequencies,
and every optimal angular measure has exactly these two frequencies in its support.
The resulting exact self-similar tail spans four real dimensions; after
enforcing $x_{-1}=x_0$, one fixed function attaining this factor can be chosen in
$\reals^5$.
\end{proposition}

\begin{proof}
We state the exact certificate in the normalization $\sigma_t=z^t$, as in Appendix~\ref{app:uniform-upper-bound}.
If $(\eta,v)$ denotes angular weights and the initial interpolation value in this normalization, the original normalization is
\begin{equation}
\label{eq:normalization-change}
 d\nu=\frac{|p_s(z)|^2\,d\eta}{\int|p_s(z)|^2\,d\eta},
 \qquad \varphi_0=\frac{v}{\int|p_s(z)|^2\,d\eta}.
\end{equation}
The inverse uses $|d_s(z)|^2$ in the same way.
Both weights are continuous and strictly positive, so these transformations preserve supports and every interpolation inequality.
Let
\[
 \xi_\star\doteq(q_{100},v_\star,w_\star,
 \chi_{1\star},\chi_{2\star},\lambda_{49},\lambda_{63},\lambda_{64})
\]
be the unique zero in the rational box of radius $10^{-58}$ around
\begin{align*}
 q_{100}&=0.859899778052548365740446035906789693\ldots,\\
 v_\star&=0.002854749733611906316731334872665651\ldots,\\
 w_\star&=0.432893789374821564486049487863967595\ldots,\\
 \chi_{1\star}&=1.847585923705342323643049369488243114\ldots,\\
 \chi_{2\star}&=1.833337063888573702334307789415459859\ldots.
\end{align*}
Here $\chi_j\doteq2\sqrt q\cos\theta_j$ is a \emph{real frequency coordinate}.  The
corresponding angles are
\[
 \theta_1=0.08709509851058300453\ldots,
 \qquad
 \theta_2=0.15162655445126233148\ldots,
\]
and the certificate measure is
$\eta_\star\doteq w_\star\delta_{\theta_1}
 +(1-w_\star)\delta_{\theta_2}$.
The optimal angular pair $(\nu_\star,\varphi_0^\star)$ satisfying \eqref{eq:star-feasible} at $q=q_{100}$ is obtained from $(\eta_\star,v_\star)$ by \eqref{eq:normalization-change}.

To specify the algebraic system, fix $s=\sqrt{100}=10$ and define, for integers $m\geq1$,
\[
 L_m(q,\chi)\doteq-\frac{|p_{10}(z)|^2c_{m0}(\theta)}{1-q^m},\qquad
 U_m(q,\chi)\doteq \frac{|p_{10}(z)|^2c_{0m}(\theta)}{1-q^m},
\]
where $z=\sqrt q\,e^{i\theta}$, $\theta=\arccos(\chi/(2\sqrt q))$, and
$c_{ij}$ is from \eqref{eq:edge-cost} at $s=10$.
Write
\[
 \langle A\rangle_w\doteq wA(q,\chi_1)+(1-w)A(q,\chi_2),
\]
and set
\[
 \Phi_{100}(q,\chi)\doteq\lambda_{49}L_{49}(q,\chi)
 +\lambda_{63}L_{63}(q,\chi)
 +\lambda_{64}L_{64}(q,\chi)-U_1(q,\chi).
\]
Dependence on the dual weights is suppressed in $\Phi_{100}$; derivatives
with respect to $\chi$ keep those weights fixed.
The eight equations are
\begin{gather}
 \langle U_1\rangle_w=v,
 \qquad
 \langle L_{49}\rangle_w=
 \langle L_{63}\rangle_w=
 \langle L_{64}\rangle_w=v,
\label{eq:k100-primal-system}\\
 \lambda_{49}+\lambda_{63}+\lambda_{64}=1,
 \qquad
 \Phi_{100}(q,\chi_1)=\partial_\chi\Phi_{100}(q,\chi_1)=\partial_\chi\Phi_{100}(q,\chi_2)=0.
\label{eq:k100-dual-system}
\end{gather}
The omitted equality $\Phi_{100}(q,\chi_2)=0$ follows from
\eqref{eq:k100-primal-system} and the positive component weights.  An
interval Krawczyk inclusion proves existence and uniqueness of the root in
the rational box.  All three dual weights are positive and equal
approximately to
\[
 (\lambda_{49},\lambda_{63},\lambda_{64})
 \approx ( 0.7275254636036373,\ 0.0389115479752263,\
   0.2335629884211364).
\]

Figure~\ref{fig:k100-tangency} plots $\chi\mapsto\Phi_{100}(q_{100},\chi)$, keeping the
three dual multipliers fixed at the same algebraic solution. Its two contacts
correspond to the two rotation angles above.
\begin{figure}[tbp]
\centering
\includegraphics[width=\linewidth]{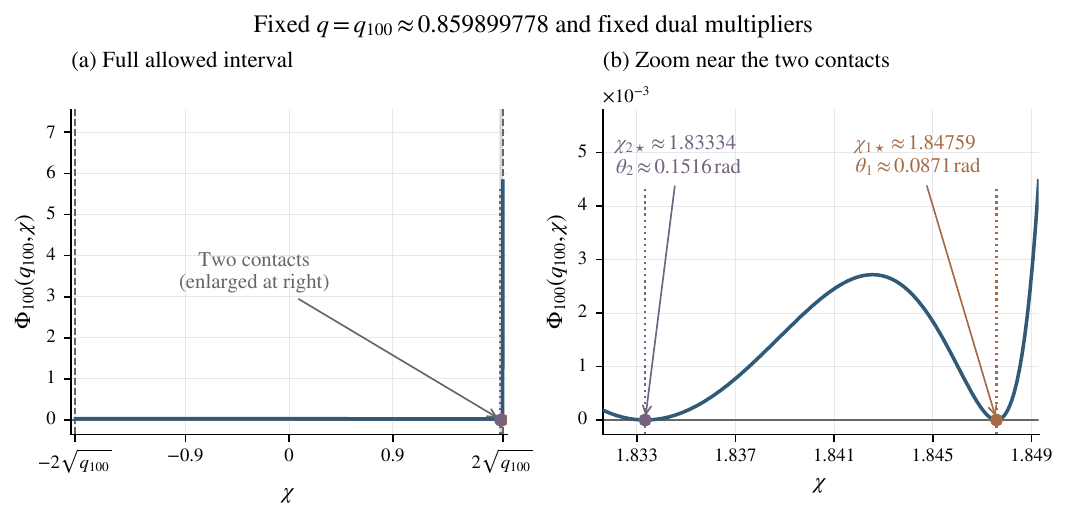}
\caption{Numerical illustration of $\Phi_{100}(q_{100},\chi)$ with $q$ and the dual
multipliers fixed. Left: the full allowed interval
$[-2\sqrt{q_{100}},2\sqrt{q_{100}}]$.
Right: enlargement of the shaded region, showing both contacts
$\chi_{1\star}$ and $\chi_{2\star}$, labeled by their rotation angles
$\theta_j=\arccos(\chi_{j\star}/(2\sqrt{q_{100}}))$.}
\label{fig:k100-tangency}
\end{figure}

\paragraph{The four-dimensional self-similar tail.}
Set
\[
 z_j\doteq\frac{\chi_{j\star}}2
 +i\sqrt{q_{100}-\frac{\chi_{j\star}^2}{4}},
 \qquad a_1\doteq\sqrt{w_\star},
 \qquad a_2\doteq\sqrt{1-w_\star}.
\]
Thus $|z_1|^2=|z_2|^2=q_{100}$.  In $\mathbb C^2\cong\reals^4$, define
\begin{align}
 \sigma_t&\doteq(a_1z_1^t,a_2z_2^t),\notag\\
 r_t&\doteq(a_1p_{10}(z_1)z_1^t,
        a_2p_{10}(z_2)z_2^t),
 &x_{t+1}&\doteq\frac{99}{100}r_t,
 \qquad t\in\mathbb Z,\notag\\
 y_t&\doteq r_t+\frac1{99}\sigma_t,
 &g_t&\doteq y_t+\sigma_t,
 \qquad t\geq0.
\label{eq:k100-tail}
\end{align}
Here $p_{10}$ means \eqref{eq:inverse-transfer} with $s=10$.  Its definition makes these
identities exactly consistent with the Nesterov momentum and gradient steps.  In particular,
\begin{equation}
 \frac{\|x_{t+1}\|^2}{\|x_1\|^2}=q_{100}^t,
 \qquad t\geq0.
\label{eq:k100-exact-decay}
\end{equation}
The four eigenvalues $z_1,\bar z_1,z_2,\bar z_2$ are pairwise distinct and
all four modal coefficients are nonzero.  The associated four-column
Vandermonde determinant is therefore nonzero, so the real span of the tail is
exactly four-dimensional.
For completeness, with the objective minimum normalized to $f(0)=0$, the
corresponding function values can be taken as
\begin{equation}
 f_t\doteq\langle\sigma_t,y_t\rangle-q_{100}^tv_\star
 -\frac1{198}\|\sigma_t\|^2+\frac12\|y_t\|^2.
\label{eq:k100-values}
\end{equation}
Indeed, with $\varphi(0)=0$, Lemma~\ref{lem:stars} allows the interpolation
values $\varphi_t=q_{100}^t v_\star$ in the normalization of
\eqref{eq:k100-tail}. They satisfy
\[
 \varphi_i-\varphi_j\leq\langle r_i,\sigma_i-\sigma_j\rangle
\]
for every ordered pair, including comparisons with the optimum.
There is therefore a proper closed convex function $\varphi$ with
$\varphi(\sigma_t)=\varphi_t$ and $r_t\in\partial\varphi(\sigma_t)$ for every $t$.
Reversing the transform \eqref{eq:phi} then gives \eqref{eq:k100-values},
so the entire sequence of query points, gradients, and function values comes
from one function in $\Fclass{100}(\reals^4)$.

\paragraph{Why two components are necessary.}
As in the case for $\kappa=10$ in Appendix~\ref{app:kappa-ten},
we exactly check the two interpolation inequalities for each index
$m=1,\ldots,180$ and then use geometric bounds for both families at larger
indices. Equality holds only in the upper inequality at $m=1$ and the
lower inequalities at $m=49,63,64$. The interval calculation also proves the global
factorization
\[
 \Phi_{100}(q_{100},\chi)=
 \frac{(\chi-\chi_{1\star})^2(\chi-\chi_{2\star})^2R_{62}(\chi)}{D(\chi)},
 \qquad
 \frac{R_{62}(\chi)}{D(\chi)}>0,
 \quad |\chi|\leq2\sqrt{q_{100}}.
\]
For any feasible angular measure $\nu$ at $q=q_{100}$, the corresponding angular measure $\eta$ from \eqref{eq:normalization-change} satisfies
$\int\Phi_{100}(q_{100},\chi(\theta))\,d\eta(\theta)\leq0$, while the factorization gives
$\Phi_{100}(q_{100},\chi)\geq0$.  We call its two zeros the \emph{contact
frequencies}, because the dual inequality is tight there.  Such a measure must
therefore be supported on these two frequencies, and the change of weights preserves support.  Neither contact works
alone: directed interval arithmetic proves
\begin{align*}
 L_{56}(q_{100},\chi_{1\star})
 -U_1(q_{100},\chi_{1\star})
 &>0.00511062771457,\\
 L_{35}(q_{100},\chi_{2\star})
 -U_1(q_{100},\chi_{2\star})
 &>0.0281844327967.
\end{align*}
Each singleton therefore has an empty feasible interval for $\varphi_0$.  Every
optimal angular measure consequently has support cardinality exactly two, not
merely at most two. This counts angles in $[0,\pi]$; a disk measure may split
the mass at each angle between conjugate points.

For the matching upper bound, let $C_m^\sigma(q,\chi)$ be the cost of the
cycle $(0,1,\ldots,m,0)$ for $\sigma_t=z^t$ and $r_t=p_{10}(z)z^t$.
Summing its consecutive comparisons and its closing comparison gives
\[
 C_m^\sigma(q,\chi)=(1-q^m)\bigl(U_1(q,\chi)-L_m(q,\chi)\bigr).
\]
With $q=r^2$, define
\[
 S(r,\theta)\doteq
 \sum_{m\in\{49,63,64\}}\frac{\lambda_m}{1-q_{100}^m}
 C_m^\sigma(r^2,2r\cos\theta).
\]
The coefficients $\lambda_m/(1-q_{100}^m)$ are positive and independent of $r$.
They are the cycle weights $w_\gamma$ of Appendix~\ref{app:uniform-upper-bound} for $\gamma=(0,1,\ldots,m,0)$, whereas $w_\star$ and $1-w_\star$ weight the two angles.
At $r=\sqrt{q_{100}}$, $S(r,\theta)=-\Phi_{100}(q_{100},2r\cos\theta)$.
Set
\[
 N(r,\theta) \doteq 99\lvert10re^{i\theta}-9\rvert^4S(r,\theta).
\]
The multiplicative factor is positive on the relevant rectangle.  At
$r=\sqrt{q_{100}}$, the factorization above gives $S(r,\theta)\leq0$.  The
checker proves $\partial_rN(r,\theta)<-0.585$ throughout
$\sqrt{q_{100}}\leq r\leq\sqrt{0.9}$ by covering this rectangle with $1544$
directed rational boxes.  Hence $S(r,\theta)<0$ for every
$r>\sqrt{q_{100}}$ in this interval.  Its coefficients remain nonnegative
weights on finite cycle inequalities, so every
$q_{100}<q\leq0.9$ is infeasible. Lemma~\ref{lem:closed-rate-problem}
already gives $\qopt(10)\leq0.9$. Thus
\[
 \rho_{\rm fix}(100)=\rho_{\rm hor}(100)=\qopt(10)=q_{100}.
\]

The exact trajectory \eqref{eq:k100-tail} occupies four real dimensions and does not by itself enforce $x_{-1}=x_0$.
The construction in Lemma~\ref{lem:initialization-splice} scales this trajectory and adds one orthogonal scalar coordinate to $x_t$, equal to zero for $t\geq2$.
It therefore produces one function in
$\Fclass{100}(\reals^5)$ with the standard initialization and asymptotic
factor $q_{100}$. Equation~\eqref{eq:k100-exact-decay} describes the
four-dimensional sequence in \eqref{eq:k100-tail}, not the five-dimensional
trajectory with $x_{-1}=x_0$.
Figure~\ref{fig:k100-multimode} shows an orthogonal
three-dimensional projection of the genuine four-dimensional tail.
\end{proof}

\begin{figure}[h]
\centering
\includegraphics[width=\linewidth]{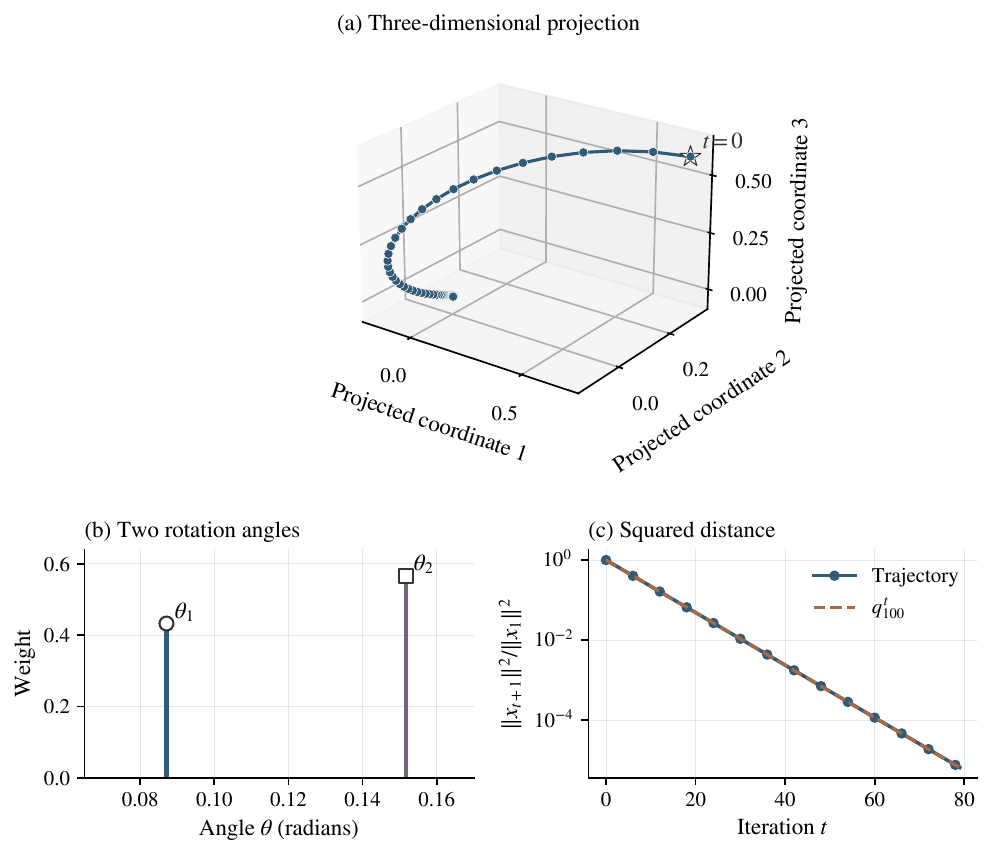}
\caption{The $\kappa=100$ example, certified by exact arithmetic.
(a) An orthogonal three-dimensional projection of the four-dimensional trajectory after initialization.
The star marks $t=0$.
(b) The two rotation angles and their positive weights.
(c) The normalized squared norm in four dimensions equals $q_{100}^t$.
The construction with standard initialization has a fifth coordinate that is zero from iteration two onward and is not shown.
The exact certificate and trajectory data are included in the supplement.}
\label{fig:k100-multimode}
\end{figure}

\section{Algebraic degree and formulas by radicals}
\label{app:rate-radicals}

We prove Remark~\ref{rem:no-radical-formula} using another exact example
with one contracting and rotating component. Here the interpolation
comparisons $(0,1)$, $(5,0)$, and $(6,0)$ are active, meaning that they hold
with equality. We first identify the optimal squared-distance factor, then
show that it has an algebraic specification but cannot be expressed by
radicals over the rationals. This rules out typical descriptions of convergence rates expressed as radicals but we stress that it does not rule out other
notions of closed form.

\subsection{Two active comparison cycles and an exact rate}

Write $z=\sqrt q\,e^{i\theta}$, where $0<q<1$ and
$\theta\in[0,\pi]$, and put $\chi=z+\bar z=2\sqrt q\cos\theta$.
Thus $q$ is the squared contraction factor, while the real coordinate
$\chi$ encodes the rotation angle, as in Appendix~\ref{app:kappa-ten}.
For $s=\sqrt\kappa>1$, write $\alpha_s\doteq1-1/s$, so
$sz-(s-1)=s(z-\alpha_s)$. At fixed $s$, the coefficient $p_s(z)$ in \eqref{eq:inverse-transfer} gives $r_t=p_s(z)z^t$ when $\sigma_t=z^t$.
The corresponding query points satisfy $y_t=(p_s(z)+(s^2-1)^{-1})z^t$. The examples in Appendices~\ref{app:kappa-ten} and \ref{app:kappa-hundred} have $\kappa$ fixed and as such do not need $\alpha_s$.

We again sum interpolation comparisons around a cycle. For an integer
$m\geq1$, take the consecutive comparisons $(k,k+1)$ for $k=0,\ldots,m-1$
and the closing comparison $(m,0)$. The values on the left-hand sides
cancel, leaving a constraint associated with $(0,1,\ldots,m,0)$.
To express it in real coordinates, as in Appendix~\ref{app:kappa-ten} we introduce the sums of conjugate powers
$T_j\doteq z^j+\bar z^j$, which satisfy
\[
 T_0=2,\qquad T_1=\chi,\qquad T_{j+1}=\chi T_j-qT_{j-1},
\]
for $j\geq1$. In the polynomial family below, $\alpha$ is a
placeholder for a real number that we later set to $\alpha_s$. Clearing the transfer-function denominator in the cycle expansion produces real parts involving
$z^2(\bar z-\alpha)^2$ and its product with $\bar z$. We abbreviate these by
\[
\begin{aligned}
 D&\doteq2\operatorname{Re}\!\left[z^2(\bar z-\alpha)^2\right]
   =2q^2-2\alpha q\chi+\alpha^2(\chi^2-2q),\\
 E&\doteq2\operatorname{Re}\!\left[z^2(\bar z-\alpha)^2\bar z\right]
   =(q^2+\alpha^2q)\chi-4\alpha q^2,\\
 H_m&\doteq\sum_{j=0}^{m-1}q^j.
\end{aligned}
\]
The polynomial expressions follow from $z\bar z=q$ and $z+\bar z=\chi$.
In particular, $D-E$ collects the contribution with the factor $1-\bar z$.
For $m\geq1$, define the polynomial
\begin{equation}
\label{eq:radical-cycle-polynomial}
 P_m(\alpha,q,\chi)\doteq
 \frac{(D-E)H_m+q^mD-q^2T_m+2\alpha qT_{m+1}-\alpha^2T_{m+2}}
      {\chi-1-q}.
\end{equation}
The numerator is divisible by $\chi-1-q$, so the quotient is a polynomial.
To see this, substitute $\chi=1+q$. The recurrence gives $T_j=1+q^j$,
and the numerator becomes
\[
 (q-\alpha)^2\bigl[(1-q)H_m+q^m-1\bigr]=0,
\]
where the last equality uses the finite geometric-sum identity
$(1-q)H_m=1-q^m$.
Expanding the cycle cost from \eqref{eq:disk-cycle-cost} gives the identity
\begin{equation}
\label{eq:radical-cycle-normalization}
 \mathcal C_{(0,1,\ldots,m,0)}^s(z)
   =\frac{1+q-\chi}{2}P_m(\alpha_s,q,\chi).
\end{equation}
The prefactor is positive on $|\chi|\leq2\sqrt q$, $0<q<1$.
For comparison, the polynomial in \eqref{eq:k10-cycle-polynomial} satisfies
$P(q,\chi)=-50P_{18}(1-1/\sqrt{10},q,\chi)$.
We will use $m=5$ and $m=6$, corresponding to the two cycles
$(0,1,\ldots,5,0)$ and $(0,1,\ldots,6,0)$. These are two cycles of
interpolation comparisons for a single rotating component, not two
components as in Appendix~\ref{app:kappa-hundred}, nor periodic trajectories.

\paragraph{Selecting an exact candidate.} We follow the same approach as before in Appendix~\ref{app:kappa-ten}.
Set
\[
 s_\dagger\doteq\frac{1421}{1250},\qquad
 \alpha_\dagger\doteq1-\frac1{s_\dagger}=\frac{171}{1421},\qquad
 \kappa_\dagger\doteq s_\dagger^2=\frac{2019241}{1562500}.
\]
Here $s_\dagger$, and therefore $\alpha_\dagger$ and $\kappa_\dagger$,
are exact rational numbers. Consequently,
\[
 P_5(\alpha_\dagger,q,\chi),\;P_6(\alpha_\dagger,q,\chi)
 \in\mathbb Q[q,\chi].
\]
In contrast, the polynomial we used for $\kappa=10$ has coefficients involving
$\sqrt{10}$ and belongs to $\mathbb Q(\sqrt{10})[q,\chi]$.
We use this separate construction so that the elimination and Galois
argument can be carried out directly over $\mathbb Q$. It is the rational
choice of $s_\dagger$, not just of the condition number, that ensures
rational coefficients. The rational data are the fixed parameters and
polynomial coefficients; the unknown roots need not be rational.

Let $(q_\dagger,\chi_\dagger)$ be the common zero
\begin{equation}
\label{eq:mixed-radical-system}
 P_5(\alpha_\dagger,q,\chi)=P_6(\alpha_\dagger,q,\chi)=0
\end{equation}
in the rational isolating box of radius $10^{-72}$ around
\[
 q_\dagger=0.0259177578383525961624\ldots,\qquad
 \chi_\dagger=0.2515525327377247329537\ldots .
\]
As before interval Newton inclusion proves existence and uniqueness there,
as well as $q_{\mathrm{quad}}(s_\dagger)<q_\dagger<1$ and $\chi_\dagger^2<4q_\dagger$.

\begin{lemma}[Identification of the exact rate]
\label{lem:mixed-radical-rate}
The isolated root satisfies $\qopt(s_\dagger)=q_\dagger$.
\end{lemma}
\begin{proof}
\emph{Interpolation inequalities and the lower bound.}
We use the normalization of Appendix~\ref{app:uniform-upper-bound}, with
$p_s$ defined in \eqref{eq:inverse-transfer}. For
$z=\sqrt q\,e^{i\theta}$, $s=s_\dagger$, and $q>q_{\mathrm{quad}}(s)$, define
\[
 K_m(z)\doteq(1-\bar z)\frac{1-q^m}{1-q}+q^m-z^m,
 \qquad C_m^\sigma(q,\chi)\doteq\operatorname{Re}[p_s(z)K_m(z)].
\]
Here $C_m^\sigma$ is the sum of the interpolation terms around
$(0,1,\ldots,m,0)$ for $\sigma_t=z^t$, $r_t=p_s(z)z^t$. Unlike
$C_m^\pm$ in Appendix~\ref{app:kappa-ten}, it is a whole-cycle cost, not
the right-hand side of an individual comparison. These cycle costs are
rational functions of $(q,\chi)$, and
\[
 \mathcal C_{(0,1,\ldots,m,0)}^s(z)
   =(s^2-1)|z-\alpha_s|^4 C_m^\sigma(q,\chi).
\]
Since $s=s_\dagger$ is now fixed, $\alpha_s=\alpha_\dagger$ and
$|z-\alpha_\dagger|^2=q-\alpha_\dagger\chi+\alpha_\dagger^2$.
Combining this identity with \eqref{eq:radical-cycle-normalization} gives
the explicit relation to the polynomials:
\begin{equation}
\label{eq:radical-polynomial-bridge}
 C_m^\sigma(q,\chi)=
 \frac{1+q-\chi}
 {2(s_\dagger^2-1)(q-\alpha_\dagger\chi+\alpha_\dagger^2)^2}
 P_m(\alpha_\dagger,q,\chi).
\end{equation}
The common factor is independent of $m$ and strictly positive for
$q_{\mathrm{quad}}(s_\dagger)<q<1$ and $|\chi|\leq2\sqrt q$.
In particular, the two
active polynomial equations are exactly $C_5^\sigma=C_6^\sigma=0$.

The isolated pair $(q_\dagger,\chi_\dagger)$ gives the complex multiplier
$z_\dagger\doteq\chi_\dagger/2+i\sqrt{q_\dagger-\chi_\dagger^2/4}$ for the contracting
and rotating component. Choose the value at index zero of the convex
interpolation function $\varphi$ from \eqref{eq:phi} to be
\[
 v\doteq\frac{\operatorname{Re}[p_s(z_\dagger)(1-\bar z_\dagger)]}
         {1-q_\dagger}.
\]
With $\varphi_t\doteq q_\dagger^t v$, the interpolation inequalities for
$(0,m)$ and $(m,0)$ from \eqref{eq:star-feasible}, expressed in this
normalization, are equivalent to
\begin{equation}
\label{eq:mixed-radical-all-lags}
 0\leq C_m^\sigma(q_\dagger,\chi_\dagger)
 \leq\operatorname{Re}p_s(z_\dagger)|1-z_\dagger^m|^2,
 \qquad m\geq1.
\end{equation}
The lower equalities at $m=5,6$ and the upper equality at $m=1$ follow
algebraically. They correspond to the active comparisons $(5,0)$, $(6,0)$,
and $(0,1)$, respectively. The upper inequalities at $m=5,6$ are strict
because $\operatorname{Re}p_s(z_\dagger)>0$ and $|z_\dagger|<1$.
Outward-rounded interval arithmetic verifies the positive real part and all
remaining inequalities for $m=1,\ldots,30$. Here $m$ is the difference
between the two iteration indices, not a bound on the iteration count. With
$K_\infty\doteq(1-\bar z)/(1-q)$ and
$C_\infty^\sigma\doteq\operatorname{Re}(p_s(z)K_\infty)$, the identity
\[
 C_m^\sigma=C_\infty^\sigma
       +q^m\operatorname{Re}[p_s(z)(1-K_\infty)]
       -\operatorname{Re}(p_s(z)z^m)
\]
gives geometric bounds for every $m\geq31$. The interval calculation shows
that both limiting margins are strictly larger than the corresponding
bounds on the remaining terms. Thus both inequalities hold for every
$m\geq1$. The positive rescaling in \eqref{eq:normalization-change} gives
an angular pair satisfying \eqref{eq:star-feasible} at $q=q_\dagger$.
Hence $q_\dagger\in\mathcal Q(s_\dagger)$, establishing the lower bound.

\emph{Ruling out larger asymptotic factors.}
For the upper bound, we need a positive weighted sum of cycle costs that is
strictly negative at every allowed angle when $q$ is just above $q_\dagger$.
We construct this function from the two active costs $C_5^\sigma$ and
$C_6^\sigma$. Both vanish at the candidate; we choose their weight so that
the derivative of the combination with respect to $\chi$ vanishes there
too. We therefore take the weight
\[
 \lambda_\dagger\doteq
 \frac{\partial_\chi C_6^\sigma(q_\dagger,\chi_\dagger)}
      {\partial_\chi C_6^\sigma(q_\dagger,\chi_\dagger)
        -\partial_\chi C_5^\sigma(q_\dagger,\chi_\dagger)}
 =0.4156030236805012584\ldots\in(0,1)
\]
and define the weighted sum, which in this case is a convex combination
\begin{equation}
\label{eq:radical-dual-combination}
 \Phi_\dagger(q,\chi)\doteq\lambda_\dagger C_5^\sigma(q,\chi)
 +(1-\lambda_\dagger)C_6^\sigma(q,\chi).
\end{equation}
Here $\lambda_\dagger$ and $1-\lambda_\dagger$ weight the two cycles, not two angles; the construction uses one rotating component.
By \eqref{eq:radical-polynomial-bridge}, this is equivalently
\begin{equation}
\label{eq:radical-polynomial-combination}
 \Phi_\dagger(q,\chi)=
 \frac{(1+q-\chi)\bigl[\lambda_\dagger P_5(\alpha_\dagger,q,\chi)
 +(1-\lambda_\dagger)P_6(\alpha_\dagger,q,\chi)\bigr]}
 {2(s_\dagger^2-1)(q-\alpha_\dagger\chi+\alpha_\dagger^2)^2}.
\end{equation}
Thus $\Phi_\dagger$ is an auxiliary rational function built from the two defining
polynomials, with the same sign as their weighted combination on the
allowed interval, which we use for the upper bound.
The elimination below uses the equations $P_5=P_6=0$ which only involve rational coefficients; note, we do not require the weight $\lambda_\dagger$ to be rational.
We keep $\lambda_\dagger$ fixed as $q$ varies. The algebraic definitions give
$\Phi_\dagger=\partial_\chi\Phi_\dagger=0$ at $(q_\dagger,\chi_\dagger)$.
Outward-rounded interval arithmetic establishes $\partial_{\chi\chi}\Phi_\dagger<-100$
on a neighborhood of $\chi_\dagger$ at $q=q_\dagger$, and
$\partial_q\Phi_\dagger(q_\dagger,\chi_\dagger)<-300$. The same method establishes a
strict negative bound away from $\chi_\dagger$ using finitely many intervals
with rational endpoints in $[-161/500,161/500]$, where
$2\sqrt{q_\dagger}<161/500$. An exact rational check confirms that these
intervals and the neighborhood used for the concavity bound together cover
the entire interval; no part of the allowed frequency interval is omitted.
Thus $\chi_\dagger$ is a zero-valued maximum of $\chi\mapsto\Phi_\dagger(q_\dagger,\chi)$.
This is the contact with the horizontal axis; the maximum, rather than a
minimum, reflects the sign of the cycle costs used here.

Local strict concavity and the negative partial derivative with respect to
$q$, together with the strict negative bound away from $\chi_\dagger$ on a
compact interval, make this same weighted sum strictly negative for every
allowed rotation angle when $q>q_\dagger$ is sufficiently close to
$q_\dagger$. Its coefficients are positive weights on finite cycle
inequalities. The positive normalization factor transfers it to the costs
of Lemma~\ref{lem:strict-alternative}.
Proposition~\ref{prop:strict-separator-rate} therefore gives
$\rhohor(\kappa_\dagger)\leq q$ for all such $q$. Letting
$q\downarrow q_\dagger$ and using Theorem~\ref{thm:main} gives
$\qopt(s_\dagger)\leq q_\dagger$. Feasibility supplies the reverse
inequality. This factor is also attained by a trajectory in $\reals^2$ with
$x_{-1}=x_0$, using a separate choice of the initial point and first gradient
rather than the additional coordinate of Lemma~\ref{lem:initialization-splice}.
\end{proof}

Figure~\ref{fig:radical-tangency} plots the weighted sum
\eqref{eq:radical-dual-combination}, keeping both $q=q_\dagger$ and
$\lambda_\dagger$ fixed. Its contact corresponds to the rotation angle
of the single component.
\begin{figure}[tbp]
\centering
\includegraphics[width=\linewidth]{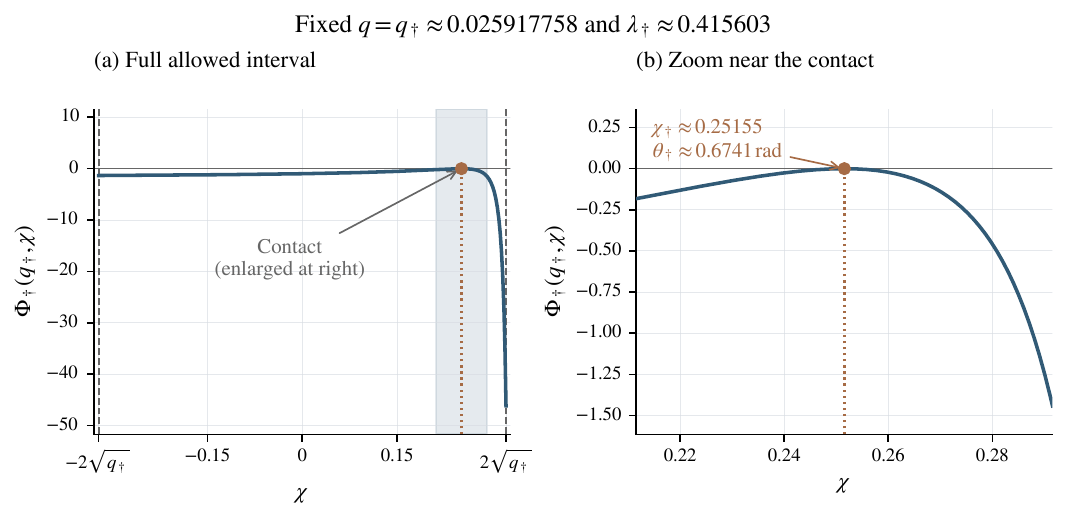}
\caption{Numerical illustration of $\Phi_\dagger(q_\dagger,\chi)$ with $q$ and the
weight $\lambda_\dagger$ fixed. Left: the full allowed interval
$[-2\sqrt{q_\dagger},2\sqrt{q_\dagger}]$.
Right: enlargement of the shaded region, showing the zero-valued maximum
at $\chi_\dagger$, labeled by the rotation angle
$\theta_\dagger=\arccos(\chi_\dagger/(2\sqrt{q_\dagger}))$.}
\label{fig:radical-tangency}
\end{figure}

\subsection{The minimal polynomial and its Galois group}

We now show why the exact rate $q_\dagger = \qopt(s_\dagger)$ derived above cannot be written using rational numbers, arithmetic operations, and finitely many nested $n$th roots. Recall the
parameters of the preceding example:
\[
 s_\dagger=\frac{1421}{1250},\qquad
 \kappa_\dagger=s_\dagger^2=\frac{2019241}{1562500},\qquad
 \alpha_\dagger=1-\frac1{s_\dagger}=\frac{171}{1421}.
\]
Lemma~\ref{lem:mixed-radical-rate} identifies its exact squared-distance factor
as $\qopt(s_\dagger)=q_\dagger$, where
\[
q_\dagger=0.0259177578383525961624\ldots.
\]
We will give a polynomial equation specifying this number and then show why
no radical expression can specify it. The second step is essential: degree
alone does not rule out radicals, as the root $2^{1/20}$ of $q^{20}-2$ illustrates.
For background on the algebra used here, see \citet{stewart2022galois}; precise
versions of the two Galois-theoretic results invoked below are given in
\citet{milne2022fields}.

\paragraph{The polynomial.}
Define
\begin{equation}
\label{eq:radical-minimal-polynomial}
 F(q)\doteq\sum_{j=0}^{20}a_jq^j,
\end{equation}
with all coefficients specified below. They are integers with no common
factor greater than one, and $a_{20}>0$. This normalization avoids fractional
coefficients. Once irreducibility is proved, $F/a_{20}$ is the \emph{minimal
polynomial}: the polynomial of least degree with rational coefficients that
vanishes at $q_\dagger$, normalized to have leading coefficient one.
This last normalization is called \emph{monic}.
\begin{center}
\small
\begin{tabular}{@{}r r@{\qquad}r r@{}}
\toprule
$j$ & $a_j$ & $j$ & $a_j$ \\
\midrule
$20$ & $265994468953997804783976976$ & $9$ & $8189861320519662192293565908$ \\
$19$ & $807144202358229947702699124$ & $8$ & $6349950181251294444105843789$ \\
$18$ & $1612480889879987742825308665$ & $7$ & $4197340353548692377246982409$ \\
$17$ & $2946256217869261342444311626$ & $6$ & $2364361116661553108814131398$ \\
$16$ & $4820345226837491495536334889$ & $5$ & $1329224037648399967202081163$ \\
$15$ & $6720011736645743438628341186$ & $4$ & $525722768550046707732034423$ \\
$14$ & $8515544960311467977164916002$ & $3$ & $-27937412603949322688243409$ \\
$13$ & $10093745217537977492178754770$ & $2$ & $401307772246057263077361$ \\
$12$ & $10921704964320530460147150207$ & $1$ & $-1444247861804047301319$ \\
$11$ & $10853408849668692979789731558$ & $0$ & $731086699811838561$ \\
$10$ & $9752173559218132978619645698$ & & \\
\bottomrule
\end{tabular}
\end{center}

\begin{proposition}[A rate not expressible by radicals]
\label{prop:rate-not-radical}
The factor $\qopt(s_\dagger)$ is the unique real zero of $F$ in the rational
interval
\begin{equation}
\label{eq:radical-root-interval}
 I\doteq(0.025917757838352596,\;0.025917757838352597).
\end{equation}
The polynomial $F$ is irreducible over $\mathbb Q$ and its splitting-field
Galois group is $S_{20}$. In particular, $\qopt(s_\dagger)$ has algebraic
degree $20$ and is not expressible by radicals over $\mathbb Q$.
\end{proposition}
\begin{proof}
\emph{Eliminating the angle.}
The preceding example is determined by two active cycle constraints,
$P_5(\alpha_\dagger,q,\chi)=P_6(\alpha_\dagger,q,\chi)=0$.
Here $z=\sqrt q\,e^{i\theta}$, with $\theta\in[0,\pi]$, specifies the
contracting and rotating component: $q=|z|^2$ is its squared contraction
factor and $\chi=z+\bar z=2\sqrt q\cos\theta$ encodes its rotation angle.
The needed specialization of \eqref{eq:radical-cycle-polynomial} is
\[
\begin{aligned}
 T_0&=2,\qquad T_1=\chi,\qquad T_{j+1}=\chi T_j-qT_{j-1},\\
 H_m&=\sum_{j=0}^{m-1}q^j,\\
 D_\dagger&\doteq2q^2-2\alpha_\dagger q\chi+\alpha_\dagger^2(\chi^2-2q),\\
 E_\dagger&\doteq(q^2+\alpha_\dagger^2q)\chi-4\alpha_\dagger q^2,
\end{aligned}
\]
and, for $m=5,6$,
\[
\begin{aligned}
 P_m(\alpha_\dagger,q,\chi)=\frac1{\chi-1-q}\Bigl[&
 (D_\dagger-E_\dagger)H_m+q^mD_\dagger\\
 &{}-q^2T_m+2\alpha_\dagger qT_{m+1}
       -\alpha_\dagger^2T_{m+2}\Bigr].
\end{aligned}
\]
As before, the numerator is divisible by $\chi-1-q$, so these are polynomials.
We eliminate $\chi$ from the two polynomial equations by taking their
\emph{resultant}. This is a polynomial in $q$ alone that vanishes whenever
the equations have a common $\chi$ solution. Thus it gives a necessary
condition on $q$, algebraically projecting out $\chi$. It may also vanish
at values that do not correspond to an admissible trajectory, so we must
identify the relevant factor. Exact elimination gives, up to a nonzero
rational scalar,
\[
 \operatorname{Res}_\chi(P_5,P_6)
  =q^7(q+1)^2(q-\alpha_\dagger^2)^3
       (q^2-q+1)^2(q^2+q+1)^3F(q),
\]
where $P_m$ here denotes $P_m(\alpha_\dagger,q,\chi)$.
None of the factors preceding $F$ vanishes in $I$, so the common solution
from the preceding lemma satisfies $F(q_\dagger)=0$.
Sturm's rule, which counts real roots in a rational interval using exact
polynomial arithmetic, gives one root in $I$ and one in the smaller
$q$ interval isolating the common solution. The smaller interval lies in $I$.
This identifies the specified root of $F$ with $\qopt(s_\dagger)$.

\emph{Reading the modular factorizations.}
A polynomial is \emph{irreducible over $\mathbb Q$} if it cannot be factored
into two positive-degree polynomials with rational coefficients.
Rather than factor $F$ over $\mathbb Q$ directly, we reduce its integer
coefficients modulo a prime $p$ and factor using arithmetic modulo $p$.
Exact computation gives
\[
 \begin{array}{c|l}
 p&\text{degrees of the irreducible factors of }F\bmod p\\ \hline
 13&1,19\\
 107&1,1,2,3,13\\
 223&20
 \end{array}.
\]
The leading coefficients modulo these primes are $9$, $102$, and $124$,
respectively, so the degree does not drop. Each reduction is
\emph{squarefree}, meaning that no irreducible factor is repeated.

The \emph{Galois group} $\mathfrak{G}$ consists of the permutations of the complex roots
that preserve their algebraic relations over $\mathbb Q$.
Dedekind's theorem turns each row of the table into information about such
permutations: $\mathfrak{G}$ contains a permutation whose cycle lengths are precisely
the listed factor degrees~\citep[Theorem~4.28]{milne2022fields}.\footnote{The
cited theorem is stated for monic integer polynomials. With $a\doteq a_{20}$,
apply it to $a^{19}F(u/a)$, which is monic with integer coefficients and has
the same splitting field as $F$. At the three primes, $a$ is invertible,
so the change of variable and the scalar factor preserve factor degrees
and squarefreeness.}
Here cycles refer to permutations of polynomial roots, not to the
iteration-index comparisons used to construct $F$. A cycle of length $k$
sends each of $k$ roots to the next and the last back to the first, fixing
all other roots. The three rows give the following argument.
\begin{enumerate}
\item \emph{Irreducibility and transitivity.}
Modulo $223$, there is just one irreducible factor, of degree $20$.
Because the integer coefficients of $F$ have no common factor greater than
one and its leading coefficient is nonzero modulo $223$, this proves
irreducibility over $\mathbb Q$.
Consequently $\mathfrak{G}$ is \emph{transitive}: any root can be sent to any other
root by some permutation in $\mathfrak{G}$.
\item \emph{Moving pairs of roots.}
The row for $p=13$ supplies a $19$-cycle fixing the remaining root.
Its powers can send any of the other nineteen roots to any other while
keeping that root fixed. Transitivity gives the same property for any
choice of fixed root. Thus $\mathfrak{G}$ is \emph{$2$-transitive}: it can send any
ordered pair of distinct roots to any other such pair.
\item \emph{Swapping two roots.}
The row for $p=107$ supplies disjoint cycles of lengths $1,1,2,3,13$.
Applying that permutation $39=3\cdot13$ times makes the $3$- and $13$-cycles
return to their starting points, while the $2$-cycle still swaps its two
roots. The result is a \emph{transposition}, a permutation swapping just two
roots. By $2$-transitivity, relabelling this swap using elements of $\mathfrak{G}$
gives a swap of any chosen pair. Every permutation is a product of such
swaps, so $\mathfrak{G}=S_{20}$, the full symmetric group on the twenty roots.
\end{enumerate}

\emph{Excluding radicals.}
The needed consequence of the Galois criterion for solvability by radicals
is this: an irreducible polynomial over $\mathbb Q$ whose Galois group is
$S_n$, with $n\geq5$, has no root expressible by
radicals~\citep[Theorem~3.28 and its following discussion; proof in Theorem~5.34]{milne2022fields}.
Irreducibility matters here: a radical expression for one root would give
such expressions for all the other roots as well.
Since $F$ is irreducible and its group is $S_{20}$,
this excludes a radical expression for $q_\dagger$.
Evaluating an everywhere-defined radical formula with rational coefficients
at the rational input $s_\dagger$ would contradict this conclusion, proving
Remark~\ref{rem:no-radical-formula}. This does not preclude specifying the
rate as the unique root of a polynomial in an explicitly given interval,
as we have done here.
\end{proof}

\subsection{Robustness of the radical obstruction}
\label{app:radical-extensions}

Can additional constants help express the single value $q_\dagger$, or the
rate as a function of $s$ near $s_\dagger$? A radical expression permits field
arithmetic and finitely many nested $n$th roots ($n\geq2$), including complex
intermediate values.

\begin{proposition}[Extensions at the certified parameter]
\label{prop:radical-base-extensions}
The factor $q_\dagger$ is not expressible by radicals over $\mathbb Q(\tau)$
for any number $\tau$ transcendental over $\mathbb Q$, nor over any finite
extension $E/\mathbb Q$ with $[E:\mathbb Q]<20$. Adjoining the answer itself
attains the bound: $\mathbb Q(q_\dagger)$ has degree $20$.
\end{proposition}
\begin{proof}
Let $L$ be the splitting field of $F$, so $\operatorname{Gal}(L/\mathbb Q)=S_{20}$.
A transcendental $\tau$ remains transcendental over $L$. Clearing denominators
and comparing coefficients of powers of $\tau$ shows that a basis of $L$ over $\mathbb Q$
stays linearly independent over $\mathbb Q(\tau)$; likewise for
$\mathbb Q(q_\dagger)$. Thus adjoining $\tau$ preserves the irreducibility
and group $S_{20}$ of $F$, excluding radical expressions for its roots.

For a finite extension $E$, the splitting group over $E$ identifies with
$\mathfrak{G}_E=\operatorname{Gal}(L/(L\cap E))$, a subgroup of $S_{20}$ of index
$d=[L\cap E:\mathbb Q]<20$. Acting on the cosets maps $S_{20}$ to $S_d$.
Since $20!/2>d!$, the simple alternating group $A_{20}$ of even permutations
cannot inject into $S_d$. Its action is therefore trivial, so $\mathfrak{G}_E$
contains $A_{20}$. This makes $\mathfrak{G}_E$ transitive and nonsolvable: $F$ stays
irreducible over $E$, and the radical criterion again applies.
\end{proof}

We now seek a formula near $s_\dagger$, with constants fixed independently
of $s$. These may be \emph{arbitrary complex numbers}, even $q_\dagger$ itself.

Keeping $\alpha$ indeterminate in \eqref{eq:radical-cycle-polynomial} gives
a parameter-dependent version of $F$, the polynomial $H\in\mathbb Z[\alpha,q]$
defined by
\begin{equation}
\label{eq:generic-radical-resultant}
\begin{aligned}
 \operatorname{Res}_\chi(P_5,P_6)
 ={}&-2\alpha^{12}q^7(q+1)^2(q-\alpha^2)^3\,\times\\[-2pt]
 & (q^2-q+1)^2(q^2+q+1)^3H(\alpha,q).
\end{aligned}
\end{equation}
It has degree $20$ in $q$, degree $8$ in $\alpha$, leading $q$ coefficient
$16$, and $H(\alpha_\dagger,q)=16F(q)/a_{20}$.

\begin{proposition}[No local radical formula with arbitrary fixed constants]
\label{prop:geometric-radical-obstruction}
There is an open interval $J\subset(1,\infty)$ containing $s_\dagger$ on
which $\qopt(s)$ is a real-analytic root of $H(1-1/s,q)$.
The splitting-field Galois group of $H$ over $\mathbb C(\alpha)$ is $S_{20}$.
Consequently no formula built from rational functions of $s$ with arbitrary
fixed complex coefficients and finitely many radicals agrees with $\qopt(s)$
on any nonempty open subinterval of $J$.
\end{proposition}
\begin{proof}
\emph{Persistence of the optimal branch.}
At the isolated common zero, outward-rounded evaluation gives
\[
 4.5\cdot10^{-8}
 <\det\frac{\partial(P_5,P_6)}{\partial(q,\chi)}
       (\alpha_\dagger,q_\dagger,\chi_\dagger)
 <4.6\cdot10^{-8}.
\]
The analytic implicit-function theorem continues $(q,\chi)$ locally in $s$,
preserving $q>\alpha_s^2$ and $\chi^2<4q$. The three active comparisons
remain equalities. All other finite inequalities in
Lemma~\ref{lem:mixed-radical-rate} have strict margins; its geometric tail
bounds remain uniform on a small compact neighborhood with $q<1$.
Thus interpolation feasibility persists. The dual weight remains in $(0,1)$;
strict concavity near contact, the negative test-$q$ derivative, and
uniform negativity away from contact also persist.
For each continued pair, hold its dual weight fixed as the test $q$ varies.
The same lower and strict-separator upper arguments identify the continued
factor with $\qopt(s)$. Shrink $J$ to exclude the other resultant factors.

\emph{Galois group over $\mathbb Q(\alpha)$.}
At $\alpha_\dagger$, the roots are distinct and their group is $S_{20}$.
Each root lifts uniquely to a formal power-series root in
$\alpha-\alpha_\dagger$, with coefficients in the splitting field $L$ of $F$.
Acting coefficientwise, $\operatorname{Gal}(L/\mathbb Q)$ permutes these
lifts faithfully because their constant terms are distinct. Thus the group
over $\mathbb Q(\alpha)$ contains, and hence equals, $S_{20}$.

\emph{Allowing arbitrary complex constants.}
The discriminant
$\mathcal{D}(\alpha)\doteq\operatorname{Disc}_qH(\alpha,q)\in\mathbb Z[\alpha]$
detects collisions of roots in $q$. Exact computation gives
\begin{equation}
\label{eq:generic-radical-branch-certificate}
 \mathcal{D}(8)\equiv0,\qquad \mathcal{D}'(8)\equiv25\pmod{29}.
\end{equation}
These congruences imply that $\mathcal{D}$ has a simple complex zero: if every
characteristic-zero irreducible factor were repeated, reduction modulo $29$
could not produce a simple root. Since the leading $q$ coefficient is the
nonzero constant $16$, a simple discriminant zero is a simple collision of
two roots; continuation around it exchanges those roots and fixes the others.
This transposition lies in the subgroup fixing the algebraic constant field,
called the geometric group. It is normal in $S_{20}$, so contains all
conjugate transpositions and equals $S_{20}$.
Extending the algebraically closed constants from
$\overline{\mathbb Q}$ to $\mathbb C$ preserves this group.

Hence $H$ is irreducible over $\mathbb C(\alpha)$ with nonsolvable group.
A radical expression agreeing with a root on an open interval would give a
radical algebraic branch over that field, contradicting the radical
criterion. The rational change $\alpha=1-1/s$ proves the claim.
\end{proof}

\begin{remark}[Adjoining the contact coordinate]
The contact coordinate and the rate generate the same degree-$20$ field:
$\mathbb Q(\chi_\dagger)=\mathbb Q(q_\dagger)$.
Exact elimination of the specialized $P_5,P_6$ gives rational-coefficient
relations $U(q)\chi+V(q)=0$ and $\widetilde U(\chi)q+\widetilde V(\chi)=0$.
Exact gcd checks exclude zeros of both leading coefficients at the isolated
pair, so each coordinate is recovered rationally from the other.
These obstructions do not exclude isolated-root descriptions or other
nonradical notions of closed form.
\end{remark}

\section{Rate values and their numerical evaluation}
\label{app:rate-table}

While expressing the optimal contraction factor as a formula with radicals might not be possible in general as we have seen in Section~\ref{app:rate-radicals}, we can still use our characterization to compute contraction values for varying values of $\kappa$ in Table~\ref{tab:nag-rate-values}. To this end, we evaluate
$\qopt(\sqrt\kappa)$ for selected condition numbers and compare it with
Gannot's upper bound denoted by $q_{\rm G}(\sqrt\kappa)$. The values are obtained by optimizing over mode weights and angles in the equivalent angular formulation. As can be seen, Gannot's upper bound is quite tight with largest gaps for moderate $\kappa$. 

\begin{table}[h]
\centering
\small
\begin{tabular}{@{}r r r r l@{}}
\toprule
$\kappa$ & $q_\star$ & $q_{\rm G}$ & $q_{\rm G}-q_\star$ & status \\
\midrule
1 & $0$ & -- & -- & exact \\
1.001 & $4.99001617487858\times 10^{-7}$ & $4.99001622504\times 10^{-7}$ & $5.01595\times 10^{-15}$ & numerical \\
1.01 & $4.90160022279601\times 10^{-5}$ & $4.90160037565\times 10^{-5}$ & $1.52852\times 10^{-12}$ & numerical \\
1.1 & $0.00414080335569045$ & $0.00414080555192$ & $2.19623\times 10^{-9}$ & numerical \\
2 & $0.133970856851373$ & $0.133998720281$ & $2.78634\times 10^{-5}$ & numerical \\
5 & $0.400092644577688$ & $0.40047760383$ & $3.84959\times 10^{-4}$ & numerical \\
10 & $0.564423277972893$ & $0.565235163135$ & $8.11885\times 10^{-4}$ & exact CA \\
20 & $0.688283179210141$ & $0.689465036059$ & $0.00118186$ & numerical \\
50 & $0.801693369348438$ & $0.803120025898$ & $0.00142666$ & numerical \\
100 & $0.859899778052548$ & $0.861059858423$ & $0.00116008$ & exact CA \\
200 & $0.901252726206213$ & $0.902086880416$ & $8.34154\times 10^{-4}$ & numerical \\
500 & $0.937861932828071$ & $0.938397069587$ & $5.35137\times 10^{-4}$ & numerical \\
1000 & $0.95626894031021$ & $0.956612260497$ & $3.4332\times 10^{-4}$ & numerical \\
\bottomrule
\end{tabular}
 \ \caption{Selected squared-distance optimal factor $\qopt$ and the Gannot upper bound
$q_{\rm G}$. For $\kappa=1$ the values are analytically exact (and trivial).
\emph{Exact CA} (exact computer-assisted) refers to separately established
algebraic values with matching lower and upper rate proofs via Appendices~\ref{app:kappa-ten} and~\ref{app:kappa-hundred}; the values in the table are rounded. Other rate entries and all gaps are numerical approximations via the procedure described in the section.}
\label{tab:nag-rate-values}
\end{table}

\paragraph{Obtaining numerical rates.}
Basically the numerical rates are obtained by means of bisection. More precisely, for a given $\kappa>1$, we range over test squared factors $q$. For each such $q$ we want to assess numerically whether it is above or below $\qopt(\sqrt\kappa)$ and to this end we choose a finite grid of
angles in $[0,\pi]$. The unknowns are nonnegative mode weights summing to one
and a nonnegative common potential value. We impose the two interpolation
inequalities \eqref{eq:star-feasible} for $t=1,\ldots,M$, using the equivalent
transformed-gradient normalization of Appendix~\ref{app:uniform-upper-bound}. With $q$ and the angles fixed, these
constraints are linear in the unknowns. The linear program maximizes a
common slack in the inequalities: a nonnegative optimal slack indicates
feasibility of this finite system, subject to numerical accuracy. If we
considered all angles and all interpolation inequalities and solved the
system exactly, then, within our search range, infeasibility would imply
$q > \qopt(\sqrt\kappa)$. We are, however, dealing with three approximations:
floating-point arithmetic, a finite angular grid, and a finite subset of the
interpolation inequalities. We assess numerical accuracy through angular
refinement and checks at additional iteration indices, nonetheless we stress that the obtained rates are only approximations. 

The bisection in $q$ starts with the initial interval
$[(1-1/\sqrt\kappa)^2, q_{\rm G}]$, whose endpoints are the quadratic factor
and Gannot's upper bound $q_{\rm G}$ defined below. If the computed optimal
slack is negative, we take $q$ as the new upper endpoint of the numerical
search; otherwise, we take it as the new lower endpoint. These updates do
not constitute certified bounds on $\qopt(\sqrt\kappa)$. We first solve on a grid
covering the angular interval, with additional resolution near zero. We then
add grid points around the angles receiving positive weight and repeat the
bisection on this finer grid. This refines the locations of the modes further.
For each condition number, the index limit for the LP is fixed at
\[
 M=\max\{80,\lceil8\sqrt\kappa\rceil\}.
\]
After the angular refinement, we check the recovered weights and potential
against the two interpolation inequalities for $t=1,\ldots,K$, where
\[
 K=\max\{2000,\lceil40\sqrt\kappa\rceil\}.
\]
Thus the returned measure is tested against a large number of constraints that were not
used to construct it. We stress that these are only finite, floating-point checks, i.e., an approximate test which does not prove feasibility for every integer $t\geq1$. The exact characterization of Theorem~\ref{thm:main} would require using all iteration indices.

\paragraph{Exact values.}
While the aforementioned numerical computation is only approximate, we can also do exact calculations for specific values. To illustrate this, we did this for $\kappa=10$ and $\kappa=100$ in Appendices~\ref{app:kappa-ten}
and~\ref{app:kappa-hundred} and report these values in Table~\ref{tab:nag-rate-values} instead of the numerical approximations via the procedure described above. These algebraic solutions and matching lower and upper rate bounds are obtained from algebraic computations, exact identities and interval arithmetic. The infinitely many interpolation inequalities are handled by finite checks together with analytic bounds for the remaining iteration indices. At $\kappa=1$, the rate is zero because the first gradient step reaches the minimizer.

\paragraph{Computing the comparison bound.}
For $\kappa>1$, Gannot's upper bound is the smaller positive root $r_-$ of
\begin{align*}
 h_\kappa(r)={}&8(1-\sqrt\kappa)^4
 -(1-\sqrt\kappa)^2(15\kappa-10\sqrt\kappa-1)r\\
 &+2\kappa(3\kappa-4\sqrt\kappa+1)r^2+\kappa^2r^3,
\end{align*}
as stated in \citet[Proposition~2.13]{gannot2022frequency}.
That proposition bounds the norm factor by $\sqrt{r_-}$, so the corresponding
squared-distance bound is $q_{\rm G}=r_-$. We evaluate this cubic for the table and also use it to provide the upper endpoint of the initial LP search interval.
The dash at $\kappa=1$ in Table~\ref{tab:nag-rate-values} records the proposition's assumption $\kappa>1$; the plotted comparison curve in Figure~\ref{fig:nag-rate-curve} uses the continuous limit zero.

\end{document}